\documentclass[11pt,a4paper]{article}
\usepackage{graphicx}
\usepackage{caption}
\usepackage{subcaption}
\usepackage[margin=1in]{geometry}
\usepackage{float} 
\usepackage{ulem}
\usepackage{amsmath}
\usepackage{amsthm}
\usepackage{amssymb}
\usepackage{answers}
\usepackage{physics}
\usepackage{setspace}
\usepackage{graphicx}
\usepackage{enumitem}
\usepackage{multicol}
\usepackage{blkarray}
\usepackage{tikz}
\usetikzlibrary{automata, positioning, arrows}
\usepackage[colorlinks = true,
            linkcolor = blue,
            urlcolor  = blue,
            citecolor = blue,
            anchorcolor = blue,
            pagebackref=false]{hyperref}

\usepackage{mathrsfs}
\usepackage{lmodern}
\usepackage{amsmath,amsthm,amssymb}

\usepackage[T1]{fontenc}
\usepackage[noend]{algpseudocode}
\usepackage{comment}
\usepackage{textcomp}
\usepackage[ruled,vlined,linesnumbered]{algorithm2e}
\let\oldnl\nl
\newcommand{\nonl}{\renewcommand{\nl}{\let\nl\oldnl}}
\makeatother

\newcommand{\R}{\mathbb{R}}

\newcommand{\f}{\mathbb}

\newcommand{\ol}{\overline}
\newcommand{\cu}{\subseteq}

\DeclareMathOperator{\vect}{vec} 
\DeclareMathOperator{\mat}{mat} 

\DeclareMathOperator{\conv}{conv}
\DeclareMathOperator{\cone}{cone}

\DeclareMathOperator{\logdet}{logdet}

\newtheorem{theorem}{Theorem}[section]
\newtheorem{corollary}{Corollary}[theorem]
\newtheorem{lemma}[theorem]{Lemma}
\newtheorem{property}[theorem]{Property}

\newtheorem{definition}[theorem]{Definition}

\newtheorem{remark}[theorem]{Remark}

\newtheorem{assumptions}[theorem]{Assumption} 
\newtheorem{openquestion}[theorem]{Question}

\usepackage[utf8]{inputenc}
\usepackage[english]{babel}
\usepackage{graphicx}
\graphicspath{{images/}{../images/}}
\usepackage{blindtext}

\usepackage{subfiles}
\title{Identifiability of Nonnegative Tucker Decompositions \\
-- \\
Part I: Theory} 

\author{Subhayan Saha, Giovanni Barbarino, Nicolas Gillis\thanks{Email: \{subhayan.saha, giovanni.barbarino, nicolas.gillis\}@umons.ac.be. The authors acknowledge the support
 by the European Union (ERC consolidator, eLinoR, no 101085607).  GB is member of the Research Group GNCS (Gruppo Nazionale per il Calcolo Scientifico) of INdAM (Istituto Nazionale di Alta Matematica). 
 } \\ 
University of Mons, Rue de Houdain 9, 7000 Mons, Belgium   
	}
\date{}

\definecolor{brightpink}{rgb}{1.0, 0.0, 0.5}
\newcommand{\ngc}[1]{{\color{brightpink} (\textbf{NG:} #1)}}

\newcommand{\gbc}[1]{{\color{blue} (\textbf{GB:} #1)}}
\definecolor{blues}{rgb}{0.0, 0.7, 0.7}

\begin{document}

\maketitle

\begin{abstract}
Tensor decompositions have become a central tool in data science, with applications in areas such as data analysis, signal processing, and machine learning. A key property of many tensor decompositions, such as the canonical polyadic decomposition, is identifiability: the factors are unique, up to trivial scaling and permutation ambiguities. This allows one to recover the groundtruth sources that generated the data. The Tucker decomposition (TD) is a central and widely used tensor decomposition model. However, it is in general not identifiable. In this paper, we study the identifiability of the nonnegative TD (nTD). By adapting and extending identifiability results of nonnegative matrix factorization (NMF), we provide uniqueness results for nTD. Our results require the nonnegative matrix factors to have some degree of sparsity (namely, satisfy the separability condition, or the sufficiently scattered condition), while the core tensor only needs to have some slices (or linear combinations of them) or unfoldings with full column rank (but does not need to be nonnegative). Under such conditions, we derive several procedures, using either unfoldings or slices of the input tensor, to obtain identifiable nTDs by minimizing the volume of unfoldings or slices of the core tensor. 
\end{abstract}













  













  

\section{Introduction}\label{sec:introduction}

Tensors are important algebraic objects that appear in different branches of science such as mathematics, physics, computer science and chemistry. For our purpose, tensors can be viewed as multi-dimensional arrays with entries from the underlying field of real numbers~$\R$. Following this, order-$1$ tensors are vectors and order-$2$ tensors are matrices. For a tensor $\mathcal T \in \R^{n_1 \times n_2 \times \dots \times n_d}$,  
its order is the number of dimensions, $d$, also referred to as \textit{modes}, while the dimension of the $k$th mode is $n_k$ for $k \in [d] := \{1,2,\dots,d\}$.  
\par
Tensor decompositions have generated significant interest in recent years due to their applications in different fields such as signal processing, computer vision, chemometrics, neuroscience and others; see, e.g.,  \cite{Kolda2009TensorDA, sidiropoulos2017tensor} for comprehensive surveys, and \cite{ballard2025tensor} for a very recent book.  
One of the central motivations for studying such decompositions and applying them in real-world problems comes from the fact that some models of tensor decompositions are identifiable under very mild assumptions. 
One such model is the canonical polyadic decomposition (CPD) that decomposes a given tensor as the  \textit{minimal} sum of \textit{rank-one} tensors. Such a decomposition is essentially unique, meaning that is is unique up to permutations and scalings of the rank-one factors, 
if the input tensor satisfies some \textit{mild} conditions \cite{KRUSKAL77,domanov2016,koiran2024overcomplete,kothari2024overcompletete}. As a consequence, there is a lot of research on developing algorithms for the CPD model~\cite{bader2015matlab, vervliet2016tensorlab, huang2016flexible, kossaifi2019tensorly, fu2020computing,koiransaha2025,koiran2024undercomplete}.

In this paper, we focus on another widely-used model of tensor decomposition called the Tucker decomposition (TD) which is a generalization of the CPD. 
It is a form of \textit{higher-order PCA} which decomposes a tensor into a core tensor and multiple matrices such that the given tensor can be obtained by transforming the core tensor by a matrix along each mode. 
\par
Let us focus on order-$3$ tensors for simplicity. 
Given an order-$3$ tensor $\mathcal{G} \in \R^{r_1 \times r_2 \times r_3}$ and matrices $U_i \in \R^{n_i \times r_i}$ for all $i \in [3]$, the following multilinear transformation
\begin{equation} \label{eq:Tuckerdef}
\mathcal T = (U_1,U_2,U_3).\mathcal{G} \quad \in \quad \R^{n_1 \times n_2 \times n_3}, 
\end{equation}
is defined as 
\begin{equation}\label{eq:TuckerDecdefn}
    \mathcal \mathcal T_{j_1,j_2,j_3} = \sum_{i_t \in [r_t] \text{ for all } t \in [3]} (U_1)_{j_1,i_1}(U_2)_{j_2,i_2}(U_3)_{j_3,i_3} \mathcal{G}_{i_1,i_2,i_3}.
\end{equation}
The decomposition $\mathcal T = (U_1,U_2,U_3) . \mathcal{G}$ is 
called an $(r_1,r_2,r_3)$-TD of $\mathcal T$.
Note that for order-$2$ tensors, that is, for matrices, an $(r_1,r_2)$-TD corresponds to a matrix tri-factorization of the form 
\begin{equation}\label{eq:triNMF}
   X = U_1 G U_2^\top,  
\end{equation}
where $U_1 \in \R^{n_1 \times r_1}$, $U_2 \in \R^{n_2 \times r_2}$, and $G \in \R^{r_1 \times r_2}$ is the core matrix. 

CPD is a special case of TD when $r_i = r$ for all $i$ while $\mathcal{G}$ is a diagonal tensor, that is, $\mathcal{G}_{i_1,i_2,\dots,i_d} \neq 0$ only when $i_1=i_2=\dots=i_d$. 
In contrast to CPD, TD of an arbitrary tensor is not essentially unique as it suffers from rotation ambiguities. 
In fact, for an $(r_1,r_2,r_3)$-TD of $\mathcal T = (U_1,U_2,U_3).\mathcal{G}$,   and for any invertible matrices 
$R_i \in \R^{r_i \times r_i}$  
for $i \in [3]$, $(U_1R_1,U_2R_2,U_3R_3) . \mathcal{G}'$ is another $(r_1,r_2,r_3)$-TD for $\mathcal T$ where $\mathcal{G}' = (R_1^{-1},R_2^{-1},R_3^{-1}).\mathcal{G}$.
Hence, one needs to impose additional constraints or regularization on the TD in order to get uniqueness.

\paragraph{Outline and contribution of the paper} 

In this paper, we focus on nonnegativity constraints on the $U_i$'s. 
We refer to such decompositions as nonnegative Tucker decompositions (nTDs). 
We will use NTD, as done in the literature, only when the core tensor, $\mathcal G$, is also imposed to be nonnegative. However, in most of our results, we will not need this assumption. 
If a given tensor has an nTD of a certain dimension, we are interested in the following two questions: 
\begin{itemize}
    \item Under what conditions is such a decomposition \textit{unique}? 
    
    \item How can we (approximately) \textit{compute} such a decomposition? 
    
\end{itemize} 
Part I of the paper, ``Identifiability of nTD: Theory'', focuses on the first question, and is divided in 7 sections: 
\begin{itemize}

\item Section~\ref{sec:notation} presents the notation, definitions and the preliminary results necessary to present our contributions; in particular, it presents identifiability results of NMF that are the starting point of our work.  

\item Section~\ref{sec:previousworks} summarizes previous works on the identifiability of nTDs. At the end of this section, we summarize our main identifiability results; see Table~\ref{tab:summary3}. 
(Skipping directly to Table~\ref{tab:summary3}, page~\pageref{tab:summary3}, allows the  interested reader familiar with the tensor notation to have a quick overview of our results.)

    \item Section~\ref{sec:NMF} deals with order-2 nTDs, that is, nonnegative matrix tri-factorizations; see~\eqref{eq:triNMF}. We generalize identifiability results from the NMF literature to order-2 nTDs. Identifiability is obtained via sparsity conditions on the factors. Our first result, Theorem~\ref{th:serptrisymNMF}, relies on the separability condition (Definition~\ref{def:separablematricesfirst}) which requires the factors to contain a diagonal matrix as a submatrix. 
    Our second result, Theorem~\ref{thm:nmfmain}, relies on the sufficiently scattered condition (the SSC, Definition~\ref{def:SSCalternate}), a relaxation of separability, which requires the factors to contain rows sufficiently spread in the nonnegative orthant, and allows one to retrieve the groundtruth factors using volume minimization of the core matrix.

    \item Section~\ref{sec:idenorder3tensors} deals with order-$3$ nTDs. 
    We propose 5 procedures to compute order-$3$ nTDs, with identifiability guarantees 
    (Theorems~\ref{thm:order2unfold}, \ref{thm:idea1order3}, \ref{thm:idea2order3}, \ref{thm:idea3order3}, \ref{thm:idea4order3}). 
    They rely on the SSC for the factors, but use different conditions on the core tensor, $\mathcal G$. 
    The proofs rely on the identifiability of NMF (Section~\ref{sec:NMFprelimin}), and of order-2 nTDs (Section~\ref{sec:NMF}). 
    The first procedure relies on unfoldings of the input tensor, while the other four on its slices. 

    \item 
    Section~\ref{app:orderd} generalizes the results for order-$3$ from Section~\ref{sec:idenorder3tensors} to any order. 

\item In Sections~\ref{sec:idenorder3tensors} and~\ref{app:orderd}, we sometimes need that the Kronecker product of two or more factors  satisfies the SSC. 
In Section~\ref{subsec:KronSSCconj}, we explore sufficient conditions for this property to hold. 
    
\end{itemize} 

In this paper, we only consider identifiability result for exact decompositions. Robustness analysis in the presence of noise is a study of further research.  
Part II of this paper~\cite{sahaPartII2025} will focus on the second question: it will provide efficient algorithms to compute identifiable nTDs, and illustrate their effectiveness on synthetic and real data sets.

\section{Notation, definitions and preliminaries} \label{sec:notation}

Let us introduce the notation, definitions, and results from the literature used in this paper.

\subsection{Tensors}

We first fix the notion of uniqueness that we consider in this paper. 
\begin{definition}[Essential Uniqueness of TD]
Let $\mathcal T\in \R^{n_1 \times n_2 \times \dots \times n_d}$ be an order-$d$ tensor which has an $(r_1,\dots,r_d)$-TD of the form $\mathcal T = (U^{\#}_1,\dots,U^{\#}_d).\mathcal{G}^{\#}$. 
Let $f$ be an objective function and $\mathcal{S}$ be a feasible set for 
$(U_1,\dots,U_d,\mathcal{G})$, and define the optimization problem \begin{equation}\label{eq:metaopt}
    \min_{U_1,\dots,U_d,\mathcal{G}} f(U_1,\dots,U_d,\mathcal{G}) 
     \quad \text{such that } \quad 
     (U_1,\dots,U_d,\mathcal{G}) \in \mathcal{S}. 
\end{equation} 
Note that $f$ and $\mathcal{S}$ will depend on $\mathcal T$. 
Then the decomposition $\mathcal T = (U^{\#}_1,\dots,U^{\#}_d).\mathcal{G}^{\#}$ with $(U^{\#}_1,\dots,U^{\#}_d,\mathcal{G}^{\#}) \in \mathcal S$ is an essentially unique TD with respect to $(f,\mathcal{S})$ if for any optimal solution of~\eqref{eq:metaopt}, $(U^*_1,\dots,U^*_d,\mathcal{G}^*)$, there exist permutation matrices $\Pi_i$ and diagonal matrices $D_i$ such that $U^{\#}_i = U_i^*\Pi_i D_i$ for all $i \in [d]$ and $\mathcal{G}^{\#} =(D_1^{-1} \Pi^\top_1,\dots,D_d^{-1}  \Pi_d^\top).\mathcal{G}^*$. For simplicity, we will also say that $\mathcal T = (U^{\#}_1,\dots,U^{\#}_d).\mathcal{G}^{\#}$ is essentially unique w.r.t.\ \eqref{eq:metaopt}. 
\end{definition}

Since we consider nTDs in this paper, we will always have 
\[
\mathcal{S} \; \subseteq \; \{ (U_1,\dots,U_d,\mathcal{G}) \ | \ U_i \geq 0 \text{ for } i \in [d] \}. 
\]

Let us now define slices and unfoldings of a tensor.  
\begin{definition}[Slices of an order-$3$ tensor]\label{def:slices}
Let $\mathcal T\in \mathbb{R}^{n_1 \times n_2 \times n_3}$ be an order-$3$ tensor. Then the $n_3$ slices along the third mode, denoted by $\mathcal T^{(3)}_k \in \mathbb{R}^{n_1 \times n_2}$ for $k \in [n_3]$, are matrices given by
\begin{equation}
    (\mathcal T^{(3)}_k)_{i,j} = \mathcal T_{i,j,k} \quad \text{ for all } i,j. 
\end{equation}
One can similarly define the slices along the first and the second modes and denote 
them by $\mathcal T^{(1)}_i \in \mathbb{R}^{n_2 \times n_3}$ 
and $\mathcal T^{(2)}_j \in \mathbb{R}^{n_1 \times n_3}$, respectively. 
\end{definition}    
\begin{definition}[Unfoldings of an order-$3$ tensor]
    Let $\mathcal T\in \mathbb{R}^{n_1 \times n_2 \times n_3}$ be an order-$3$ tensor. Then the unfolding of $\mathcal T$ along the third mode is given by the  matrix $\mathcal T_{(3)} \in \mathbb{R}^{n_1n_2 \times n_3}$ where 
    \[
    (\mathcal T_{(3)})_{(i,j),k} = \mathcal T_{i,j,k}, \quad 
     \text{ for all } i,j,k, 
    \] 
    where  we use the notation $(i,j)= i + (j-1) n_1$. 
One can similarly define the unfoldings along the first and the second modes  and denote them by $\mathcal T_{(1)} \in \mathbb{R}^{n_2 n_3 \times n_1}$ and $\mathcal T_{(2)} \in \mathbb{R}^{n_1 n_3 \times n_2}$, respectively. 
\end{definition} 

These notions can be extended to higher-order tensor; see Section~\ref{app:orderd}.  


Let us define the Kronecker products between two matrices. 
\begin{definition}[Kronecker products]\label{def:KronPdt}
Let $A \in \R^{m \times M}$ and $B \in \R^{n \times N}$ be matrices with columns $a_1,\dots,a_M$ and $b_1,\dots,b_N$, respectively. Then the Kronecker product between $A$ and $B$, denoted by $A \otimes B$, is the $mn \times MN$ matrix with column at position $p = j + (i-1)N$ given by   
\[
(A \otimes B)_{:,p} = 
a_i \otimes b_j = \vect(a_i b_j^\top) =
[a_i b_j(1); a_i b_j(2); \dots; a_i b_j(n)]  \in \mathbb{R}^{mn}, 
\] 
where $\vect$ vectorizes a matrix column wise. 
\end{definition}


\subsection{NMF: concepts and identifiability} \label{sec:NMFprelimin}

To prove our identifiability results for nTD, we will heavily rely on tools and results from the NMF literature. Recall that, in the exact case, NMF is the following problem: 
given a nonnegative matrix $X \in \mathbb{R}^{m \times n}_+$ and a factorization rank $r$, find $W \in \mathbb{R}^{m \times r}_+$ and $H \in \mathbb{R}^{n \times r}_+$ such that $X = WH^\top$. 
Note that NMF decompositions are typically not unique unless additional constraints or regularization are used; we refer the interested reader to the survey~\cite{xiao2019uniq} and~the book~\cite[Chapter~4]{Gil20} for more details.

\paragraph{Cones} For any matrix $A \in  \R^{m \times n}$, the convex cone generated by the columns of $A$, denoted $\cone(A)$, is given by the set of \textit{conic} combinations, that is, linear combinations with nonnegative weights, of the columns of $A$: 
$$
\cone(A) = \{x\ |\ x = Ay \text{ for } y \in \R^n, y \geq 0\}.
$$ 
The dual cone of $A$, denoted by $\cone^*(A)$, is given by
\begin{align}
    \cone^*(A) 
    & = 
    \left\{ y \ |\ x^{\top}y \geq 0 \text{ for all } x \in \cone(A)   \right\} \nonumber \\ 
    & = 
    \left\{ y \ |\ z^{\top} A^\top y \geq 0 \text{ for all } z \geq 0   \right\} 
    = 
    \left\{y \ |\ A^{\top}y \geq 0 \right\}.\label{eq:dualcone} 
\end{align}

\paragraph{Separability} The notion of separability dates back to the hyperspectral community under the pure-pixel assumption~\cite{boardman1995mapping}. 
The terminology was introduced by Donoho and Stodden~\cite{donoho2004does}, and it was later used by Arora et al.~\cite{AGKM11} to obtain unique and polynomial-time solvable NMF problems; see~\cite[Chapter~7]{Gil20} for a survey on separable NMF. 
\begin{definition}\label{def:separablematricesfirst}
A matrix $H \in \R^{n \times r}_+$ is called separable if there exists an index set $\mathcal{K} \subseteq [n]$  where $|\mathcal{K}| = r$ such that $H(\mathcal{K},:)$ is a diagonal matrix with positive diagonal elements. 
\end{definition}
Equivalently, a matrix $H$ is \textit{separable} if the convex cone generated by its rows spans the entire nonnegative orthant, that is, $\cone(H^\top) = \mathbb{R}^r_+$; see Figure~\ref{fig:sepSSC} for an illustration  when $r=3$.  

 We say that $X$ admits a separable NMF $(W,H)$ of size $r$ if there exists a decomposition of $X$ of the form $X = WH^\top$ of size $r$ such that $H \in \mathbb{R}^{n \times r}$ is a separable matrix. 
 By the scaling degree of freedom, this implies that there exists a decomposition $X = W'H'^\top$ where $W' = X(:,\mathcal{K})$, that is, the columns of $W$ are a subset of the columns of $X$. 

\begin{theorem}\label{thm:idenseparableNMF}\cite{AGKM11}
    Let $X = WH^\top \in \R^{m \times n}$ be a separable NMF of size\footnote{This can be relaxed to $\cone(X)$ having  $r$ rays.  
    We use the rank here for simplicity. Note also that $W$ does not need to be nonnegative; see Remark~\ref{rem:nonneg}. \label{footnote:rays}} 
    $r = \rank(X)$. 
    Then for any other separable NMF $(W_*,H_*)$ of size $r$ of $X$,  there exists a permutation matrix $\Pi$ and diagonal matrix $D$ such that $W_* = W D \Pi$ and $H_* = H D^{-1} \Pi$.
\end{theorem} 
The proof of Theorem~\ref{thm:idenseparableNMF} directly follows from the fact that the extreme rays of $\cone(X)$ coincide with that of $\cone(W)$ under the separability assumption. 
Given a matrix $X$ and rank~$r$, there exist fast and robust polynomial-time algorithms to compute a separable NMF: they identify the extreme rays of $\cone(X)$~\cite[Chapter~7]{Gil20}; 
we will generalize such algorithms to order-2 nTDs in Part II of this paper~\cite{sahaPartII2025}. 


\paragraph{The sufficiently scattered condition (SSC)} The separability assumption is relatively strong. To relax it, a crucial notion is the sufficiently scattered condition (SSC), which was introduced in~\cite{huang2013non}. 

\begin{definition}\label{def:SSCalternate}[Sufficiently scattered condition (SSC)\footnote{Slight variants of the SSC exist in the literature. Refer to Section 4.2.3.1 in \cite{Gil20} for a more detailed description and the relation between these variants.}]
A matrix $H \in \R_+^{n \times r}$ with $r\ge 2$ satisfies the SSC if the following two conditions hold:  
\begin{enumerate}
    \item SSC1: $\mathcal{C} = \{x \in \R_+^r \ | \ e^{\top}x \geq \sqrt{r-1} \|x\|_2 \} \subseteq \cone(H^\top)$, where $e$ is the vector of all ones of appropriate dimension. 
    
    \item SSC2: $\cone^*(H^\top) \cap \textbf{bd}(\mathcal{C}^*) =  \{\lambda e_k  \ | \  \lambda \geq 0  \text{ and } k \in [r]\}$, where $\textbf{bd}$ denotes the border of a cone,  
    $\mathcal{C}^* = \{x \in \R^r  \ | \ e^{\top}x \geq \|x\|_2 \}$, and $e_k$ is the $k$th unit of appropriate dimension. 
\end{enumerate}
\end{definition}
SSC1 requires that $\cone(H^\top)$ contains the ice-cream cone that is tangent to every facet of the nonnegative orthant; see Figure~\ref{fig:sepSSC} for an illustration. Equivalently, by using the definition of dual cone \eqref{eq:dualcone}, one can rewrite the condition in SSC1 as 
    $
    \cone^*(H^\top)\cu \mathcal C^*
    $.
SSC2 is typically satisfied if SSC1 is, and allows one to avoid pathological cases; see~\cite[Chapter 4.2.3]{Gil20} for more details. 
\begin{figure}[ht!]
	\begin{center}
		\includegraphics[width=0.8\textwidth]{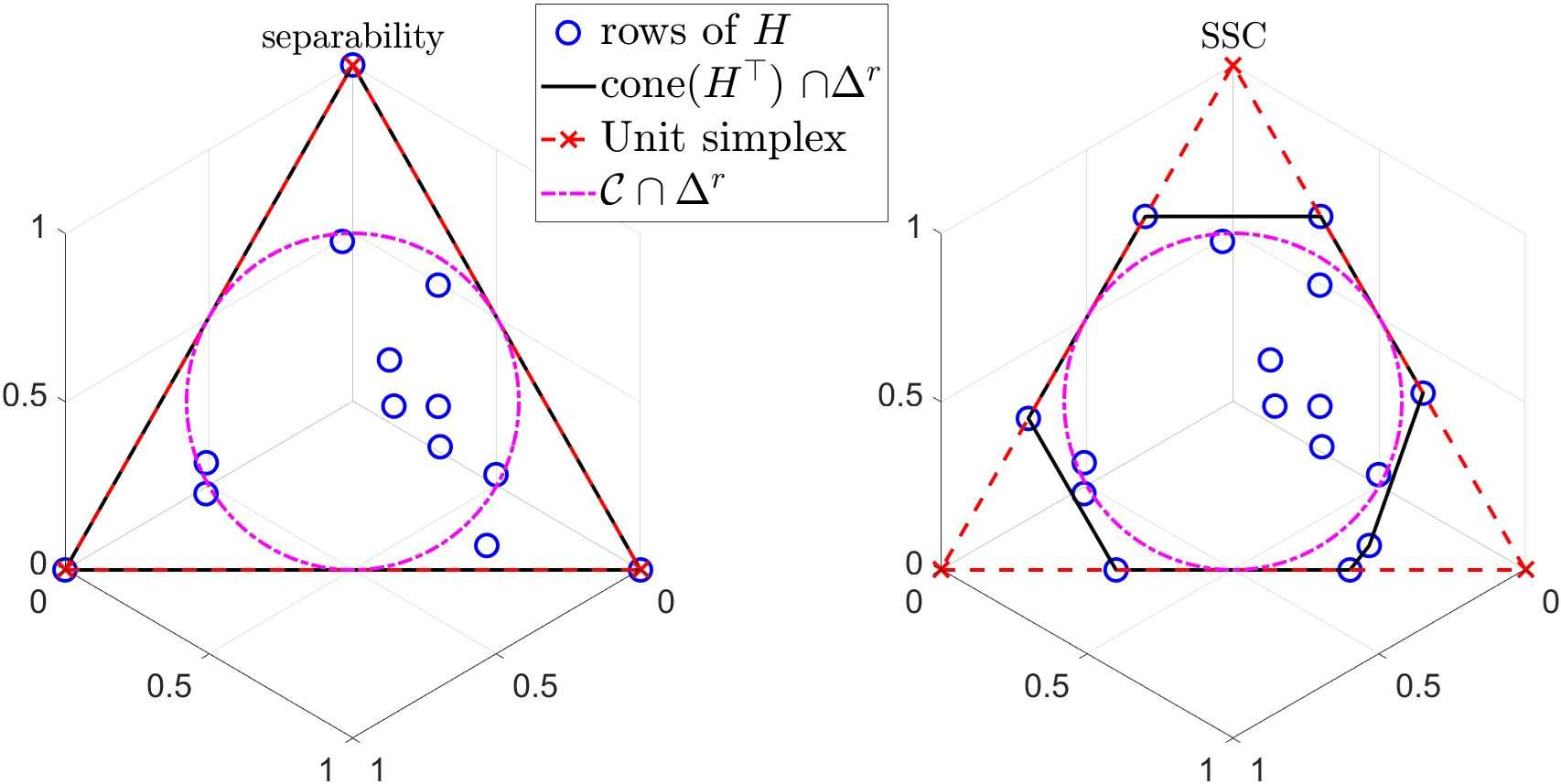}   
		\caption{Comparison of separability (left) and the SSC (right) for the matrix $H$ whose columns lie on the probability simplex, $\Delta^r = \{ x \ | \ x \geq 0, e^\top x = 1 \}$, in the case $r = 3$. 
			On the left, separability requires the columns of $H$ to contain the unit vectors, that is, $H(:,\mathcal{K}) = I_r$ for some $\mathcal{K}$.
			On the right, the SSC requires $\mathcal{C} \subseteq  \cone(H)$. 
			 Figure adapted from~\cite{abdolali2021simplex}. \label{fig:sepSSC}
   }  
	\end{center}
\end{figure} 

 Note that separability implies the SSC, that separability and the SSC coincide when $r \leq 2$, and that the SSC implies that a matrix is full rank.


To leverage the SSC, we need to notion of volume. 
Given a matrix $W \in \mathbb{R}^{m \times r}$ with $\rank(W) = r$, the quantity $\det(W^\top W)$ is a measure of the volume of the columns of $W$; more precisely, it is the volume of the convex hull of the columns of $W$ and the origin in the linear subspace
spanned by the columns of $W$. 
Combining volume minimization of $W$ under the SSC of $H$ leads to unique NMFs.  
In fact, let us consider the following optimization problem, referred to a minimum-volume (min-vol) NMF:  
\begin{equation}\label{eq:optminvol} 
    \min_{W \in \mathbb{R}^{m \times r}, H \in \mathbb{R}^{n \times r}} \; \det(W^\top W) 
    \quad \text{such that } 
    \quad X = W H^\top, 
    H^\top e = e,  
     \text{ and } H \geq 0. 
\end{equation}
We have the following identifiability result. 
\begin{theorem}[Identifiability of min-vol NMF]\cite{fu2018identifiability}\label{thm:idenminvol}
Let $X \in \mathbb R^{m \times n}$ admit the decomposition  $X = W_{\#} H_{\#}^{\top}$ where $H_{\#} \in \R^{r \times n}_+$ satisfies the \textit{SSC} and $r = \rank(X)$. 
Then $X = W_{\#} H_{\#}^{\top}$ is essentially unique w.r.t.~\eqref{eq:optminvol}.  
\end{theorem}
In simple terms, the SSC of $H$ in an NMF $X = WH^\top$ with $H\top e = e$ implies that there exist no other factorization where the first factor has a smaller volume. This intuition was introduced in the hyperspectral unmixing literature~\cite{craig1994minimum}. Note that the first variant of identifiability for min-vol NMF under the SSC was proved in~\cite{FMHS15, lin2015identifiability} under the constraint $H e = e$, which is not w.l.o.g.\ since normalizing the rows of $H$ makes the column of $X$ belong to the convex hull of the columns of $W$. 
Later it was relaxed to the normalization $H^\top e=e$ in~\cite{fu2018identifiability}, and to $W^\top e = e$ in~\cite{leplat2020blind}.

\begin{remark}[Nonnegativity of $X$ and $W$] \label{rem:nonneg} The nonnegativity of $X$ and $W$ is not necessary for Theorems~\ref{thm:idenseparableNMF} and~\ref{thm:idenminvol} to hold. Hence, there is a slight abuse of language when referring to separable NMF and min-vol NMF since, in such decompositions, $W$ could potentially have negative entries. The reason is that these models appeared in the NMF literature, hence authors kept the name NMF, although using the term semi-NMF would have been more appropriate. We refer the interested reader to the discussions in~\cite[Chapter~4]{Gil20} for more details.      
\end{remark}

\section{Previous Works on the identifiability of nTD} \label{sec:previousworks}


As for NMF, the nonnegativity constraints in nTD allows for a meaningful interpretation of the factors which may have a probabilistic or a physical meaning. 
In fact, nTD has been used in many applications, including 
clustering \cite{clustering2024}, 
hyperspectral image denoising and compression \cite{denoising18,compression2016}, 
audio pattern 
extraction~\cite{marmoret2021uncover, smith2019unmixer}, 
image fusion~\cite{ZHKS21}, 
community detection in multi-layer networks~\cite{wang2019multiway, agterberg2024estimating}, topic modeling~\cite{AHJK13}, 
and EEG signal analysis \cite{Rostakova2020126138,yin2022EEG}; see also~\cite{cichocki2009nonnegative} and the references therein. 
We will present several applications in Part II of this paper~\cite{sahaPartII2025}, as well as standard algorithms to compute nTDs.  
Let us now briefly describe previous results on the identifiability of nTD. 
 
\paragraph{Connections between NTD and NMF}

Zhou et al.~\cite{Zhou_2015} have explored the connections between the identifiability of NTD and NMF. They stated that if there exists a unique NMF of the unfoldings along all modes of the tensor, then the NTD of the given tensor is essentially unique. Moreover, they also look at the so-called population NTD model where one factor matrix is the identity matrix. In that case, they showed that uniqueness of the NMF of the  unfolding  along the mode corresponding to the factor matrix which is the identity translates to uniqueness of the NTD as well. 

As we will see in the rest of this section, further research in this direction has focused on imposing stricter conditions on the factors and the core tensor on the underlying NTD of the given tensor in order to achieve identifiability results.

\paragraph{Separability} 

Agterberg and Zhang~\cite{agterberg2024estimating} studied the identifiability of order-$3$ nTD where $U_i \in \mathbb R_+^{n_i \times r_i}$ are separable for $i \in [3]$. Their motivation comes from the interpretability of this model as a tensor version of the 
mixed-membership stochastic blockmodel~\cite{airoldi2008mixed}.  
Each factor, $U_i$, provides the membership of $r_i$ communities; more precisely, 
$U_i(:,j)$ is the membership vector of the $j$th community along the $i$th mode, with the magnitude of the entry representing the intensity of membership within that community. The core tensor, $\mathcal G$,   provides the relationship between the communities;  $\mathcal G(i,j,k)$ indicates the magnitude of the interactions between three communities: the $i$th in mode 1, the $j$th in mode 2 and the $k$th in mode 3. 
Separability of the $U_i$'s in this context requires that, for each community, there exists a pure node that only belongs to that community.  
Agterberg and Zhang~\cite{agterberg2024estimating} proved the following result. 
\begin{theorem}\cite[Proposition 1]{agterberg2024estimating} \label{thm:agterberg2024}
Let $\mathcal T = (U_1,U_2,U_3).\mathcal{G} \in \R^{n_1 \times n_2 \times n_3}$ be an order-$3$ tensor where $U_i \in \mathbb{R}^{n_i \times r_i}_+$ are separable for $i \in [3]$. 
If, for all $k \in [3]$, the unfolding of the tensor $\mathcal T$ along the $k$th mode, 
that is, $\mathcal T_{(k)}$, 
has rank $r_k$, then for any other decomposition of the form $\mathcal T = (U'_1,U'_2,U'_3).\mathcal{G}'$ where  $U'_i $ are separable for $i \in [3]$, there exist 
there exist permutation matrices $\Pi_1,\Pi_2,\Pi_3$ and diagonal matrices $D_1,D_2,D_3$ such that 
$U_{k} = {U'_{k}} D_k \Pi_k$ for all $k \in [3]$  and  $\mathcal{G} = (\Pi_1^\top D_1^{-1},\Pi_2^\top D_2^{-1},\Pi_3^\top D_3^{-1}).\mathcal{G}'$. 
\end{theorem}

 The proof essentially follows the identifiability of separable NMF (Theorem~\ref{thm:idenseparableNMF}), and can actually be generalized to any order; see Theorem~\ref{thm:agterberg2024gen} in Appendix~\ref{app:separab}. 
Moreover, in Section~\ref{sec:seporderd}, we show that one can replace the rank assumption on the unfoldings along all modes with a rank assumption on the unfolding along only one mode (although this requires some restrictions on the dimensions of the core tensor).


\paragraph{Sufficiently Scattered Condition (SSC)}  

Sun and Huang~\cite{SH23} studied the identifiability of  nTD where $U_i \in \R_+^{n_i \times r_i}$ satisfy the SSC for $i \in [d]$.  
They considered the following optimization problem: 
\begin{equation}\label{eq:optSH23}
    \max_{\mathcal{G}, U_1,\dots,U_d} \sum_{i=1}^d \logdet(U_i^\top U_i) \; \text{ such that } \; 
    U_i \geq 0, 
    U_i^\top e = e \text{ for } i \in [d], 
    \mathcal T = (U_1,\dots,U_d).\mathcal{G} , 
\end{equation}
where they maximize the volume of the factors of the nTD. They claimed that if there exists a decomposition of the core tensor $\mathcal T$ such that all $U_i$'s satisfy the SSC, then the nTD of $\mathcal T$ obtained from an optimal solution to (\ref{eq:optSH23}) recovers the $U_i$'s, up to permutation~\cite[Theorem~1]{SH23}.  
Unfortunately, in this conference version of the paper~\cite{SH23}, the authors do not discuss the crucial assumptions regarding the core tensor (nor provide a proof).  
Hence, we provide here a complete result, whose proof follows directly from max-vol NMF identifiability on each unfolding~\cite[Theorem 7.1]{olivierthesis2024}.  
\begin{theorem} \label{thm:maxvolNTD}
Let $\mathcal T = (U_{1}^{\#},\dots,U_{d}^{\#}).\mathcal{G}^{\#}$ where $U_{i}^{\#} \in \R^{n_i \times r_i}$ 
satisfy the SSC, 
and  
 the unfoldings of $\mathcal{G}^{\#}$ along all modes $k \in [d]$, denoted by $G_{(k)}^{\#}$, have rank $r_k$. 
 Then $\mathcal T = (U_{1}^{\#},\dots,U_{d}^{\#}).\mathcal{G}^{\#}$ is essentially unique w.r.t.~\eqref{eq:optSH23}.  
\end{theorem}

\paragraph{The $d$-persistent topic models}  

Anandkumar et al.~\cite{AHJK13} discussed the notion of identifiability of the $d$-persistent topic models. For $d$-persistent topic, the $2d$-th order moment of the model, denoted $M_d$, is an unfolding of an order-$2d$ tensor which has a decomposition of the form 
$$
M_{2d} \; = \; (A ^{\odot d}) \, G \, (A ^{\odot d})^\top, 
$$ 
where $A^{\odot d}$ is the Khatri-Rao product of the matrix $A$ taken $d$ times and $G$ is the covariance matrix of the \textit{hidden variables}. 
This is a special case of a TD, and a generalization of the CPD. 
They showed that the columns of $A$ are identifiable under some non-trivial conditions. Using this tri-NMF structure, one could further enquire the identifiability of the non-negative version of the $2d$-persistent topic model using the framework introduced in this paper.
This is left for a future work.


\paragraph{How our work goes beyond the state of the art}   

In this paper, we will exploit the structure of the slices of the tensor and of its unfoldings, and will \textit{minimize the volume of the core tensor}. 
Our modeling approach has several advantages: 

\begin{itemize}
    \item When using unfoldings, instead of enforcing conditions on the unfoldings of the tensor along all the modes, as in previous works (in particular  Theorems~\ref{thm:maxvolNTD} and~\ref{thm:agterberg2024}, and  in~\cite{Zhou_2015}), 
    we restrict our attention to the unfolding along just a single mode, leading to stronger identifiability results;     
    see Section~\ref{sec:unfoldings} for order-$3$ tensors, and Section~\ref{app:unfoldorderd} for higher orders. 
    

    \item When using slices, 
    the dimensions of the matrix factorization subproblems that we will solve will be significantly smaller compared to previous works --we will only need to consider the factorization of $d-1$ slices. 
    We will show that, under certain restrictions on the factors and the core tensor, the slices have an essentially unique factorization which will imply the identifiability of the underlying nTD; 
    see Section~\ref{sec:slices} for order-$3$ tensors, 
    and Section~\ref{app:orderdslices} for higher orders.

\end{itemize}

Table~\ref{tab:summary3} summarizes the main cases in which we can prove identifiability for order-$3$ nTDs. These results generalize to any order; see Section~\ref{app:orderd}. 
\renewcommand{\arraystretch}{1.2} 
\begin{center}
\begin{table}[h!]
\begin{center}
\caption{Main conditions on the $U_i$'s, $\mathcal G$ and $\mathcal T$ to have an identifiability procedure for order-$3$ nTD. 
The matrix $\mathcal G_{(3)}$ is the third mode unfolding of $\mathcal G$. The matrix $\mathcal T^{(k)}_{i}$ is the $i$th slice along the $k$th mode. By symmetry of the problem, the roles of the modes are interchangeable.  
\label{tab:summary3}
}
\label{numreal}  
\begin{tabular}{c|c|c|c}
 $\mathcal T = (U_1,U_2,U_3).\mathcal G$ & $\mathcal G$  & $U_i$'s & $\mathcal T$ \\ 
\hline 
Procedure~0
& $\rank(\mathcal G_{(3)}) = r_3 = r_1 r_2$ 
& $U_1 \otimes U_2$ is SSC  
& 
\\ 
(Theorem~\ref{thm:order2unfold}) 
&
& $U_3$ is SSC   
& \\ 
\hline
Procedure~1
& $r_3 \leq r_1 = r_2$  
& $U_i$'s are SSC 
&  
$\exists i_2$ s.t. $\rank(\mathcal T^{(2)}_{i_2})  = r_3$ 
\\ 
(Theorem~\ref{thm:idea1order3})
& 
& 
&  
$\exists i_3$ s.t. $\rank(\mathcal T^{(3)}_{i_3}) = r_1$ 
\\ 
\hline 
Procedure~3 
&  $\sqrt{r_3} \leq r_1 = r_2$  
& $U_i$'s are SSC  
& $\exists i$ s.t. $\rank(\mathcal T^{(3)}_{i}) = r_1$
\\ 
(Theorem~\ref{thm:idea3order3}) 
&  $\rank(\mathcal G_{(3)}) = r_3$ 
& 
& 
\\ 
\end{tabular} 
\end{center}
\end{table}
\end{center}
Note that Procedure~2 (resp.\ 4) will be a randomized version of Procedure~1 (resp.\ 3) with weaker assumptions: instead of requiring two (resp.\ one) slices to have maximum rank, we will just need the span of all the slices to have maximum rank.

Finally, our approaches also have algorithmic advantages over previous works. 
We defer the details to Part II of this paper~\cite{sahaPartII2025}.


\section{Order-2 nTD a.k.a.\ Nonnegative Matrix Tri-Factorization} \label{sec:NMF}

Given a data matrix $X \in \R^{m \times n}_+$, order-2 nTD decomposes $X$ into a product of three factor matrices $U_1 \in \R^{m \times r}_+$, $G \in \R^{r \times r}_+$ and $U_2 \in \R^{n \times r}_+$ such that 
    $X = U_1 G U_2^\top$, which is also known as nonnegative matrix tri-factorization~\cite{ding2006orthogonal}. 
When $G$ is the identity matrix, order-2 nTD boils down to NMF. Otherwise, order-2 nTD can never be identifiable; in fact, we have 
\[
U_1 \, G \, U_2^\top 
\; = \; U_1 \, I_r \,  (U_2 G^\top)^\top 
\; = \; (U_1 G) \, I_r \, U_2^\top  . 
\]  
Therefore, we must use additional constraints and/or regularization to make order-2 nTD identifiable. In practice, people used sparsity constraints; see, e.g., \cite{ding2006orthogonal, wang2011fast, zhang2012overlapping, kolomvakis2025boolean}. 

Surprisingly, although order-2 nTDs have been extensively used in practice, there exists, to the best of our knowledge, no identifiability results in this case. This section aims to fill in this gap. 
We discuss the case of separabiliy in Section~\ref{sec:idenseparable}, and the case of the SSC in Section~\ref{sec:minvolTriMF}.

\subsection{Order-2 nTD under separability}\label{sec:idenseparable}

The notion of separable NMF can be extended from two to three factors and such decompositions have been extensively studied in the matrix factorization literature in the symmetric case. Given a symmetric matrix $A \in \R^{m \times m}$, one can consider its decomposition $A = W S W^\top$, sometimes referred to as tri-symmetric NMF. 
This decomposition occurs in community detection in graphs, and is closely related to the stochastic block models: the matrix $W$ contains the community information, and the matrix $S$ the interactions between the communities; see, e.g., \cite{xiao2019uniq} and the references therein. 
Arora et al.~\cite{AGHMMSWZ13} gave a provably-robust polynomial-time algorithm for tri-symmetric NMF with identifiability guarantees. 
They used it in the context of topic modeling where $A$ is the word co-occurence matrix, the matrix $W$ is the word-by-topic matrix such that $W(i, k)$ is the probability for word $i$ to be picked under the topic $k$, and the matrix $S$ is a topic-by-topic matrix that accounts for the interactions between the topics. 
In topic modeling, the following separability assumption makes sense: for each topic, there exists an anchor word, that is, a word that is only used by that topic. 
Following Definition~\ref{def:separablematricesfirst}, this is equivalent to requiring
that the matrix $W^\top$ is separable, that is, there exists $\mathcal{K}$ such that $W(\mathcal{K}, :)$ is a diagonal matrix. Moreover, to eliminate the scaling degree of freedom, one can assume  w.l.o.g.\ that $W^\top e = e$. 

In this section, we present a generalization of the symmetric model to the non-symmetric case. In the noiseless setting, we assume that there exist matrices $U_1 \in \R^{n_1 \times r_1}_+$, $G \in \R^{r_1 \times r_2}$, $U_2 \in \R^{n_2 \times r_2}_+$ such that $X = U_1 G U_2^\top$. 
Moreover, we assume that $U_1$ and $U_2$ are separable. 
This model can be interpreted in the following way from the point of community detection: the matrices $U_1$ and $U_2$ are the membership indicator matrices where $U(j_1,k_1)$ is the membership indicator of item $j_1$ for community $k_1$ and $U_2(j_2,k_2)$ is the membership indicator of item $j_2$ for community $k_2$. The entry $G(k_1,k_2)$ is the strength between the communities $k_1$ and $k_2$. Since $U_1 \neq U_2$ in general, $U_1$ and $U_2$ correspond to different communities. For example, if $X$ is a movie-by-user matrix (e.g., $X(i,j) = 1$ if user $j$ has watched movie $i$, and $X(i,j) = 0$ otherwise), $U_1$ will correspond to communities of movies while $U_2$ to communities of users.   Separability of $U_1$ (and similarly for $U_2$) requires that, for each community (that is, for each column of $U_1$), 
there exist a node that is only associated to that community. Such nodes are referred to as pure nodes or anchor nodes.

In the following theorem, we show that when $r_1 = r_2 = r = \rank(X)$ and under the separability assumption, such a decomposition is \textit{essentially unique}. 
\begin{theorem} \label{th:serptrisymNMF}
 Let $X = U_1 G U_2^\top$ where $U_1 \in \R^{n_1 \times r}_+$ and $U_2 \in \R^{n_2 \times r}_+$ are separable matrices, and $r = \rank(X)$. 
 For any other decomposition of $X = U^*_1 G^* {U^*_2}^\top$ of size $r$ where $U^*_1$ and $U^*_2$  are separable, there exist permutation matrices $\Pi_1,\Pi_2$ and diagonal matrices $D_1,D_2$ such that 
 $U^*_i = U_i D_i \Pi_i$ for $i \in [2]$, and 
 $G^* = \Pi_1^\top D_1^{-1}G D_2^{-1}\Pi_2$.
\end{theorem}

\begin{proof}
Let $X = W H^\top$ with $W= U_1G$ and $H = U_2$ be an NMF of $X$ of size $r$, and $X = W_* H_*^\top$ with $W_*= U_1^* G^*$ and $H = U_2^*$ be an another one. 
Since, $U_2$ and $U_2^*$ are separable matrices, 
using Theorem~\ref{thm:idenseparableNMF}, 
there exist a permutation matrix $\Pi_2$ and a diagonal matrix $D_2$ such that $U_2^* = U_2 D_2 \Pi_2$. 
Using the same argument on $X^\top 
= U_2 G^\top U_1^\top = U_2^* {G^*}^\top {U_1^*}^\top$, we obtain that $U_1^* =  U_1 D_1  \Pi_1$ for some permutation matrix $\Pi_1$ and a diagonal matrix $D_1$. 
Finally, 
$X = U_1 G U_2^\top = U_1^* G^* {U_2^*}^\top$  
 implies
 that $G^* = \Pi_1^\top D_1^{-1} G  D_2^{-1}\Pi_2$ since all factor matrices $(U_1,U_2,U_2^*,U_2^*)$ are separable, hence have rank $r$.  
\end{proof}

Note that this result is not applicable  
when $r_1 \neq r_2$. This is because $G$ is rank deficient, since $\rank(G) \leq \min(r_1,r_2)$. However, it can be easily extended if $\cone(X)$ has $r_1$ rays and $\cone(X^\top)$ has $r_2$ rays; similarly as for separable NMF (see Footnote~\ref{footnote:rays}). However, to keep it simple, we restrict our analysis to the full-rank case. 

In part II of this paper~\cite{sahaPartII2025}, we will extend the ideas from \cite{AGHMMSWZ13} to give a polynomial-time robust algorithm for computing separable order-2 nTDs.

\subsection{Order-2 nTD under the SSC} \label{sec:minvolTriMF}

Separability of both $U_1$ and $U_2$ is rather strong. For example, in terms of community detection, it requires all communities to have a pure node. 
In order to provide a relaxation to the notion of separability, 
we now generalize min-vol NMF under the SSC (Theorem~\ref{thm:idenminvol}) for order-2 nTD.  



We propose the following formulation: Given the matrix $X \in \R^{n_1 \times n_2}$ and the factorization rank $r = \rank(X)$, 
\begin{equation}\label{eq:optcriterion}
    \min_{U_i \in \R^{n_i \times r} \, i \in [2], 
    G \in \R^{r \times r}} 
    |\det(G)| 
    \quad 
     \text{ such that } 
     \quad  
     X = U_1 G U_2^\top, 
     U_i^{\top} e =  e, 
     U_i \geq 0 \text{ for } i \in [2]. 
\end{equation}
We refer to this problem as min-vol order-2 nTD. The constraints $U_i^{\top} e =  e$ for $i \in [2]$ remove the scaling ambiguity. 
Let us provide an identifiability result for~\eqref{eq:optcriterion} under the SSC. This is a generalization of the symmetric case studied in~\cite{huang2016anchor, FHS19}, where $X$ is assumed to be a symmetric matrix and $U_1 = U_2$. 
\begin{theorem}\label{thm:nmfmain}
Let $X = U_1^{\#} G^{\#} {U_2^{\# \top}}$ be a matrix of rank $r$ where $U^{\#}_i \in R^{n_i\times r}$ satisfy the SSC and ${U^{\# \top}_i} e = e$ for $i \in [2]$. 
Then $X = U_1^{\#} G^{\#} {U_2^{\# \top}}$ is essentially unique w.r.t.~\eqref{eq:optcriterion}. 
\end{theorem}
\begin{proof} 
Since $\rank(X) = r$ and $X = U_1^{\#}G^{\#}U_2^{\# \top} = U_1^* G^* U_2^{* \top}$, we have $\rank(U_1^{\#}) = \rank(U_1^*) = \rank(U_2^*) = \rank(U_2^{\#}) = r$. 
Hence there exist invertible matrices $A, B \in \R^{r \times r}$ such that $U_1^* = U_1^{\#} A$, 
$U_2^* = U_2^{\#} B$, 
and $G^* = A^{-1} G^{\#} B^{-\top}$. 
Since $(U_1^*,G^*,U_2^*)$ is an optimal solution to (\ref{eq:optcriterion}), it follows that
\begin{equation}\label{eq:detmultiplicativity}
    |\det(G^*)| = \left|\det\big(A^{-1}G^{\#}B^{-1}\big)\right| \leq \left| \det\big( G^{\#} \big) \right|.
\end{equation}
This gives us that 
\begin{equation}\label{eq:detAdetBlb}
    |\det(A)| \ |\det(B)| \geq 1.
\end{equation} 
Moreover, since $(U_1^{\#},G^{\#},U_2^{\#})$ and $(U_1^*,G^*,U_2^*)$ are feasible solutions to (\ref{eq:optcriterion}), 
\begin{equation}\label{eq:l1normofrowsandcols}
    e = U_1^{* \top} e = A^{\top} U_1^{\# \top}e = A^{\top} e, 
    \;  \text{ and } \; 
    e = U_2^{* \top} e = B^\top U_2^{\# \top} e = B^{\top} e. 
\end{equation}
Moreover, $U_1^* = U_1^{\#}A \geq 0$ and $U_2^{*} = U_2^{\#}B \geq 0$. Following the definition of dual cones in (\ref{eq:dualcone}), this gives us that for all $i,j \in [r]$,
\begin{equation}\label{eq:rowscolscontainedindualcone}
    A[:,j] \in \cone^*\big(U_1^{\# \top}\big) \text{ and } B[:,i] \in \cone^*\big(U_2^{\# \top}\big), 
\end{equation}
where $A[:,j]$ denotes the $j$th the columns of $A$. 
Since $U_1^{\#}$ and $U_2^{\#}$ satisfy the SSC, we have for all $i,j \in [r]$ 
\begin{equation}\label{eq:indualcone}
A[:,j], B[:,i] \in \mathcal{C}^* = \left\{x \in \R^r_+ \ | \ e^{\top}x \geq \|x\|_2 \right\}.
\end{equation}
Then, using (\ref{eq:indualcone}) and (\ref{eq:l1normofrowsandcols}), 
we get 
\begin{equation}\label{eq:detAdetBub}
    |\det(A)| \leq \prod_{j=1}^r \|A(:,j)\|_2 \leq  \prod_{j=1}^r e^{\top}A(:,j) = 1,   
\end{equation}
and similarly for $B$. 
Using (\ref{eq:detAdetBlb}) and (\ref{eq:detAdetBub}), we obtain $|\det(A)| = |\det(B)| = 1$.
Since the equality holds in (\ref{eq:detAdetBub}), $A(:,j), B(i,:) \in \textbf{bd}(\mathcal{C}^*)$. This, along with (\ref{eq:rowscolscontainedindualcone}) and the definition of SSC, gives us that for all $i,j \in [r]$,
\begin{align*}
    A[:,j],B[i,:] \in  \{\lambda e_k | \lambda \geq 0  \text{ and } k \in [r]\}. 
\end{align*}
Since $A^{\top} e = e$, $B^{\top} e = e$ and $|\det(A)| = |\det(B)| = 1$, 
$A$ and $B$ are permutation matrices. 
\end{proof}

The proof of the previous theorem implies that identifiability is actually achieved by any feasible solution of~\eqref{eq:optcriterion} whose objective function value is smaller than that of the ground truth. 
Let us state this result explicitly as it will be useful later on. 
\begin{corollary}\label{corr:suboptimalsoltn}
  Let $X = U_1^{\#} G^{\#} {U_2^{\# \top}}$ be a matrix of rank $r$ where $U^{\#}_1$ and $U^{\#}_2$ satisfy the SSC 
  and 
${U^{\# \top}_i} e = e$ for $i \in [2]$.  
Let $U_1^* \in \R_+^{m \times r}, U_2^* \in \R_+^{n \times r}$ and $G^* \in \R^{r\times r}$ be matrices such that $X$ has a decomposition of the form $X = U_1^{*} G^{*} {U_2^{* \top}}$ where ${U^{* \top}_i} e = e$ for $i \in [2]$ and $|\det(G^*)| \leq \big|\det\big(G^{\#}\big)\big|$. Then there exist permutation matrices  $\Pi_1,\Pi_2$ such that 
$U_1^* = U_1^{\#} \Pi_1$, 
$U_2^* = U_2^{\#} \Pi_2$, 
and 
$G^* = \Pi_1^{\top} G^{\#} \Pi_2$.    
\end{corollary}

 Can the above result be extended to the case when $r_1 \neq r_2$? As discussed in the previous section, this corresponds to a \textit{rank-deficient} case. The characterization from (\ref{eq:triNMF}) does not work because $\det(G) = 0$. 
 For $r_1 \leq r_2$, one could also try to minimize $\det\big(G^\top G\big)$ instead, and the proof of Theorem~\ref{thm:nmfmain} works for that function when $r_1 = r_2$, since $\det\big(G^\top G\big) = \det(G)^2$ for a square matrix $G$. 
 The proof does not extend directly for the rank-deficient case: in fact, 
there exist matrices $A,B$ such that 
$$ 
\det(B^{-\top}G^{\top}A^{-\top}A^{-1}GB^{-1}) < \det(G^{\top}G)
$$ 
while $\det(A) \det(B) < 1$. Take for example $B = 1, A = \begin{bmatrix}
    1 & \frac{1}{2} \\
    0 & \frac{1}{2}
\end{bmatrix}$, $G = \begin{bmatrix}
    1 \\
    0
\end{bmatrix}$. We have $C := A^{-1}G = \begin{bmatrix}
    1 \\
    0
\end{bmatrix}$, $\det(C^{\top}C) = 1 = \det(G^{\top}G)$ but $\det(A)^2 = \frac{1}{4} < 1$. 
Note that here $A$ satisfies $A^\top e = e$. 
 

\section{Order-$3$ nTDs}\label{sec:idenorder3tensors}

In this section, we provide 5 different approaches to obtain identifiable order-$3$ nTDs. 
Each approach relies on different assumptions and optimization models. We first consider unfoldings for which we provide a natural approach using directly order-2 nTD identifiability under the SSC (Theorem~\ref{thm:nmfmain}); see Section~\ref{sec:unfoldings}. 
Then we propose 4 approaches relying on slices; see Section~\ref{sec:slices}. 

These 5 approaches can be generalized to higher-order tensors, which we explain in Section~\ref{app:orderd}. We start with order-$3$ tensors to make the presentation simpler and help the reader follow the line of thoughts more easily.

From now on, we focus on the SSC using minimum-volume minimization. However, all identifiability results presented in this section directly specialize to the separable case, since separability implies the SSC. 
To avoid redundancy and keep the paper more concise, 
we do not explicitly state these identifiability results in the separable case. 
The reason we specified the identifiability for separable order-2 nTD (Theorem~\ref{th:serptrisymNMF}) is twofold:  
\begin{itemize}
    \item This provided an example of identifiability for nTD under the separability, and 
    
    \item We will design a polynomial-time and robust algorithm for  separable order-2 nTD in part II of this paper~\cite{sahaPartII2025}. It generalizes the algorithm for the symmetric case proposed in~\cite{AGHMMSWZ13}. We will use this algorithm to initialize higher-order nTD algorithms. 

\end{itemize}

Throughout this section, we will assume the tensor follows the following assumption. 
\begin{assumptions}\label{ass:generalass}
The order-$3$ tensor $\mathcal T\in \R^{n_1 \times n_2 \times n_3}$ satisfies the following conditions: 
\begin{enumerate}
    \item $\mathcal T$ has an nTD of the form $\mathcal T = (U^{\#}_1,U^{\#}_2,U^{\#}_3).\mathcal{G}^{\#}$ 
    where $U^{\#}_i \in \R_+^{n_i \times r_i}$ for all $i \in [3]$ and $\mathcal{G}^{\#} \in \R^{r_1 \times r_2 \times r_3}$, and no columns of the $U^{\#}_i$'s are zero. 
    
    \item To simplify the presentation and remove the scaling ambiguity, we assume w.l.o.g.\ that  $(U^{\#}_i)^\top e = e$  for $i \in [3]$. 
    
\end{enumerate}
\end{assumptions}

\subsection{Using unfoldings} \label{sec:unfoldings} 

Let us start with the most natural approach: 
reduce an order-$3$ nTD to an equivalent order-2 nTD using unfoldings. For this, let us recall the following property.  

\begin{property}[Unfoldings] \label{lem:changeofbasisunfoldings}
Let $\mathcal T\in \mathbb{R}^{n_1 \times n_2 \times n_3}$ be an order-$3$ tensor with $\mathcal T = (U_1,U_2,U_3).\mathcal G$ where $\mathcal G \in \mathbb{R}^{r_1 \times r_2 \times r_3}$. 
The unfolding of $\mathcal T$ along the third mode is given by
$$
\mathcal T_{(3)} \; = \;  (U_1 \otimes U_2) \, \mathcal G_{(3)} \, U_3^{\top}, 
$$
where $\mathcal G_{(3)}$ is the unfolding of $\mathcal G$ along the third mode. 
By symmetry, the same property applies to unfoldings along the first and second modes. 
\end{property}

Property~\ref{lem:changeofbasisunfoldings} allows us to apply directly Theorem~\ref{thm:nmfmain} for order-2 nTD to obtain an identifiable order-$3$ nTD. 
However, this requires a very strong conditions, namely $r_3 = r_1 r_2$, because Theorem~\ref{thm:nmfmain} only applies to order-2 nTD where the two inner ranks are equal to each other. 
Let us state the assumptions needed to obtain a first identifiability result for order-$3$ nTD. 

\begin{assumptions}\label{ass:unfold}
The order-$3$ tensor $\mathcal T\in \R^{n_1 \times n_2 \times n_3}$ satisfies Assumption~\ref{ass:generalass} and  
\begin{enumerate}
    \item $\rank(\mathcal G_{(3)}^{\#}) = r_3 = r_1 r_2$. 

    
    \item $(U_1^{\#} \otimes U_2^{\#})$ and $U^{\#}_3$ satisfy the SSC . 

\end{enumerate}
\end{assumptions}

Note that, for simplicity of the presentation, this assumption arbitrarily unfolds the tensor in the third mode. Of course, by symmetry of the problem, one can apply the same reasoning to unfoldings in the first and second modes. 
We will discuss the condition $(U_1^{\#} \otimes U_2^{\#})$ satisfies the SSC in Section~\ref{subsec:KronSSCconj}. 


By Property~\ref{lem:changeofbasisunfoldings}, $\mathcal T_{(3)} \; = \;  (U_1 \otimes U_2) \, \mathcal G_{(3)} \, U_3^{\top}$, and, under Assumption~\ref{ass:unfold}, we can recover, up to permutations, $U_1 \otimes U_2$, $\mathcal G_{(3)}$ and $U_3$ using order-2 nTDs (Theorem~\ref{thm:nmfmain}). 
It remains to recover $U_1$ and $U_2$ from their Kronecker product, $U_1 \otimes U_2$, which we show how to do below (taking into account the unknown permutation). 
Let us describe this procedure: \vspace{0.1cm} 

\fbox{%
	\parbox{0.95\linewidth}{%
    \begin{center} \vspace{-0.2cm}   
       \textbf{Procedure 0: Unique order-$3$ nTD under Assumption~\ref{ass:unfold}} 
    \end{center}
\begin{enumerate}
    \item Computation of $U^*_3$: Solve the following min-vol order-2 nTD problem  
    \begin{align}\label{eq:unfoldstep1}
    \min_{G, U_{1,2}, U_3}  \big|\det\big( G \big)\big|  \text{ such that } & 
     \mathcal T_{(3)} = U_{1,2} G U_3^\top,  \\ 
     & 
     U_{1,2}^\top e = e, 
     U_3^\top e = e \text{ and } 
     (U_{1,2}, U_3) \geq 0, \nonumber 
    \end{align}
    to obtain an optimal solution $(G^*, U^*_{1,2}, U^*_3)$.

    \item Computation of $U_1^*, U_2^*$: 
    Find a permutation $\Pi^*_{1,2}$ such that $U^*_{1,2}\Pi^*_{1,2}$ admits a decomposition $U^*_{1,2}\Pi^*_{1,2} = U_1^* \otimes U_2^*$ with $U_1^{* \top} e = e$ and $U_2^{* \top} e = e$; see 
    Property~\ref{prop:kronpdtperm} and Theorem~\ref{thm:uniquenesskronpdt} below. 
    
    \item Computation of $\mathcal{G}^*: $ Compute the matrix $\mathcal{G}_{(3)}^* = (\Pi_{1,2}^*)^\top G^*$ and fold it to form the  order-$3$ tensor $\mathcal{G}^*$. 
    \vspace{-0.1cm}  
\end{enumerate}
	}%
}


\vspace{0.1cm} 

Theorem~\ref{thm:order2unfold} shows that, under Assumption~\ref{ass:unfold}, a decomposition computed by Procedure~0 is essentially unique. 
\begin{theorem}\label{thm:order2unfold}
Let $\mathcal T\in \R^{n_1 \times n_2 \times n_3}$ be an order-$3$ tensor satisfying 
Assumption~\ref{ass:unfold}, with the corresponding $(r_1,r_2,r_3)$-nTD given by  $(U^{\#}_1,U^{\#}_2,U^{\#}_3).\mathcal{G}^{\#}$ where $r_3 = r_1 r_2$.  
Then for any other decomposition $(U^{*}_1,U^{*}_2,U^{*}_3).\mathcal{G}^{*}$ obtained with Procedure~0, 
there exist permutation matrices $\Pi_1,\Pi_2,\Pi_3$ such that $U^*_i = U^{\#}_i \Pi_i$ for all $i \in [3]$ and $\mathcal{G}^* = (\Pi_1^\top,\Pi_2^\top,\Pi_3^\top).\mathcal{G}^{\#}$. 
\end{theorem}
\begin{proof}
By Assumption~\ref{ass:unfold}, $(U_1^{\#} \otimes U_2^{\#})$ and $U^{\#}_3$ satisfy the SSC and $\rank(\mathcal G^{\#}_{(3)}) = r_3 = r_1 r_2$, 
hence solving~\eqref{eq:unfoldstep1} will find matrices $(U^*_{1,2},G^*,U^*_3)$ such that there exist permutation matrices $\Pi_{1,2}, \Pi_3$ such that $U^*_{1,2} = (U_1^{\#} \otimes U_2^{\#})\Pi_{1,2}$, $U^*_3 = U^{\#}_3\Pi_3$ and $G^* = \Pi_{1,2}^\top G_{(3)}^{\#} \Pi_3$ by Theorem~\ref{thm:nmfmain}. 

Since, at the end of step~2, we have recovered matrices $U_1^*,U_2^*$ such that 
$$
(U_1^* \otimes U_2^*) = U^*_{1,2}\Pi^*_{1,2} =  (U_1^{\#} \otimes U_2^{\#})\Pi_{1,2}\Pi^*_{1,2}. 
$$ 
 Property~\ref{prop:kronpdtperm}, see below, shows  that this implies that there exist permutation matrices $\Pi_1,\Pi_2$ such that $U_i^* = U_i^{\#}\Pi_i$ for all $i \in [2]$. This also gives us that $\Pi_{1,2}^* \Pi_{1,2} = \Pi_1 \otimes \Pi_2$. Then we have that $\mathcal G_{(3)}^* = (\Pi_{1,2}^*)^\top G^* = (\Pi_1 \otimes \Pi_2)^\top G_{(3)}^{\#} \Pi_3$. On folding this matrix to form an order-$3$ tensor, we can conclude that $\mathcal{G}^* = (\Pi_1^\top, \Pi_2^\top, \Pi_3^\top).\mathcal{G}^{\#}$.
\end{proof} 


The second step of Procedure~0 requires to find a permutation $\Pi_{1,2}^*$ such that the following problem has a feasible solution: given $U_{1,2}^* \Pi_{1,2}^* \in \R^{n_1n_2 \times r_1r_2}$, 
find ${U_i^*}^\top e = e$, $U_i^* \geq 0$ for $i \in [2]$ such that $U_{1,2}^* \Pi_{1,2}^* = U_1^* \otimes U_2^*$. 
Theorem~\ref{thm:uniquenesskronpdt} and  Property~\ref{prop:kronpdtperm} below imply that for any such permutation, $U_1^*$ and $U_2^*$ will be permutations of 
$U_1^{\#}$ and $U_2^{\#}$, respectively. 

Trying all possible permutations of $[r_1r_2]$ for $\Pi_{1,2}^*$ is of course impractical. We will explain in part II of this paper~\cite{sahaPartII2025} how this can be done efficiently. In a nutshell, this can be done by looking at $U_{1,2}^*$ column-wise, which requires to check at most $r_1r_2$ possibilities.  Moreover, this issue disappears when using the all-at-once procedure described after Property~\ref{prop:kronpdtperm}.

\begin{theorem}\label{thm:uniquenesskronpdt}
Let $X$ be a $n_1n_2 \times r_1r_2$ 
matrix and  
with $X = U_1 \otimes U_2$ where $U_1 \in \R^{n_1 \times r_1}$ and $U_2 \in \R^{n_2 \times r_2}$. 
Such a decomposition can be computed trivially, and is unique 
if $U_1^\top e = U_2^ \top e = e$.  
\end{theorem}
\begin{proof}
From Definition~\ref{def:KronPdt}, we get that the $(i,j)$th column of $X$ is given by 
$$
X_{:,(i,j)} = U_1(:,i) \otimes U_2(:,j) \quad 
\text{for all } i \in [r_1] \text{ and } j \in [r_2]. 
$$ 
 Rewriting $X_{:,(i,j)}$ as an $r_1 \times r_2$ matrix, denoted by $\mat(X_{:,(i,j)})$, we get that
\begin{equation}\label{eq:U1U2top}
\mat(X_{:,(i,j)}) = U_1(:,i) U_2(:,j)^\top \quad 
\text{for all } i \in [r_1] \text{ and } j \in [r_2].      
\end{equation}
Note that $\mat(X_{:,(i,j)}) \neq 0$ for all $(i,j)$ since the columns of $U_1$ and $U_2$ are non-zero (their entries sum to one). 
Rank-one matrix decompositions are unique, up to scaling, and easy to compute. In fact, given two vectors $u,v$ and $A = uv^\top$ with $A(i,j) \neq 0$ for some $(i,j)$, all rank-one approximation of $A = u' v'^\top$ have the form $u' = \alpha A(:,j)$ 
and $v' = \beta A(i,:)$ for any $\alpha, \beta$ such that 
$\alpha \beta = A(i,j)^{-1}$.  
Since $U_i^\top e = e$ for $i \in [1,2]$, the scaling degree of freedom is absent, and the columns of $U_1$ and $U_2$ are uniquely determined using~\eqref{eq:U1U2top}. 
Another way to see uniqueness is to notice that $\mat(X_{:,(i,j)})e = U_1(:,i)$, and $\mat(X_{:,(i,j)})^\top e = U_2(:,j)$.  
\end{proof} 

In Part II of this paper~\cite{sahaPartII2025}, we will design an algorithm for an optimization version of this problem that minimizes $\|X - U \otimes V\|_F^2$ under the above constraints. Note that, in the unconstrained case, this problem can be solved via an SVD computation~\cite{VanLoan1993}. 

\begin{property}\label{prop:kronpdtperm}
Let $X$ be an $n_1n_2 \times r_1r_2$ matrix with a decomposition of the form $X = U^{\#}_1 \otimes U^{\#}_2$ where $U^{\#}_1 \in \R^{n_1 \times r_1}$ and $U^{\#}_2 \in \R^{n_2 \times r_2}$ are full column rank matrices. Let $Y = X \Pi$ where $\Pi$ is a permutation matrix of dimension $r_1 r_2$, 
and $Y = U^*_1 \otimes U^*_2$ where $U^*_1 \in \R^{n_1 \times r_1}$ and $U^*_2 \in \R^{n_2 \times r_2}$ are full column rank matrices. Moreover, $(U_1^*)^\top e = (U_1^{\#})^\top e = e$ and $(U^*_2)^\top e = (U_2^{\#})^\top e = e$. Then there exist permutation matrices $\Pi_1$ and $\Pi_2$ such that $U_1^* = U_1^{\#} \Pi_1$ and $U_2^* = U_2^{\#} \Pi_2$. 
\end{property}
\begin{proof}
Let $\pi$ be the permutation map on the indices of the columns of $X$ induced by multiplication by the permutation matrix $\Pi$ on the right. Then for all $(i,j) \in [r_1r_2]$, there exist $i',j'$ such that $\pi(i,j) = (i',j')$. 
Then 
$$
(U_1^* \otimes U_2^*)_{:,(i',j')} = (U_1^*)_{:,i'} \otimes (U^*_2)_{:,j'} = Y_{:,(i',j')} = X_{:,(i,j)} = (U_1^{\#})_{:,i} \otimes (U_2^{\#})_{:,j}.
$$ 
Using Theorem \ref{thm:uniquenesskronpdt}, this gives us that $(U_1^*)_{:,i'} = (U_1^{\#})_{:,i}$ and $(U^*_2)_{:,j'} = (U_2^{\#})_{:,j}$ for all $(i,j) \in [r_1r_2]$. 
If now there exists $i''\ne i'$ such that $\pi(i,k) = (i'',k')$ for some $k,k'$ then $(U_1^*)_{:,i} = (U_1^{\#})_{:,i'} = (U_1^{\#})_{:,i''}$ which is not possible since $U_1^{\#}$ has full column rank. As a consequence, $\pi$ maps $(i,k)$ to $(i',\cdot)$ independently from $k$, thus inducing a permutation on the indexes in the first position.
The same reasoning can be repeated for the indexes in the second position, and conclude that $\pi  = \pi_1\otimes \pi_2$ where $\pi(i,j) = (\pi_1(i), \pi_2(j))$. 

In terms of matrices, this can be written as $\Pi = \Pi_1\otimes \Pi_2$,  where 
 $\Pi_1 \in S_{r_1}$, $\Pi_2 \in S_{r_2}$ are permutation matrices and  $U_1^* = U_1^{\#} \Pi_1$,  $U_2^* = U_2^{\#} \Pi_2$. 
\end{proof}

\paragraph{All-at-once procedure via penalization} 

 Instead of the two-step approach described in Procedure 0, one can consider the following all-at-once procedure for computing the nTD via penalization: 
 for any $\lambda > 0$, solve the following min-vol order-2 nTD problem  
    \begin{align}\label{eq:unfoldord3alt}
    \min_{\mathcal G_{(3)}, U_{1,2}, U_1, U_2 , U_3}  & \big|\det\big(\mathcal G_{(3)}\big)\big| + \lambda \|U_{1,2} - U_1 \otimes  U_2\|_F^2  \\  \text{ such that }  & 
     \mathcal T_{(3)} = U_{1,2} \mathcal G_{(3)} U_3^\top,  
     U_i^\top e = e \text{ and } U_i \geq 0 \text{ for all } i \in [3], \nonumber 
    \end{align}
    to obtain an optimal solution $(\mathcal{G}^*_{(3)}, U^*_1, U^*_2, U^*_3)$. 
The following theorem shows that, under Assumption~\ref{ass:unfold}, an nTD obtained by solving~\eqref{eq:unfoldord3alt} is essentially unique.

\begin{theorem} \label{thm:unfoldsorder3alternate} 
Let $\mathcal T\in \R^{n_1 \times n_2 \times n_3}$ be an order-$3$ tensor satisfying 
Assumption~\ref{ass:unfold}, with the corresponding $(r_1,r_2,r_3)$-nTD given by  $(U^{\#}_1,U^{\#}_2,U^{\#}_3,\mathcal{G}^{\#})$ where $r_3 = r_1 r_2$.  
Then for any other decomposition $(U^{*}_1,U^{*}_2,U^{*}_3).\mathcal{G}^{*}$ obtained as an optimal solution of~\eqref{eq:unfoldord3alt},  
there exist permutation matrices $\Pi_1,\Pi_2,\Pi_3$ such that $U^{\#}_i = U^*_i \Pi_i$ for all $i \in [3]$ and $\mathcal{G}^{\#} = (\Pi_1^\top,\Pi_2^\top,\Pi_3^\top).\mathcal{G}^*$. 
\end{theorem} 
\begin{proof}
We first want to show that there exists a feasible solution to the optimization problem in (\ref{eq:unfoldord3alt}). Since $(U^{\#}_1,U^{\#}_2,U^{\#}_3).\mathcal{G}^{\#}$ is the nTD induced by Assumptions \ref{ass:unfold}, using Property \ref{lem:changeofbasisunfoldings}, we get that
\begin{equation}\label{eq:order3groundstate}
  \mathcal T_{(3)} = U_{1,2}^{\#} \mathcal{G}^{\#}_{(3)} (U_3^{\#})^\top \text{ where }  U_{1,2}^{\#} = U^{\#}_1 \otimes U^{\#}_2  .
\end{equation}
It follows that $(U^{\#}_1,U^{\#}_2,U^{\#}_3,\mathcal{G}_{(1,2)}^{\#})$ is indeed a feasible solution to the optimization problem in (\ref{eq:unfoldord3alt}) for all $\lambda$.  
Let $(U^*_1,U^*_2,U^*_3,\mathcal{G}_{(3)}^*)$ be an optimal solution to (\ref{eq:unfoldord3alt}). Then
\begin{equation}\label{eq:unfoldingeq3}
\begin{split}
    |\det\big(\mathcal{G}^*_{(3)}\big)| &\leq |\det\big(\mathcal{G}^*_{(3)}\big)| + \lambda \|U_{1,2}^* - U^*_1 \otimes U^*_2\|_F^2 \\
    &\leq |\det\big(\mathcal{G}^{\#}_{(3)}\big)| + \lambda \|U_{1,2}^{\#} - U^{\#}_1 \otimes U^{\#}_2\|_F^2 = |\det\big(\mathcal{G}^{\#}_{(3)}\big)|.   
\end{split}
\end{equation}
The last equality follows from the fact that $ U_{1,2}^{\#} = U^{\#}_1 \otimes U^{\#}_2$ as stated in (\ref{eq:order3groundstate}).
 Since by Assumption \ref{ass:unfold},  $U^{\#}_1\otimes U^{\#}_2$ and $U_3^\#$ satisfy the SSC, then $\rank(U^{\#}_1\otimes U^{\#}_2) = r_1r_2 = r_3 = \rank(U^{\#}_3)$. 
 This implies $\rank(\mathcal T_{(3)}) = r_3$,  and using Corollary~\ref{corr:suboptimalsoltn}, there exist permutation matrices $\Pi_{12}, \Pi_3$ such that 
\begin{equation}\label{eq:permutationkronpdt}
    U_{1,2}^* = U_{1,2}^{\#} \Pi_{12}, U_3^* = U_3^{\#} \Pi_3 \text{ and } G^*_{(3)} = (\Pi_{12})^\top G^{\#}_{(3)} \Pi_3.
\end{equation}
From this, it follows that $|\det\big(\mathcal{G}^*_{(3)}\big)| = |\det\big(G^{\#}_{(3)}\big)|$ and hence, using (\ref{eq:unfoldingeq3}), $U_{1,2}^* = U^*_1 \otimes U^*_2$. 

Since $ U_{1,2}^* = U^*_1 \otimes U^*_2 =  (U^{\#}_1 \otimes U^{\#}_2)\Pi_{12}$, using 
Proposition~\ref{prop:kronpdtperm}, there exist permutation matrices $\Pi_1, \Pi_2$ such that $U^*_1 = U_1^{\#} \Pi_1$ and $U^*_2 = U_2^{\#} \Pi_2$. This also gives us that $\mathcal{G}_{(1,2)}^* = (\Pi_1 \otimes \Pi_2)^\top \mathcal{G}_{(1,2)}^{\#} \Pi_3 $. Rewriting this in tensor form gives $\mathcal{G}^* = (\Pi_1^\top,\Pi_2^\top,\Pi_3^\top).\mathcal{G}^{\#}$.

\end{proof}


\subsection{Using slices} \label{sec:slices} 

In the previous, we provided an identifiability result for order-$3$ nTD using unfoldings. However, this result is not very satisfactory as it applies only when one rank of the order-$3$ nTD is equal to the product of the other two. In this section, we provide four more practical   identifiability results relying on slices of the input tensor. 
 For that, let us recall the following property. 

\begin{property}[Slices]\label{lem:changeofbasisslices}
Let $\mathcal T\in \mathbb{R}^{n_1 \times n_2 \times n_3}$ be an order-$3$ tensor with $\mathcal T = (U_1,U_2,U_3).\mathcal G$ where $\mathcal G \in \mathbb{R}^{r_1 \times r_2 \times r_3}$.  
The $j$th slice of $\mathcal T$ along the third mode is given by
$$
\mathcal T^{(3)}_j = U_1\bigg(\sum_{k \in [r_3]}(U_3)_{j,k}\mathcal G^{(3)}_k\bigg)U_2^{\top}, 
$$
where $\mathcal G^{(3)}_k$ is the $k$th slice of $\mathcal G$ along the third mode. 
By symmetry, the same property applies to the slices along the first and second modes.  
\end{property}

\subsubsection{Approach 1: two modes of the input tensor have one slice with maximum rank} \label{sec:Idea1}

In this section, we give identifiability results for order-$3$ nTD under the assumption that there exist two modes of the given tensor that have one slice with maximum possible rank. 
We first state the assumptions explicitly and then in Remark \ref{rmk:maxpossiblernk} we explain why these slices have maximum rank. 

\begin{assumptions}\label{ass:order3idea1}
The order-$3$ tensor $\mathcal T\in \R^{n_1 \times n_2 \times n_3}$ satisfies Assumption~\ref{ass:generalass} and  
\begin{enumerate}
    \item $r_3 \leq r = r_1 = r_2$. 
    
    \item $U^{\#}_i$ satisfies the SSC for $i \in [3]$. 
    
    \item There exists $i_2 \in [n_2]$ and $i_3 \in [n_3]$ such that  $\rank(\mathcal T^{(3)}_{i_3}) = r$ and $\rank(\mathcal T^{(2)}_{i_2})  = r_3$.
\end{enumerate}
\end{assumptions}

Note that, for simplicity of the presentation, this assumption arbitrarily chooses modes $2$ and $3$ to have at least one slice of maximum rank. Of course, by symmetry of the problem, any combination of two modes can be chosen. 

\begin{remark}[Why maximal rank?] \label{rmk:maxpossiblernk}
Since $\mathcal T = (U^{\#}_1,U^{\#}_2,U^{\#}_3).\mathcal{G}^{\#}$, using Property \ref{lem:changeofbasisslices}, any slice $\mathcal T^{(3)}_i $ has a decomposition of the form $\mathcal T^{(3)}_i = U_1 S_i U_2^\top$ for some $r$-by-$r$ matrix $S_i$, and hence $\rank(\mathcal T^{(3)}_i) \leq r$. 
Therefore Assumption~\ref{ass:order3idea1} requires that there exists $i_3 \in [n_3]$ such that $\mathcal T^{(3)}_{i_3}$ attains the maximum possible rank, $r$. The same reasoning applies to $\mathcal T^{(2)}_{i_2}$ for some $i_2 \in [n_2]$. This justifies the choice of the title for this section.  
\end{remark}

By Property~\ref{lem:changeofbasisslices}, $\mathcal T^{(3)}_j = U_1\left(\sum_{k \in [r_3]}(U_3)_{j,k}\mathcal G^{(3)}_k\right)U_2^{\top}$, and, under Assumption~\ref{ass:order3idea1}, we can recover, up to permutations, $U_1$ and $U_2$ using order-2 nTDs (Theorem~\ref{thm:nmfmain}). 
It remains to recover $U_3$ and $\mathcal G$, which can be done via an identifiable min-vol NMF formulation (Theorem~\ref{thm:idenminvol}).  
Let us describe this two-step procedure:

\vspace{0.1cm} 

\fbox{%
	\parbox{0.95\linewidth}{%
    \begin{center} \vspace{-0.2cm}   
       \textbf{Procedure 1: Unique order-$3$ nTD under Assumption~\ref{ass:order3idea1}} 
    \end{center}
\begin{enumerate}
    \item Computation of $U^*_1, U^*_2$: Solve the following min-vol order-2 nTD problem  
    \begin{equation}\label{eq:idea1tensororder3}
    \min_{U_1,U_2,S_{i_3}}  |\det(S_{i_3})| \text{ such that } 
     \mathcal T_{i_3}^{(3)} = U_1S_{i_3}U_2^\top,  
     U_i^\top e = e \text{ and } U_i \geq 0 \text{ for } i \in [2], 
    \end{equation}
    to obtain an optimal solution $ (U^*_1,U^*_2,S^*_{i_3})$. 

    \item Computation of $U^*_3$: Form the matrix $T^*_{i_2} = (U^*_1)^{\dagger} \mathcal T_{i_2}^{(2)}$ where $U^\dagger$ denotes the Moore-Penrose inverse of $U$. Then solve the following min-vol NMF problem  \begin{equation}\label{eq:idea1tensororder3b}
    \min_{U_3, S_{i_2}}  \det(S_{i_2}^\top S_{i_2}) \quad \text{ such that }  \quad 
     T^*_{i_2} = S_{i_2}U_3^\top,  
     U_3^\top e = e \text{ and } U_3 \geq 0, 
\end{equation} 
to obtain an optimal solution $(U^{*}_3, S_{i_2}^*)$. 


    \item Computation of $\mathcal{G}^* = ((U^{*}_1)^{\dagger},(U^{*}_2)^{\dagger},(U^{*}_3)^{\dagger}).T$.   \vspace{-0.1cm}  
\end{enumerate}
	}%
}


\vspace{0.1cm} 

Theorem~\ref{thm:idea1order3} shows show that, under Assumption~\ref{ass:order3idea1}, a decomposition computed by this procedure is essentially unique. 
\begin{theorem}\label{thm:idea1order3} 
Let $\mathcal T\in \R^{n_1 \times n_2 \times n_3}$ be an order-$3$ tensor satisfying 
Assumption~\ref{ass:order3idea1}, with the corresponding $(r,r,r_3)$-nTD given by  $(U^{\#}_1,U^{\#}_2,U^{\#}_3,\mathcal{G}^{\#})$. 
Then for any other decomposition $(U^{*}_1,U^{*}_2,U^{*}_3).\mathcal{G}^{*}$ obtained with Procedure~1, 
there exist permutation matrices $\Pi_1,\Pi_2,\Pi_3$ such that $U^{*}_i = U^{\#}_i \Pi_i$ for all $i \in [3]$ and $\mathcal{G}^{*} = (\Pi_1^\top,\Pi_2^\top,\Pi_3^\top).\mathcal{G}^{\#}$. 
\end{theorem}
\begin{proof}
Let us first show that the nTD $(U^{\#}_1,U^{\#}_2,U^{\#}_3,\mathcal{G}^{\#})$ of $\mathcal T$ provides a feasible solution to \eqref{eq:idea1tensororder3} and \eqref{eq:idea1tensororder3b}. 
By Assumption~\ref{ass:order3idea1}, the slices $\mathcal T^{(2)}_{i_2}$ and $\mathcal T^{(3)}_{i_3}$ have rank $r$ and $r_3 \leq r$, respectively. 
Using Property~\ref{lem:changeofbasisslices}, we have  
\begin{equation}\label{eq:changeofslices}
        \mathcal T^{(3)}_{i_3} = U^{\#}_1 D^{\#}_{i_3} (U^{\#}_2)^{\top}, 
\end{equation} 
where $S^{\#}_{i_3} = \sum_{k \in [r_3]}(U^{\#}_3)_{i_3,k}\mathcal{G}^{\#}_k$ and $\mathcal{G}^{\#}_k$ is the $k$th slice along the third mode of $\mathcal{G}^{\#}$. 
This shows that $(U^{\#}_1,D^{\#}_{i_3},(U^{\#}_2)^{\top})$ is a feasible solution of (\ref{eq:idea1tensororder3}), the first step of Procedure 1. 
Moreover, using Property~\ref{lem:changeofbasisslices}, it also follows that $\mathcal T^{(2)}_{i_2} = U^{\#}_1D^{\#}_{i_2}(U^{\#}_3)^{\top}$. 
Since $U^{\#}_1$ satisfies the SSC, $\rank(U^{\#}_1) = r$. Hence, $(U_1^{\#})^{\dagger}\mathcal T^{(2)}_{i_2} = D^{\#}_{i_2}(U^{\#}_3)^{\top}$ which shows that $( D^{\#}_{i_2},(U^{\#}_3)^{\top})$ is a feasible solution of (\ref{eq:idea1tensororder3b}), the second step of Procedure 1. 

Let us now show that Procedure 1 recovers an \textit{essentially unique} nTD. 
Let $(U^{*}_1,U^{*}_2,S^{*}_{i_3})$ be an optimal solution of~(\ref{eq:idea1tensororder3}). 
Since $U^{\#}_1$ and $U^{\#}_2$ satisfy the SSC and $\rank(\mathcal T^{(3)}_{i_3}) = r$ by Assumption~\ref{ass:order3idea1},  
Theorem~\ref{thm:nmfmain} applies: there exist permutation matrices $\Pi_1, \Pi_2$ such that $U_1^{*} = U_1^{\#} \Pi_1$, $U_2^{*} = U_2^{\#}\Pi_2$ and $S^{*}_{i_3} = \Pi_1^\top S^{\#}_{i_3} \Pi_2$.  
Let $\mathcal T_{i_2}^{\#} = (U_1^{\#})^{\dagger} \mathcal T^{(2)}_{i_2}$ and $\mathcal T_{i_2}^{*} = (U_1^{*})^{\dagger} \mathcal T^{(2)}_{i_2}$. Since $U_1^{*} = U_1^{\#} \Pi_1$, 
    $\mathcal T_{i_2}^{*} = \Pi_1^\top \mathcal T_{i_2}^{\#}$.
Hence the problems (\ref{eq:idea1tensororder3b}) constructed from $U_1^{\#}$ or from $U_1^{*}$ are the same, up to permutations. 
By Assumption~\ref{ass:order3idea1}, $U^{\#}_3$ satisfies the SSC, while $\rank(\mathcal T_{i_2}^*) = \rank(\mathcal T_{i_2}^{(2)}) = r_3$ because $U_1^{*} = U_1^{\#} \Pi_1$ satisfies the SSC and hence is full rank. Therefore Theorem~\ref{thm:nmfmain} applies to~(\ref{eq:idea1tensororder3b}): $U_3^{*} = U_3^{\#} \Pi_3$ for some permutation matrix $\Pi_3$. 

Moreover, since $U^{\#}_i$ have full column rank, we have that $(U^{\#}_i)^{\dagger}U^{\#}_i = I_{r_i}$ for all $i \in [3]$. Using this, we already have that $\mathcal{G}^{\#} = ((U^{\#}_1)^{\dagger},(U^{\#}_2)^{\dagger},(U^{\#}_3)^{\dagger}).T$. From this, we can conclude that $$\mathcal{G}^* = ((U^*_1)^{\dagger},(U^*_2)^{\dagger},(U^*_3)^{\dagger}).\mathcal{T} =  (\Pi_1^{\top},\Pi_2^\top,\Pi_3^\top).((U^{\#}_1)^{\dagger},(U^{\#}_2)^{\dagger},(U^{\#}_3)^{\dagger}).T = (\Pi_1^{\top},\Pi_2^\top,\Pi_3^\top).\mathcal{G}^{\#}.$$
Finally, since all factors involved have full column rank, $\mathcal{G}^*$ computed by the third step of Procedure~1 gives the desired result: 
$$
\mathcal T = (U^{\#}_1,U^{\#}_2,U^{\#}_3).\mathcal{G}^{\#} = (U^{*}_1\Pi^\top_1,U^{*}_2\Pi^\top_2,U^{*}_3\Pi^\top_3).\Big((\Pi_1, \Pi_2, \Pi_3)\mathcal{G}^*\Big) = (U^{*}_1,U^{*}_2,U^{*}_3).\mathcal{G}^*.
$$ 
\end{proof}




\subsubsection{Approach 2: the spans of slices of the core tensor along two modes have maximum possible rank} \label{sec:randomidea2}

In the previous section, we used a single maximum-rank slice of $\mathcal T$. In this section, we show how to combine slices to obtain a stronger identifiability result. We need the following definition. 
\begin{definition}[Maximal Rank of a Matrix Space]\label{def:nonsingular}
We define a subspace of matrices $\mathcal{V} \subseteq \R_+^{r \times r}$ to have \textit{maximal rank} $r$ if there exists $A \in \mathcal{V}$ such that $\rank(A) = r$.
\end{definition}
For any family of matrices $\mathcal{A} = \{A_1,\dots,A_n\}$, we define their \textit{linear span} over a field $\mathbb{K}$ as $\text{lin-span}_{\mathbb{K}}(\mathcal{A}) := \{\sum_{i=1}^n v_iA_i \ | \ v_i \in \mathbb{K}\}$.

\begin{assumptions}\label{ass:order3idea2} 
The order-$3$ tensor $\mathcal T\in \R^{n_1 \times n_2 \times n_3}$ satisfies Assumption~\ref{ass:generalass}, Assumption~\ref{ass:order3idea1}(1-2) (namely $r_3 \leq r = r_1 = r_2$, and $U^{\#}_i$ satisfies the SSC for $i \in [3]$), and   
\begin{enumerate}

    
    


    
    
    \item The set $\mathcal{V}^{(3)} := \text{lin-span}_{\R_+} 
    \big\{\mathcal G^{\#(3)}_1,\dots,\mathcal G^{\#(3)}_{r_3}\big\}$ has maximal rank $r$ 
    and the set $\mathcal{V}^{(2)} := \text{lin-span}_{\R_+} \big\{\mathcal G^{\#(2)}_1,\dots,\mathcal G^{\#(2)}_{r}\big\}$  has maximal rank $r_3$.
\end{enumerate}
\end{assumptions}

Assumption~\ref{ass:order3idea2} is a relaxation of Assumption~\ref{ass:order3idea1}. 
In fact, Assumption~\ref{ass:order3idea1} requires $\rank(\mathcal T^{(3)}_j) =  r$ for some $j$. 
Using Property~\ref{lem:changeofbasisslices}, 
\[
\mathcal T^{(3)}_j = U_1^\# \underbrace{\Big(\sum_{k \in [r_3]}(U_3^\#)_{j,k}\mathcal 
G^{\#(3)}_k\Big)}_{=: S} U_2^{\# \top} 
\] 
so that $S \in \mathcal{V}^{(3)}$ has rank $r$. Similarly, one can show that $\mathcal{V}^{(2)}$ has maximal rank $r_3$. 

Given an order-$3$ tensor $\mathcal T$ that satisfies Assumption~\ref{ass:order3idea2}, let us consider the following randomized procedure for computing a decomposition: 

\vspace{0.1cm}   
\fbox{%
	\parbox{0.95\linewidth}{%
    \begin{center} \vspace{-0.2cm}   
       \textbf{Procedure 2: Unique order-$3$ nTD under Assumption~\ref{ass:order3idea2}} 
    \end{center}
\begin{enumerate}

    \item Let $\mathcal T_{\alpha} = \sum_{i=1}^{n_3} \alpha_i\mathcal T^{(3)}_i$ and $\mathcal T_{\beta} = \sum_{i=1}^{n_2} \beta_i\mathcal T^{(2)}_i$ be two random linear combinations from a continuous distribution of the slices along the third and second mode of the tensor $\mathcal T$, respectively.

    \item Computation of $U^*_1, U^*_2$: Solve the following min-vol order-2 nTD problem   
\begin{equation}\label{eq:idea2order3}
    \min_{U_1,U_2,S_{\alpha}}  |\det(S_{\alpha})| 
    \text{ such that } 
    \mathcal T_{\alpha} = U_1 S_{\alpha} U_2^\top, 
    U_i^\top e = e \text{ and } U_i \geq 0 \text{ for } i \in [2], 
    \end{equation}
 to obtain an optimal solution $(U^*_1,U^*_2,S^*_{i_3})$. 

 \item Computation of $U^*_3$: For the matrix $T^*_{\beta} = (U^*_1)^{\dagger}\mathcal T_{\beta}$ and solve the following min-vol NMF problem 
    \begin{equation}\label{eq:idea2order3b}
    \min_{U_3, S_{\beta}}  \det(S_{\beta}^\top S_{\beta})   \text{ such that }  
     T^*_{\beta} = S_{\beta}U_3^\top, 
     U_3^\top e = e \text{ and } U_3 \geq 0,  
       \end{equation} 
to obtain an optimal solution $(U_3^*, S^*_{\beta})$.

    
    \item Compute $\mathcal{G}^* = ((U^*_1)^{\dagger},(U^*_2)^{\dagger},(U^*_3)^{\dagger}).T$.
\end{enumerate}
}
}
\vspace{0.1cm}

Theorem~\ref{thm:idea2order3} shows that, under Assumption~\ref{ass:order3idea2}, a decomposition computed by this procedure is essentially unique.

\begin{theorem}\label{thm:idea2order3}  
Let $\mathcal T\in \R^{n_1 \times n_2 \times n_3}$ be an order-$3$ tensor satisfying 
Assumption~\ref{ass:order3idea2}, with the corresponding  $(r,r,r_3)$-nTD given by $(U^{\#}_1,U^{\#}_2,U^{\#}_3).\mathcal{G}^{\#}$. 
Then, with probability~$1$, for any decomposition $(U^{*}_1,U^{*}_2,U^{*}_3).\mathcal{G}^{*}$ obtained from Procedure~2, there exist permutation matrices $\Pi_1,\Pi_2,\Pi_3$ such that $U^{*}_i = U^{\#}_i \Pi_i$ for all $i \in [3]$ and $\mathcal{G}^{*} = (\Pi_1^\top,\Pi_2^\top,\Pi_3^\top).\mathcal{G}^{\#}$.
\end{theorem}
\begin{proof}
Let us first show that the nTD $(U^{\#}_1,U^{\#}_2,U^{\#}_3,\mathcal{G}^{\#})$ of $\mathcal T$ provides a feasible solution to \eqref{eq:idea2order3} and \eqref{eq:idea2order3b} with probability~$1$.  
Let $\mathcal T^{(i)}_j$ be the $j$th slice of $\mathcal T$ along mode $i$. From Property~\ref{lem:changeofbasisslices}, we have  
\begin{align*}
    \mathcal T^{(3)}_j = U^{\#}_1 S^{(3)}_j (U^{\#}_2)^{\top}, 
\end{align*}
where $S^{(3)}_j = \sum_{k \in [r_3]}(U^{\#}_3)_{j,k}\mathcal{G}^{\#}_k$ for all $j \in [n_3]$ and $\mathcal{G}^{\#}_k$ is the $k$th slices along the third mode of $\mathcal{G}^{\#}$. Using this, we obtain 
\begin{equation}\label{eq:randomlincombtrifactorization}
    \mathcal T_{\alpha} = U^{\#}_1 S_{\alpha}(U^{\#}_2)^{\top}, 
\end{equation}
where $$
S_{\alpha} = \sum_{j \in [r]} \alpha_j S^{(3)}_j =  \sum_{k \in [r]} \sum_{j \in [n_3]} \alpha_j (U^{\#}_3)_{j,k}\mathcal{G}^{\#}_k = \sum_{k \in [r]} \langle \alpha, (U^{\#}_3)_{\cdot,k}\rangle \mathcal{G}^{\#}_k. 
$$
Define the $n$-variate polynomial $$P(x) = \det\Big(\sum_{k \in [r]} \langle x, (U^{\#}_3)_{\cdot,k}\rangle \mathcal{G}^{\#}_k\Big) =  \det\Big(\sum_{k \in [r]} ((U^{\#}_3))^{\top} x)_k \mathcal{G}^{\#}_k\Big).
$$ 
Let us show that $P(x) \not\equiv 0$. 
Since, $\mathcal{V}^{(3)}$ is a non-singular subspace, there exists a vector $t = (t_1,\dots,t_r)$ such that $\det(\sum_{i \in [r]} t_i\mathcal{G}^{\#}_i) \neq 0$. Moreover, since $U^{\#}_3 \in \R^{n_3 \times r}$ satisfies the SSC, 
it has full column rank. Hence, $(U^{\#}_3)^{\dagger} U^{\#}_3 = I_r$. 
Taking $x_0 = (((U^{\#}_3)^{\dagger})^{\dagger})^{\top}t$, we obtain 
\begin{align*}
    P(x_0)\text{$=$}\det\Big(\sum_{k \in [r]} (((U^{\#}_3)^{\dagger})^{\top} x_0)_k \mathcal{G}^{\#}_k\Big)\text{$=$}\det \Big(\sum_{k \in [r]} ((U^{\#}_3))^{\top}  ((U^{\#}_3)^{\dagger})^{\top}t)_k \mathcal{G}^{\#}_k \Big)\text{$=$}\det \Big(\sum_{k \in [r]} t_k\mathcal{G}^{\#}_k\Big)\text{$\neq$}0.
\end{align*}
Hence, if $\alpha$ is picked at random, $\rank(S_{\alpha}) = r$ with probability~$1$. 
Moreover, since, $U_1,U_2$ satisfy the SSC, $\rank(U^{\#}_1) = \rank(U^{\#}_2) = r$. This gives us that $\rank(\mathcal T_{\alpha}) = r$ with probability~$1$.
Hence, using (\ref{eq:randomlincombtrifactorization}), $(U^{\#}_1,S_{\alpha},(U^{\#}_2)^\top)$ is a feasible solution to~\eqref{eq:idea2order3}.  
\par
Since $\mathcal{V}^{(2)}$ is a matrix subspace with maximal rank $r_3$, a similar argument gives us that $\rank(S_{\beta}) = r_3$ with probability~$1$.  Moreover, using a similar argument as (\ref{eq:randomlincombtrifactorization}), we also get that $\mathcal T_{\beta} =U^{\#}_1 S_{\beta}(U^{\#}_3)^{\top}$ where $S_{\beta} = \sum_{k \in [r]} \langle \beta, (U^{\#}_2)_{\cdot,k} \rangle \mathcal{G}^{\#}_k$. Since, $U^{\#}_1$ and $U^{\#}_3$ satisfy the SSC, $\rank(U^{\#}_1) = r$ and $\rank(U^{\#}_3) = r_3$. We conclude that $(S_{\beta},(U^{\#}_3)^{\top})$ is a solution of~(\ref{eq:idea2order3}).

Now following the same proof as that of Theorem \ref{thm:idea1order3}, it follows that the $(r,r,r_3)$-nTD returned by the Procedure~2 is \textit{essentially unique} with probability~$1$. 
\end{proof}


\subsubsection{Approach 3: One slice along one mode has full rank and the unfolding of the core tensor along that mode has full column rank}\label{sec:Idea3}

In the previous two sections, we have focused on the case where some information related to slices along two modes is available. What if we do not have information related to slices along two modes, but just along one? Without loss of generality, in the rest of this section, we assume that the information is available along the third mode. The same results hold with appropriate modifications for the first and second modes as well.
\begin{assumptions}\label{ass:order3idea3}
The order-$3$ tensor $\mathcal T\in \R^{n_1 \times n_2 \times n_3}$ satisfies Assumption~\ref{ass:generalass} and   
 \begin{itemize}
    \item $\sqrt{r_3} \leq r = r_1 = r_2$. 
    
    \item  $U^{\#}_i$ satisfies the SSC  for all $i \in [3]$. 
    
    \item There exists $i^* \in [n_3]$ such that $\rank\Big(\mathcal T^{(3)}_{i^*}\Big) = r$. 
    
    \item $\rank\Big(\mathcal{G}^{\#}_{(3)}\Big) =~r_3$.

\end{itemize}
\end{assumptions}

As for Procedure~1, we can recover $U_1^{\#}$ and $U_2^{\#}$, up to permutations, using min-vol nTD, since $\mathcal T^{(3)}_{i^*} =  U_1^{\#} S U_2^{\#}$ for some full-rank square matrix $S$. 
To recover $U_3^{\#}$ and $\mathcal G^{\#}$, we solve another min-vol NMF problem for the projection of all the slices $\mathcal T_i^{(3)}$ using $U_1^{\#}$ and $U_2^{\#}$, as described in the procedure below. 
Note that here we use all the slices of $\mathcal T^{(3)}_{i}$, while only two were used in Procedure~1. 

\vspace{0.1cm}   
\fbox{%
	\parbox{0.95\linewidth}{%
    \begin{center} \vspace{-0.2cm}   
       \textbf{Procedure 3: Unique order-$3$ nTD under Assumption~\ref{ass:order3idea3}} 
    \end{center}
\begin{enumerate}

\item Computation of $U_1^*,U_2^*$: Solve the following min-vol order-2 nTD problem  
\begin{equation}\label{eq:idea3tensororder3}
    \min_{U_1, U_2, S_{i^*}}  |\det(S_{i^*})| 
    \text{ such that }  
     \mathcal T_{i^*}^{(3)} = U_1 S_{i^*} U_2^\top,  
      U_i^\top e =e \text{ and } U_i \geq 0 \text{ for } i \in [2] 
\end{equation} 
     to obtain an optimal solution $(U_1^*,U_2^*,S_j^*)$. 

   \item   Computation of $\mathcal{G}_{(3)}^*, U_3^*$: Form the matrix 
     $S^*$ whose $i$th column is given by $\vect((U^*_1)^\dagger \mathcal T_{i}^{(3)} ((U^*_2)^\top)^{\dagger})$ for all $i \in [n_3]$ and solve the min-vol NMF problem 
    \begin{equation}\label{eq:idea3tensororder3b} 
    \min_{\mathcal{G}_{(3)},U_3} \det(\mathcal{G}_{(3)}^\top \mathcal{G}_{(3)}) 
     \text{ such that } S^* = \mathcal{G}_{(3)}U_3^\top,  
    U_3^\top e = e \text{ and }U_3 \geq 0, 
\end{equation}
to obtain an optimal solution $(\mathcal{G}_{(3)}^*, U_3^*)$. 
\vspace{-0.1cm}  
\end{enumerate}
}
}
\vspace{0.1cm} 


Theorem~\ref{thm:idea3order3} shows that, under Assumption~\ref{ass:order3idea3}, a decomposition computed by this procedure is essentially unique with probability~$1$.

\begin{theorem}\label{thm:idea3order3} 
Let $\mathcal T\in \R^{n_1 \times n_2 \times n_3}$ be an order-$3$ tensor satisfying 
Assumption~\ref{ass:order3idea3}, with the corresponding $(r,r,r_3)$-nTD given by  $(U^{\#}_1,U^{\#}_2,U^{\#}_3,\mathcal{G}^{\#})$. 
Then for any other decomposition $(U^{*}_1,U^{*}_2,U^{*}_3).\mathcal{G}^{*}$ obtained with Procedure~3, 
there exist permutation matrices $\Pi_1,\Pi_2,\Pi_3$ such that $U^{*}_i = U^{\#}_i \Pi_i$ for all $i \in [3]$ and $\mathcal{G}^{*} = (\Pi_1^\top,\Pi_2^\top,\Pi_3^\top).\mathcal{G}^{\#}$. 
\end{theorem}
\begin{proof}
Let us first show that the existence of an nTD $(U^{\#}_1,U^{\#}_2,U^{\#}_3).\mathcal{G}^{\#}$ of the input tensor $\mathcal T$ that satisfies Assumptions~\ref{ass:order3idea3} also proves the existence of a solution to the optimization problems in  (\ref{eq:idea3tensororder3}) and (\ref{eq:idea3tensororder3b}). 
Using Property~\ref{lem:changeofbasisslices}, the slices along the third mode have the form 
$$\mathcal T^{(3)}_i = U^{\#}_1S^{\#}_i(U^{\#}_2)^{\top}, \quad 
\text{ where } S^{\#}_{i} = \sum_{k \in [r_3]}(U^{\#}_3)_{i,k}\mathcal{G}^{(3)}_k. 
$$ 
where $\mathcal{G}^{(3)}_k$ is the $k$th slice of $\mathcal{G}^{\#}$ along the third mode. 
 Hence $(U^{\#}_1,S^{\#}_{i^*},U^{\#}_2)$ is a feasible solution to (\ref{eq:idea3tensororder3}). 
Let $S^{\#}$ be the $r_1r_2 \times n_3$ matrix such that the $i$th column of $S^{\#}$ is given by 
$$
\vect((U^{\#}_1)^\dagger \mathcal T_{i}^{(3)} ((U^{\#}_2)^\top)^{\dagger}) =  \vect(S^{\#}_{i}) \quad \text{ for all } i \in [n_3]. 
$$ 
Then for all $j \in [r_1]$ and  $k \in [r_2]$, 
\begin{equation}\label{eq:lincombslicestounfolding}
    (S^{\#})_{jk,i} = (S_i^{\#})_{j,k} = \sum_{m \in [r_3]}(U^{\#}_3)_{i,m}(\mathcal{G}^{(3)}_m)_{j,k} =  \sum_{m \in [r_3]}(\mathcal{G}^{\#}_{(3)})_{jk,m}(U^{\#}_3)^{\top}_{m,i} = (\mathcal{G}^{\#}_{(3)}.(U^{\#}_3)^\top)_{jk,i}. 
\end{equation}

Let us now show that the solution returned by the Procedure~3 is essentially unique. 
Let $(U^{*}_1,U^{*}_2,S_{i^*}^{*})$ be an optimal solution to (\ref{eq:idea3tensororder3}). 
The assumptions of Theorem~\ref{thm:nmfmain} are satisfied, by Assumption~\ref{ass:order3idea3}, hence there exist permutation matrices $\Pi_1,\Pi_2$ such that $ U_1^* = U^{\#}_1  \Pi_1$, $ U_2^* = U^{\#}_2 \Pi_2 $ and $S_{i^*}^{*} = \Pi_1^\top S_{i^*}^{\#} \Pi_2$.

Let $S^{\#}$ be the matrix with $i$th column given by $\vect((U^{\#}_1)^\dagger \mathcal T_{i}^{(3)} ((U^{\#}_2)^\top)^{\dagger})$ and $S^{*}$ be the matrix with $j$th column given by $\vect((U^{*}_1)^\dagger \mathcal T_{j}^{(3)} ((U^{*}_2)^\top)^{\dagger})$ for all $j \in [n_3]$.  We have the following relation between their $i$th columns
\begin{equation}\label{eq:kronpermutation}
\begin{split}
    \vect(S^{*}_i) 
    = \vect((U^{*}_1)^{\dagger} \mathcal T^{(3)}_{i} ((U^{*}_2)^\top)^{\dagger})
    & = \vect(\Pi_1^\top (U^{\#}_1)^\dagger \mathcal T^{(3)}_{i} ((U^{\#}_2)^\top)^\dagger \Pi_2) \\  
    & = \vect(\Pi_1^\top S^{\#}_i \Pi_2) 
    = (\Pi_1 \otimes \Pi_2)^\top \vect(S^{\#}_i). 
\end{split}
\end{equation}
In matrix form, this gives $S^{*} = (\Pi_1 \otimes \Pi_2)^\top S^{\#}$. 
This shows that $S^*$ is equal to $S^{\#}$, up to a permutation of the rows. 
Finally, using Theorem~\ref{thm:idenminvol} on~\eqref{eq:idea3tensororder3b}: there exists a permutation matrix $\Pi_3$ such that 
$\mathcal{G}^{*}_{(3)}=(\Pi_1^\top\otimes\Pi_2^\top)\mathcal{G}^{\#}_{(3)} \Pi_3$ 
and 
$U^{*}_3 = U^{\#}_3 \Pi_3$, 
which gives us the desired result. 

\end{proof}



\subsubsection{Approach 4: The span of slices along one mode has maximum rank and the unfolding of the core tensor along that mode has full column rank}

Similarly as done between Approaches 1 and 2, we can relax the maximum-rank assumption on one slice of $\mathcal T$ along the third mode to their span.  
\begin{assumptions}\label{ass:order3idea4}
The order-$3$ tensor $\mathcal T\in \R^{n_1 \times n_2 \times n_3}$ satisfies Assumption~\ref{ass:generalass}, 
Assumption~\ref{ass:order3idea3}(1,2,4) (namely $\sqrt{r_3} \leq r = r_1 = r_2$, $U^{\#}_i$ satisfies the SSC  for $i \in [3]$, $\rank\Big(\mathcal{G}^{\#}_{(3)}\Big) =~r_3$), 
and   
   $\mathcal{V}^{(3)} = \text{lin-span}_{\R_+}\{\mathcal{G}_1^{\# (3)},\dots,\mathcal{G}_r^{\# (3)}\}$ has maximal rank $r$.  
\end{assumptions}
Given an order-$3$ tensor $\mathcal T$ that satisfies Assumption~\ref{ass:order3idea4}, let us consider the following procedure for computing a decomposition: 
\vspace{0.1cm}  

\fbox{%
	\parbox{0.95\linewidth}{%
    \begin{center} \vspace{-0.2cm}   
       \textbf{Procedure 4: Unique order-$3$ nTD under Assumption~\ref{ass:order3idea4}} 
    \end{center}
\begin{enumerate}

\item Pick $\alpha^{(i)} = \alpha^{(i)}_1,\dots,\alpha^{(i)}_{n_3}$ independently at random from a continuous distribution for all $i \in [n_3]$ and define the matrix $A \in \R^{n_3 \times n_3}$ whose columns are give by the $\alpha^{(i)}$'s, that is, $A_{\dot,i} = \alpha^{(i)}$ for $i \in [n_3]$. 
Let $\mathcal T_{\alpha^{(i)}} = \sum_{j=1}^{n_3} \alpha^{(i)}_j\mathcal T^{(3)}_j$ be $n_3$ random linear combination of the slices along the third mode of the tensor $\mathcal T$. 

\item Computation of $U_1^*,U_2^*$: Solve the following min-vol order-2 nTD problem   
\begin{equation}\label{eq:idea4tensororder3}
 \min_{U_1,U_2,S_1} |\det(S_1)|  \text{ such that } 
    \mathcal T_{\alpha^{(1)}} = U_1S_1 U_2^\top \text{, } 
    U_i^\top e = e \text{ and } U_i \geq 0 \text{ for } i \in [2],  
    \end{equation} 
     to obtain an optimal solution $(U_1^*,U_2^*,S^*_{\alpha})$.

    \item Computation of $\mathcal{G}_{(3)}^*,U_3^*$:  Let $S^*$ be the matrix with $i$th column given by $\vect((U^*_1)^\dagger \mathcal T_{\alpha^{(i)}} ((U^*_2)^\top)^{\dagger})$, and solve the min-vol NMF problem 
    \begin{equation}\label{eq:idea4tensororder3b}
    \min_{\mathcal{G}_{(3)},U_3}  \det(\mathcal{G}_{(3)}^\top \mathcal{G}_{(3)}) \text{ such that } S' = S^* A^{-1} = \mathcal{G}_{(3)}U_3^\top \text{ and } U_3^\top e = e, U_3 \geq 0. 
\end{equation}  
to obtain an optimal solution $(\mathcal{G}_{(3)}^*,U_3^*)$. 
\vspace{-0.1cm}  
\end{enumerate}
}
}
\vspace{0.1cm}


Theorem~\ref{thm:idea4order3} shows show that, under Assumption~\ref{ass:order3idea4}, a decomposition computed by this procedure is essentially unique,  with probability~$1$.  

\begin{theorem} \label{thm:idea4order3} 
Let $\mathcal T\in \R^{n_1 \times n_2 \times n_3}$ be an order-$3$ tensor satisfying 
Assumption~\ref{ass:order3idea4}, with the corresponding $(r,r,r_3)$-nTD given by  $(U^{\#}_1,U^{\#}_2,U^{\#}_3,\mathcal{G}^{\#})$. 
Then for any other decomposition $(U^{*}_1,U^{*}_2,U^{*}_3).\mathcal{G}^{*}$ obtained with Procedure~4, 
there exist permutation matrices $\Pi_1,\Pi_2,\Pi_3$ such that $U^{\#}_i = U^*_i \Pi_i$ for all $i \in [3]$ and $\mathcal{G}^{\#} = (\Pi_1^\top,\Pi_2^\top,\Pi_3^\top).\mathcal{G}^*$, with probability~$1$.  
\end{theorem}
\begin{proof}
Let us first show that the existence of an nTD $(U^{\#}_1,U^{\#}_2,U^{\#}_3).\mathcal{G}^{\#}$ of $\mathcal T$ that satisfies Assumptions \ref{ass:order3idea4} provides a feasible solution to ~\eqref{eq:idea4tensororder3} and \eqref{eq:idea4tensororder3b} with probability $1$. 
As done in (\ref{eq:randomlincombtrifactorization}), we have
\begin{equation}\label{eq:talphai}
    \mathcal T_{\alpha^{(i)}} = U^{\#}_1 S^{\#}_{i} (U^{\#}_2)^\top, \; \; \text{ where } S^{\#}_{i} = \sum_{k \in [r]}  \langle \alpha^{(i)}, (U^{\#}_3)_{\cdot, k} \rangle \mathcal{G}^{(3)}_k \text{ for all } i \in [n_3]. 
\end{equation}
Since $\mathcal{V}^{(3)}$ has maximal rank $r$, 
following the proof of Theorem~\ref{thm:idea2order3}, $\rank(\mathcal T_{\alpha_1}) = r$ with probability 1.


Since $S^{\#}$ is the matrix with its $i$-th column defined as $\vect((U^{\#}_1)^{\dagger}\mathcal T_{\alpha^{(i)}}((U^{\#}_2)^\top)^{\dagger})$, using \eqref{eq:talphai}, we get that the $i$-th column of $S^{\#}$ is given by $\vect(S_i^{\#})_{p,q} =  \sum_{k \in [r]}((U^{\#}_3)^{\top} \alpha_t)_k(\vect(\mathcal G^{(3)}_k))_{p,q}$. Rewriting this in matrix form gives us $vec(S^{\#}_i) =  \mathcal{G}^{\#}_{(3)} (U^{\#}_3)^{\top} \alpha^{(i)}$ where $\mathcal{G}^{\#}_{(3)}$ is the unfolding of the core tensor $\mathcal{G}^{\#}$ along the $3$-rd mode. Finally, this gives us that
\begin{equation}
    S^{\#} = \mathcal G^{\#}_{(3)} (U^{\#}_3)^{\top} A. 
\end{equation}
Since $A$ is invertible with probability~$1$ and $\rank(\mathcal G^{\#}_{(3)}) = r_3$, $(\mathcal G^{\#}_{(3)},(U^{\#}_3)^\top)$ is a solution for the second step of the optimization problem in (\ref{eq:idea4tensororder3}).

The remainder of proof that shows that the solution returned by Procedure~4 is \textit{essentially unique} is similar to the proof of Theorem~\ref{thm:idea3order3}. 
\end{proof}


\section{Higher-order nTDs} \label{app:orderd}

In this section, we generalize the results for order-$3$ nTDs of the previous section to any order. 
The reasoning is the same but the notation gets slightly heavier. 

As for order-$3$ nTDs, we will assume the tensor follows the following assumption. 
\begin{assumptions}\label{ass:generalassorderd}
The order-$d$ tensor $\mathcal T\in \R^{n_1 \times \cdots  \times n_d}$ satisfies the following conditions: 
\begin{enumerate}
    \item $\mathcal T$ has an nTD of the form $\mathcal T = (U^{\#}_1,\dots,U^{\#}_d).\mathcal{G}^{\#}$ 
    where $U^{\#}_i \in \R_+^{n_i \times r_i}$ for all $i \in [d]$ and $\mathcal{G}^{\#} \in \R^{r_1 \times \dots \times r_d}$, and no columns of the $U^{\#}_i$'s are zero. 
    
    \item We assume w.l.o.g.\ that $(U^{\#}_i)^\top e = e$  for $i \in [d]$. 
    
\end{enumerate}
\end{assumptions}

\subsection{Using Unfoldings}\label{app:unfoldorderd}

Let $\mathcal{I} = \{j_1,\dots,j_k\} \subset [d]$. 
For any tuple $(i_1,\dots,i_d)$, we denote  $i_{\mathcal{I}} = (i_{j_1},\dots,i_{j_k})$, and $i_{\mathcal{I}^{-1}}$ the tuple with the rest of the $(d-k)$ elements, that is, $\mathcal{I}^{-1} = [d] \backslash \mathcal{I}$. 
For an order-$d$ tensor $\mathcal T$, we define the unfolding along the mode $\mathcal{I}$, denoted by $\mathcal{T}_{\mathcal{I}}$, as the $\prod_{j \in \mathcal{I}^{-1}} n_j \times \prod_{j \in \mathcal{I}} n_j$ matrix with its $( i_{\mathcal{I}^{-1}},i_{\mathcal{I}})$ entry given by $\mathcal{T}_{i_1,\dots,i_d}$. 
Similarly to Property~\ref{lem:changeofbasisunfoldings}, we have the following property for any unfolding. 
\begin{property}[Unfoldings] \label{lem:12flatenningchangeofbasisorderd}
Let $\mathcal T= (U_1,\dots,U_d).\mathcal{G}$ where $U_i \in \R^{n_i \times r_i}$ for all $i \in [d]$ and $\mathcal{G} \in \R^{r_1 \times  \times r_d}$. 
The unfolding along mode $\mathcal I$ of $\mathcal T$ can be factorized as 
\begin{align*}
    \mathcal T_{\mathcal{I}} \; =  \;  (\otimes_{j \in \mathcal{I}^{-}}  U_j) \, \mathcal{G}_{\mathcal{I}} \, (\otimes_{j \in \mathcal{I}} U_j)^\top .
\end{align*}
\end{property}

Assuming $\otimes_{j \in \mathcal{I}^{-}}  U_j$ and $\otimes_{j \in \mathcal{I}} U_j$ satisfy the SSC, we will be able to recover them using min-vol order-2 nTD (Theorem~\ref{thm:nmfmain}), while we will be able to recover the individual $U_i$'s by using Kronecker factorizations of these matrices. 

\begin{assumptions}\label{ass:orderdunfoldSSC} 
The tensor $\mathcal{T} = (U^{\#}_1,\dots,U^{\#}_d).\mathcal{G}$ satisfies  Assumption~\ref{ass:generalassorderd} 
 and there exists a non-empty set $\mathcal{I} \subsetneq  [d]$   
such that the following conditions are satisfied 
    \begin{itemize}
        \item $\otimes_{i \in S} U^{\#}_i $ satisfy the SSC for $S \in \{\mathcal{I},\mathcal{I}^{-}\}$. 
        
        \item $\rank(\mathcal{G}_{\mathcal{I}})  = \prod_{i \in \mathcal{I}} r_i = \prod_{j \in \mathcal{I}^{-}} r_j$.
        
    \end{itemize}
\end{assumptions}

Hence adapting Procedure~0 for order-$3$ nTDs, we propose the following procedure to compute an identifiable order-$d$ nTD under Assumption~\ref{ass:orderdunfoldSSC}. We call it Procedure~$d$.0 as it generalizes Procedure~0 for order-$d$ tensors. 

\vspace{0.1cm}

\fbox{%
	\parbox{0.95\linewidth}{%
    \begin{center} \vspace{-0.2cm}   
       \textbf{Procedure $d$.0: Unique order-$d$ nTD under Assumption~\ref{ass:orderdunfoldSSC}} 
    \end{center}
\begin{enumerate}
    \item Solve the following min-vol order-2 nTD problem                      
    \begin{align}\label{eq:idea0tensororderd}
    \min_{U_{\mathcal{I}},U_{\mathcal{I}^{-1}},G}  |\det(G)| \text{ such that }  & 
     \mathcal T_{\mathcal{I}} = U_{\mathcal{I}^{-}}G U_{\mathcal{I}}^{\top},  \\ 
     & 
      U_{\mathcal{I}}^\top e = e, U_{\mathcal{I}^{-}}^\top e = e, \text{ and } ( U_{\mathcal{I}},U_{\mathcal{I}^{-}}) \geq 0, \nonumber 
    \end{align}
    to obtain an optimal solution $ (U^*_{\mathcal{I}}, U^*_{\mathcal{I}^{-1}},G^*)$. 
    
    \item  Computation of $(U^*_1,\dots,U_d^*)$: Find permutations $\Pi^*_{\mathcal{I}},\Pi^*_{\mathcal{I}^{-}}$ such that $U^*_{S}\Pi^*_{S}$ admits a decomposition $U^*_{S}\Pi^*_{S} =  \otimes_{j \in S} U_j^*$  such that $U_j^{* \top} e = e$ for all $j \in S$
    where $S \in \{\mathcal{I},\mathcal{I}^{-}\}$; see 
   Corollary~\ref{corr:kronpdtpermd} and Theorem~\ref{thm:uniquenesskronpdt}. 

    \item Computation of $\mathcal{G}^*: $ Compute the matrix $\mathcal{G}_{\mathcal{I}}^* = (\Pi_{\mathcal{I}^{-}}^*)^\top G^* \Pi^*_{\mathcal{I}}$ and fold it to form the order-$d$ tensor $\mathcal{G}^*$.
    
     

    \vspace{-0.1cm}  
\end{enumerate}
	}%
}

\vspace{0.1cm}

Following the proof of Theorem~\ref{thm:order2unfold} for order-$3$ nTD, one can show that if the given tensor satisfies 
Assumption~\ref{ass:orderdunfoldSSC}, then Procedure $d$.0 returns an \textit{essentially} unique nTD.

\begin{theorem} \label{thm:unfoldsorderd} 
Let $\mathcal T\in \R^{n_1 \times \dots \times n_d} = (U^{\#}_1,\dots,U^{\#}_d).\mathcal{G}^{\#}$ satisfy Assumption~\ref{ass:orderdunfoldSSC}. 
Then for any other decomposition $(U^*_1,\dots,U^*_d).\mathcal{G}^*$ of $\mathcal T$  obtained by 
Procedure~$d$.0, there exist permutation matrices $\Pi_i$ for all $i \in [d]$ such that $U_i^* = U_i^{\#} \Pi_i$ for all $i \in [d]$ and $\mathcal{G}^* = (\Pi_1^\top,\dots,\Pi_d^\top)\mathcal{G}^{\#}$. 
\end{theorem}
\begin{proof}
We first want to show that there exists a feasible solution to the optimization problem in (\ref{eq:idea0tensororderd}). Since $(U^{\#}_1,\dots,U^{\#}_d).\mathcal{G}^{\#}$ is the nTD induced by Assumptions \ref{ass:orderdunfoldSSC}, using Lemma~\ref{lem:12flatenningchangeofbasisorderd}, we get that
\begin{equation}\label{eq:order4groundstate}
  \mathcal T_{\mathcal{I}} = U_{\mathcal{I}^{-}}^{\#} \mathcal{G}^{\#}_{\mathcal{I}} (U_{\mathcal{I}}^{\#})^\top, \text{ where }  U_{S}^{\#} = \otimes_{i \in S}U^{\#}_i \text{ for all } S \in \{\mathcal{I},\mathcal{I}^{-}\}.  
\end{equation}

By Assumption \ref{ass:orderdunfoldSSC}, since $\otimes_{i \in S}U_i^{\#}$ satisfy $SSC$ for $S \in \{\mathcal{I},\mathcal{I}^{-}\}$  and $\rank(\mathcal G^{\#}_{\mathcal{I}}) = \prod_{i \in \mathcal{I}}r_i = \prod_{j \in \mathcal{I}^-}r_j$, solving \eqref{eq:idea0tensororderd} will find matrices $(U^*_{\mathcal{I}},U^*_{\mathcal{J}},G^*)$ such that there exist permutation matrices $\Pi'_{S}$ where  $U^*_S =(\otimes_{i \in S}U_i^{\#})\Pi'_S$ for all $S \in \{\mathcal{I},\mathcal{I}^{-}\}$ and $G^* = (\Pi'_{\mathcal{I}^-})^\top \mathcal G^{\#}_{\mathcal{I}} \Pi_{\mathcal{I}}$ by Theorem~\ref{thm:nmfmain}.

Since at the end of Step 2, we have recovered matrices $U_i^*$ for all $i \in [d]$ and permutation matrices $\Pi_{\mathcal{I}}^*,\Pi_{\mathcal{I}^-}^*$ such that 
\begin{align*}
    \otimes_{i \in S} U_i^*  = U_{S}^* \Pi_S^* = (\otimes_{i \in S} U_i^{\#})\Pi_{S} \Pi_{S}^* \text{ for all } S \in \{\mathcal{I},\mathcal{I}^{-}\}.
\end{align*}
Using Corollary \ref{corr:kronpdtpermd} (see below), this implies that there exist permutation matrices $\Pi_i$ for all $i \in [d]$  such that $U_i^* = U_i^{\#} \Pi_i$ for all $i \in [d]$. Note that, from this one can further deduce that $\Pi_{S}\Pi_{S}^* = \prod_{i \in S}\Pi_i$ for all $S \in \{\mathcal{I},\mathcal{I}^{-}\}$.

Rewriting $\mathcal G^*_{\mathcal{I}} = (\Pi^*_{\mathcal{I}^-})^\top G^* \Pi^*_{\mathcal{I}} = (\otimes_{i \in \mathcal{I}^-} \Pi_i)^\top \mathcal{G}^{\#}_{\mathcal{I}} (\otimes_{i \in \mathcal{I}}\Pi_i)$ in the tensor structure gives us that $\mathcal{G}^* = (\Pi_1^\top,\dots,\Pi_d^\top).\mathcal{G}^{\#}$.
\end{proof}

In order to compute the permutations in step 2 of Procedure $d$.0, Property \ref{prop:kronpdtperm} can be extended to a Kronecker product of $d \geq 2$ matrices. 
This follows from the following facts:
\begin{itemize}
    \item If $U_i^\top e = e$ for $i \in [d]$, then for any $k \leq d$, $(\otimes_{i \in [k]} U_i)^\top e = e$.
    \item For matrices $A,B$, $\rank(A \otimes B) = \rank(A)\rank(B)$.
\end{itemize}
Let $(\otimes_{i \in [d]} U^{*}_i) = Y = X\Pi = (\otimes_{i \in [d]} U^{\#}_i) \Pi$. 
Then one can apply Property~\ref{prop:kronpdtperm} to the two factors $\otimes_{i \in [d-1]} U^{*}_i$  and $U_d^*$, which satisfy the conditions of Property \ref{prop:kronpdtperm} due to the above facts. 
This gives us that $\otimes_{i \in [d-1]} U^{*}_i = (\otimes_{i \in [d-1]} U^{\#}_i)\Pi'$ for some permutation $\Pi'$, and we can then proceed inductively. Hence, we get the following corollary.
\begin{corollary}\label{corr:kronpdtpermd}
Let $X$ be an $\prod_{i=1}^dn_i \times \prod_{i=1}^d r_i$ matrix with a decomposition of the form $X = U^{\#}_1 \otimes \dots \otimes U^{\#}_d$ where $U^{\#}_i \in \R^{n_i \times r_i}$ are full column rank matrices for all $i \in [d]$. Let $Y = X \Pi$ where $\Pi$ is a permutation matrix of dimension $\prod_{i=1}^d r_i $, 
and $Y = U^*_1 \otimes \dots \otimes U^*_d$ where $U^*_i \in \R^{n_i \times r_i}$ are full column rank matrices for all $i \in [d]$. Moreover, $(U_i^*)^\top e = (U_i^{\#})^\top e = e$ for all $i \in [d]$. Then there exist permutation matrices $\Pi_i$ such that $U_i^* = U_i^{\#} \Pi_i$ for all $i \in [d]$.
\end{corollary}

\paragraph{All-at-once approach}

Similarly as done in Section~\ref{sec:unfoldings}, instead of Procedure $d$.0, one can consider the following all-at-once procedure for computing the nTD via penalization: for any $\lambda_1,\lambda_2 > 0$, solve the following min-vol order-$2$ NTD problem: 
\begin{equation}\label{eq:optproblemhigherorderunfoldingalt}
\begin{split}
    \min_{U_i \in \R_+^{n_i \times r_i} \text{ for }  i \in [d], \ \mathcal{G}_{\mathcal{I}}} 
    & \left|\det(\mathcal{G}_{\mathcal{I}})\right| + \lambda_1 \left(\|U_{\mathcal{I}^-} - \bigotimes_{i \in \mathcal{I}^-}U_i \|_F^2\right) + \lambda_2\left(\|U_{\mathcal{I}} - \bigotimes_{i \in \mathcal{I}}U_i \|_F^2\right) \\
    \text{such that } &\mathcal T_{\mathcal{I}} = U_{\mathcal{I}^-}\det(\mathcal G_{\mathcal{I}})U_{\mathcal{I}}^\top,  U_i^\top e = e \text{ and } U_i \geq 0 \text{ for } i \in [d], 
\end{split}
\end{equation}
where $\mathcal{G}_{\mathcal{I}} \in \R^{\Pi_{j \in \mathcal{I}^{-}}r_j \times \Pi_{i \in \mathcal{I}} r_i}$, 
to obtain an optimal solution $(U^*_1,\dots,U^*_d,\mathcal{G}^*_{\mathcal{I}})$. By a result similar to Theorem \ref{thm:unfoldsorder3alternate}, one can also show that, under Assumption~\ref{ass:orderdunfoldSSC}, for all $\lambda_1,\lambda_2 > 0$, an nTD obtained by solving (\ref{eq:optproblemhigherorderunfoldingalt}) is essentially unique.

\subsection{Using Slices}\label{app:orderdslices}

Let us define slices of higher-order tensors.  
\begin{definition}[Slices for Higher-Order]\label{def:sliceshigherorder}
Let $\mathcal T\in \mathbb{R}^{n_1 \times \dots \times n_d}$ be an order-$d$ tensor. 
For $i < j \in [d]$ and $k_p \in [n_p]$ for $p \neq i,j$, 
we define the $[i,j]$-slices of dimension $n_i \times n_j$ as 
\begin{equation}\label{def:slicesgen}
    (\mathcal T^{[i,j]}_{k_1, \dots, k_{i-1}, k_{i+1}, \dots, k_{j-1}, k_{j+1}, \dots, k_d})_{k_i,k_j} 
    = \mathcal T_{k_1, \dots, k_d}, \quad \text{ for } k_i \in [n_i] \text{ and } k_j \in [n_j].
\end{equation}    
\end{definition} 
Extending Definition~\ref{def:slices} of the slices of order-$3$ tensors directly to order-$d$ would lead to the notation $T^{(1,\dots,i-1,i+1,\dots,j-1,j+1, \dots,d)}_{k_1, \dots, k_{i-1}, k_{i+1}, \dots, k_{j-1}, k_{j+1}, \dots, k_d}$. However, we use the notation from~\eqref{def:slicesgen} for simplicity of the exposition.

\begin{remark}
The slices have been defined in this way (by fixing $d-2$ indices) for order-$d$ tensors in order to still preserve the fact that each $[i,j]$-slice is indeed an $n_i \times n_j$ matrix. We will see in Section \ref{sec:orderdunfoldandslices} that one could relax this notion of slices by fixing $j \leq d-2$ indices which would result in an order-$(d-j)$ subtensor and consider the slices to be a suitable flattening of the subtensor. 
\end{remark}

Property \ref{lem:changeofbasisslices} can be generalized to order-$d$ tensors in the following way. 
\begin{lemma}\label{lem:orderdslices}
Let $\mathcal T$ be an order-$d$ tensor such that there exists a decomposition of the form $\mathcal T= (U_1,\dots,U_d).\mathcal{G}$. Then the $[1,2]$-slices of $\mathcal T$ are given by
$$
\mathcal T^{[1,2]}_{k_3, \dots, k_d} = U_1 S^{[1,2]}_{k_3, \dots, k_d}U_2^{\top}, 
$$ 
where $$ 
S^{[1,2]}_{k_3, \dots, k_d} = \sum_{t_i \in [r_i] \text{ for } i \in \{3, \dots, d\}} \Big(\prod_{i=3}^d (U_i)_{k_i,t_i} \Big) \mathcal{G}^{[1,2]}_{t_3, \dots, t_d}
$$ 
and 
$\mathcal{G}^{[1,2]}_{t_3, \dots, t_d}$ are the $[1,2]$-slices of the tensor $\mathcal{G}$.
\end{lemma}

\subsubsection{Approach 1: $d-1$ modes of the input tensor have one slice with maximum rank} 

The idea from Section~\ref{sec:Idea1} can be generalized to any order by computing one min-vol order-2 nTD to recover $U_1$ and $U_2$, and $d-2$ min-vol MF to recover the other $U_i$'s.  
The following assumptions generalizes Assumption~\ref{ass:order3idea1}, that is, there exist at least one slice along $d$ different directions such that each of the slice achieves maximum possible rank.

\begin{assumptions}\label{ass:orderdslicesindfactors}
The order-$d$ tensor $\mathcal T\in \R^{n_1 \times \dots \times n_d} = (U^{\#}_1,\dots,U^{\#}_d).\mathcal{G}^{\#}$ satisfies Assumption~\ref{ass:generalassorderd}, and 
\begin{enumerate}

    \item $U^{\#}_i$ satisfies the SSC  
    for $i \in [d]$.  

    \item For all $i \in \{2,\dots,d\}$, there exist indices $k^{(i)} = \{k^{(i)}_1,\dots,k^{(i)}_{d-2}\}$ such that $\rank(\mathcal T^{[1,i]}_{k^{(i)}}) = r_i$.
        
\end{enumerate}
\end{assumptions}


    

Let us now describe how to identify order-$d$ nTDs satisfying Assumption~\ref{ass:orderdslicesindfactors}, following the same idea as in Procedure~1. \vspace{0.1cm} 

\fbox{%
	\parbox{0.95\linewidth}{%
    \begin{center} \vspace{-0.2cm}   
       \textbf{Procedure~$d$.1 : Unique order-$d$ nTD under Assumption~\ref{ass:orderdslicesindfactors}} 
    \end{center}
\begin{enumerate}
    \item Computation of $U^*_1, U^*_2$: Solve the following min-vol order-2 nTD problem                      
    \begin{equation}\label{eq:idea1tensororderd}
    \min_{U_1,U_2,S_2}  |\det(S_{2})| \text{ such that } 
     \mathcal T^{[1,2]}_{k^{(2)}} = U_1 S_2 U_2^{\top},  
     U_i^\top e = e \text{ and } U_i \geq 0 \text{ for } i \in [2], 
    \end{equation}
    to obtain an optimal solution $ (U^*_1,U^*_2,S^*_1)$. 
    
    \item Computation of $U^*_i$: For all $i \in [d] \setminus \{1,2\}$, let $\mathcal T'_{k^{(i)}} = (U^*_1)^{\dagger}\mathcal T^{[1,i]}_{k^{(i)}}$ and solve the following 
    min-vol NMF problem                      
    \begin{equation}\label{eq:idea1tensororderdb}
   \min_{U_i,S_{i}} \det(S_{i}^\top S_{i}) 
   \text{ such that }\mathcal T'_{k^{(i)}} = S_{i} U_i^\top,  U_i^{\top} e = e \text{ and } U_i\geq 0, 
    \end{equation} 
to obtain an optimal solution $(U^*_i,S^*_{i})$. 
    


    \item Computation of $\mathcal{G}^* = ((U^{*}_1)^{\dagger},\dots,(U^{*}_d)^{\dagger}). \mathcal T$.   \vspace{-0.1cm}  
\end{enumerate}
	}%
}


\vspace{0.1cm}


As for order-$3$ nTD in Theorem~\ref{thm:idea1order3}, we now show that, under Assumption~\ref{ass:orderdslicesindfactors}, a decomposition computed by Procedure~$d$.1 is essentially unique. 

\begin{theorem} \label{thm:orderdslices1}  
Let $\mathcal T = (U^{\#}_1,\dots,U^{\#}_d).\mathcal{G}^{\#}$ satisfy 
Assumption~\ref{ass:orderdslicesindfactors}.  
Then for any other decomposition $(U^{*}_1,\dots,U^{*}_d).\mathcal{G}^{*}$ of $\mathcal T$ obtained with Procedure~$d$.1, 
there exist permutation matrices $\Pi_i$ such that 
$U^{*}_i = U^{\#}_i \Pi_i$ for all $i \in [d]$ and $\mathcal{G}^{*} = (\Pi_1^\top,\dots,\Pi_d^\top).\mathcal{G}^{\#}$.  
\end{theorem}
\begin{proof}
First, we show that if there exists an nTD $(U_1^{\#},\dots,U_d^{\#}).\mathcal{G}^{\#}$ of $\mathcal T$ that satisfies Assumptions \ref{ass:orderdslicesindfactors}, then there exist feasible solutions to the optimization problems in \eqref{eq:idea1tensororderd} and \eqref{eq:idea1tensororderdb}. Using Lemma \ref{lem:orderdslices},  $\mathcal T^{[1,2]}_{k^{(2)}} = U^{\#}_1 S^{[1,2]}_{k^{(2)}}(U^{\#}_2)^{\top}$ such that
\begin{equation}\label{eq:D12}
  S^{[1,2]}_{k^{(2)}} = \sum_{t_i \in [r_i] \text{ for } i \in [d] \setminus \{1,2\}} \Big(\prod_{j \in [d] \setminus  \{1,2\}}(U^{\#}_j)_{k^{(2)}_j,t_j} \Big)\mathcal{G}^{[1,2]}_{t_3,\dots,t_d}. 
\end{equation}
such that $\mathcal{G}^{[1,2]}_{t_3,\dots,t_d}$ are the $[1,2]$-slices of $\mathcal{G}^{\#}$. Hence, $(U^{\#}_1,U^{\#}_2,S^{(3,4)}_{i_3,i_4}) $ is a feasible solution to the optimization problem in \eqref{eq:idea1tensororderd}.

Then using Lemma~\ref{lem:orderdslices}, we have $\mathcal T^{[1,i]}_{k^{(i)}} = U^{\#}_1S^{[1,i]}_{k^{(i)}}(U^{\#}_i)^\top$ for all $i \in [d] \setminus \{1,2\}$. Since $U_1^{\#}$ satisfies the SSC , $\rank(U_1^{\#}) = r$ and hence, 
\begin{equation}\label{eq:T'ki}
\mathcal T'_{k^{(i)}} = (U_1^{\#})^{\dagger}\mathcal T^{[1,i]}_{k^{(i)}} = S^{[1,i]}_{k^{(i)}}(U^{\#}_i)^\top.    
\end{equation}
Using the fact that $\rank(\mathcal T^{(1,i)}_{k^{(i)}}) = r_3 \leq r$, we get  $\rank(\mathcal T'_{k^{(i)}}) = r_3$. It follows that $( S^{[1,i]}_{k^{(i)}}, U^{\#}_i)$ is indeed a feasible solution to~(\ref{eq:idea1tensororderdb}) for all $i \in [d] \setminus \{1,2\}$.


Now we want to show that a solution $(U^*_1,\dots,U^*_d).\mathcal{G}^*$ returned by the procedure is essentially unique. 
Let $(U^*_1,S_2^*,U^*_2)$ be an optimal solution to the \eqref{eq:idea1tensororderd}. Since $\rank(\mathcal T^{[1,2]}_{k^{(2)}}) = r$, following the proof of Theorem \ref{thm:nmfmain}, there exist permutation matrices $\Pi_1, \Pi_2$ such that $U_1^* = U_1^{\#} \Pi_1$, $U_2^* = U_2^{\#} \Pi_2$ and $S^*_1= \Pi_1^\top S^{[1,2]}_{k^{(2)}} \Pi_2$.

Let $\mathcal T^*_{k^{(i)}} = (U_1^{*})^{\dagger}(\mathcal T^{[1,i]}_{k^{(i)}})$. Then using the previous result, we also have that $\mathcal T^*_{k^{(i)}} = \Pi_1^\top \mathcal T'_{k^{(i)}}$ where $\mathcal{T}'_{k^{(i)}}$ is as mentioned in \eqref{eq:T'ki}. Let $(S^*_i,U_i^*)$ be an optimal solution to the optimization problem in \eqref{eq:idea1tensororderdb} when run on $\mathcal T^*_{k^{(i)}}$ for all $i \in [d] \setminus \{1,2\}$. Then $(\Pi^\top_1 S^{[1,i]}_{k^{(i)}},U_i^*)$ is also an optimal solution to the second step of the optimization problem in \eqref{eq:idea1tensororderdb} when run on $T^*_{k^{(i)}}$ for all $i \in [d] \setminus \{1,2\}$. Again following the proof of Theorem \ref{thm:idenminvol}, we can conclude that there exists a permutation matrix $\Pi_i$ such that $S^*_i = \Pi_1^\top S^{[1,i]}_{k^{(i)}} \Pi_i$ and $U_i^* = U_i^{\#} \Pi_i$ for all $i \in [d] \setminus \{1,2\}$. 
Moreover, since $U^{\#}_i$ have full column rank, we have that $(U^{\#}_i)^{\dagger}U^{\#}_i = I_{r_i}$ for all $i \in [d]$. Using this, we already have that $\mathcal{G}^{\#} = ((U^{\#}_1)^{\dagger},\dots,(U^{\#}_d)^{\dagger}).\mathcal T$, and we conclude that 
\begin{align*}
\mathcal{G}^{*} = ((U_1^*)^{\dagger
},\dots,(U_d^*)^{\dagger
}).T &= (\Pi^\top_1, \dots,\Pi^\top_d).\Big(((U_1^{\#})^{\dagger
},\dots,(U_d^{\#})^{\dagger
}).T\Big) = (\Pi^\top_1, \dots,\Pi^\top_d).\mathcal{G}^{\#}    
\end{align*}

Finally, using, the following relation
$$ 
\mathcal T = (U^{\#}_1,\dots, U_d^{\#}).\mathcal{G}^{\#} = (U^{*}_1\Pi^\top_1,\dots,U^{*}_d\Pi^\top_d).\Big((\Pi_1, \dots, \Pi_d)\mathcal{G}^*\Big) = (U^{*}_1\dots,U_d^*).\mathcal{G}^*, 
$$
the decomposition returned by the procedure is an $(r,r,r_3,\dots,r_d)$-nTD of the tensor $\mathcal T$.
\end{proof}

\begin{remark}[Variant for Assumption \ref{ass:orderdslicesindfactors}(2)]
    Let $S \subset \binom{[d]}{2}$ where $|S| \geq d$. We say $S$ \textit{covers} $[d]$ if for all $j \in [d]$, there exists $i \in [d]$ such that $(i,j) \in S$ or $(j,i) \in S$. Then Assumption \ref{ass:orderdslicesindfactors}(2) can be replaced with the following one: There exists $S$ which covers $[d]$ such that for all $(i,j) \in S$, there exists $k^{(i,j)}_1,\dots,k^{(i,j)}_{d-2}$ such that $\mathcal T^{(i,j)}_{k^{(i,j)}_1,\dots,k^{(i,j)}_{d-2}}$ achieves maximum possible rank. The restriction on the dimensions of the core tensor needs to be imposed accordingly. And then, one can devise a procedure similar to Procedure~$d$.1 to compute the factors and the core tensor.
\end{remark}

\subsubsection{Approach 3 for order-$d$ tensors: fully generalized  slices}\label{sec:orderdunfoldandslices}


In this section, we generalize Approach 3 from Section \ref{sec:Idea3} to higher orders. Recall that, in that approach, we had rank guarantees on a slice along a certain mode and on the unfoldings of the core tensor along the same mode. For any two indices $i \neq j$, $[i,j]$-slices for an order-$d$ tensor $\mathcal{T} \in \R^{n_1 \times \dots \times n_d}$ are defined in Definition~\ref{def:sliceshigherorder} as $n_i \times n_j$ matrices formed by fixing $d-2$ indices of the tensor and letting the other two indices vary. However, one can generalize this definition by fixing a subset of indices $\mathcal{J} \subsetneq [d]$. 
The resulting slices are subtensors of order $\left(d - |\mathcal{J}|\right)$. 
Then, given another subset $\mathcal{K} \subsetneq [d] \setminus \mathcal{J}$, we call the slices to be the $\mathcal{K}$-unfoldings of the corresponding subtensors (following the definition of unfoldings of order-$d$ tensors).

Let $[d] = \mathcal{I} \sqcup \mathcal{J} \sqcup \mathcal{K}$, where $\sqcup$ denotes a partition into disjoint non-empty subsets, that is, 
$[d]  = \mathcal{I} \cup \mathcal{J} \cup \mathcal{K}$, 
$\mathcal{I} \neq \varnothing, \mathcal{J} \neq \varnothing, \mathcal{K} \neq \varnothing$,
 $\mathcal{I} \cap \mathcal{J} =\varnothing$, 
 $\mathcal{I} \cap \mathcal{K} =\varnothing$, and 
 $\mathcal{J} \cap \mathcal{K} =\varnothing$. 
For any index $(i_1,\dots,i_d)$, we define $i_S$ to be the tuple with elements $i_k$ where $k \in S$ for all $S \in \{\mathcal{I}, \mathcal{J}, \mathcal{K}\}$. Then for any order-$d$ tensor $\mathcal{T} \in \R^{n_1 \times \dots \times n_d}$, we can define the $(\mathcal{J},\mathcal{K})$-slices of $\mathcal{T}$ to be the $\prod_{i \in \mathcal{I}} n_i \times \prod_{k \in \mathcal{K}} n_k$ matrices $(\mathcal{T}^{(\mathcal{J})}_{i_{\mathcal{J}}})_{i_{\mathcal{I}},i_{\mathcal{K}}} = T_{i_1, \dots, i_d}$. 
Note that ideally one should refer to these slices by $\mathcal{T}^{(\mathcal{J},\mathcal{K})}_{i_{\mathcal{J}}}$ since $\mathcal{K}$ is necessary to fix the corresponding unfoldings. But here we will just refer to them as $\mathcal{T}^{(\mathcal{J})}_{i_{\mathcal{J}}}$ because the subset of modes along which the unfolding will take place will be clear from context.

Note that since $\mathcal{I},\mathcal{J},\mathcal{K} \neq \varnothing$, for the order-$3$ tensor case, this forces $|\mathcal{I}| = |\mathcal{J}| = |\mathcal{K}| = 1 $. Then in that case, the $(\mathcal{J},\mathcal{K})$-slices correspond to the actual slices of the tensor. 

The following structural result holds for $(\mathcal{J},\mathcal{K})$-slices of $\mathcal{T}$.  
\begin{lemma}
Let $\mathcal{T} \in \R^{n_1 \times \dots \times n_d}$ be an order-$d$ tensor with TD $\mathcal{T} = (U_1,\dots,U_d).\mathcal{G}$, where $U_i \in \R^{n_i \times r_i}$ for all $i \in [d]$ and $\mathcal{G} \in \R^{r_1 \times \dots \times r_d}$. 
Then, for any $\mathcal{I},\mathcal J,\mathcal{K}$ such that $[d] = \mathcal{I} \sqcup \mathcal{J} \sqcup \mathcal{K}$,  the $(\mathcal{J},\mathcal{K})$-slices denoted by $\mathcal{T}^{(\mathcal{J})}_{i_{\mathcal{J}}}$ can be decomposed as 
\begin{equation}
    \mathcal{T}^{(\mathcal{J})}_{i_{\mathcal{J}}} = \Big(\bigotimes_{i \in \mathcal{I}} U_i\Big) D^{(\mathcal{J})}_{i_{\mathcal{J}}} \Big(\bigotimes_{k \in \mathcal{K}} U_k\Big)^\top, 
\end{equation}
where 
$$
D^{(\mathcal{J})}_{i_{\mathcal{J}}} = \sum_{i'_j \in [r_j] \text{ for all } j \in \mathcal{J}} (\prod_{k \in \mathcal{J}} (U_k)_{i_k,i'_k}) \mathcal{G}^{(\mathcal{J})}_{i'_{\mathcal{J}}}, 
$$
and $i'_{\mathcal{J}}$ is the tuple with $i'_j$ for all $j \in \mathcal{J}$ and $\mathcal{G}^{(\mathcal{J})}_{i_{\mathcal{J}}}$ are the $(\mathcal{J},\mathcal{K})$-slices of $\mathcal{G}$.

Moreover, if $D^{(\mathcal{J})}$ is the $(\prod_{i \in \mathcal{I}} r_i)(\prod_{k \in \mathcal{K} }r_k) \times (\prod_{j \in \mathcal{J} }r_j)$ matrix with its columns given by $\vect(D^{(\mathcal{J})}_{i_{\mathcal{J}}})$, then $D^{(\mathcal{J})}$ can be decomposed as
\begin{equation}
D^{(\mathcal{J})} = \mathcal{G}_{\mathcal{J}} \Big(\bigotimes_{j \in \mathcal{J}} U_j\Big)^\top . 
\end{equation}
\end{lemma}
\begin{proof}
Let $\mathcal{I} \sqcup \mathcal{J} \sqcup \mathcal{K}$ be a partition of $[d]$ into disjoint non-empty subsets,  and for a tuple of indices $(i_1,\dots,i_d)$, let $i_{\mathcal{I}},i_{\mathcal{J}},i_{\mathcal{K}}$ be the tuples of indices induced by this partition. Following the definition of the multilinear transformation operation,
  we have that
  \begin{align*}
  (\mathcal{T}^{(\mathcal{J})}_{i_{\mathcal{J}}})_{i_{\mathcal{I}},i_{\mathcal{K}}} =T_{i_1,\dots,i_d} &= \sum_{i'_t \in [r_t] \text{ for all } t \in [d] }\prod_{k \in [d]}(U_k)_{i_k,i'_k} \mathcal{G}_{j_1,\dots,j_d}    \\
  &=  \Big(\bigotimes_{i \in \mathcal{I}} U_i\Big)_{i_{\mathcal{I}},i'_{\mathcal{I}}} (D^{(\mathcal{J})}_{i_{\mathcal{J}}})_{i'_{\mathcal{I}},i'_{\mathcal{K}}} \Big(\bigotimes_{i \in \mathcal{K}} U_i\Big)_{i_{\mathcal{K}},i'_{\mathcal{K}}}, 
  \end{align*}
  where 
  \begin{equation}\label{eq:DJiJ}
       D^{(\mathcal{J})}_{i_{\mathcal{J}}} = \sum_{i'_t \in [r_t] \text{ for all } t \in \mathcal{J} }(\prod_{k \in \mathcal{J}}(U_k)_{i_k,i'_k}) (\mathcal{G}^{(\mathcal{J})}_{i'_{\mathcal{J}}}), 
  \end{equation}
 and  $i'_{\mathcal{I}},i'_{\mathcal{J}},i'_{\mathcal{K}}$ is the partition of the tuple $(i'_1,\dots,i'_d)$ induced by the disjoint subsets of $[d]$.

Since $D^{\mathcal{J}}$ is a matrix of dimension $(\prod_{i \in \mathcal{I}} r_i)(\prod_{k \in \mathcal{I}} r_k) \times (\prod_{j \in \mathcal{J}} r_j)$, we will denote its entries with the indices $(i_\mathcal{I},i_{\mathcal{K}}),i_{\mathcal{J}}$. If $D^{\mathcal{J}}$ is the matrix with columns given by $\vect( D^{(\mathcal{J})}_{i_{\mathcal{J}}})$, then using \eqref{eq:DJiJ}, we get that
 \begin{align*} (D^{\mathcal{J}})_{(i_{\mathcal{I}},i_{\mathcal{K}}),i_{\mathcal{J}}} &= (D^{(\mathcal{J})}_{i_{\mathcal{J}}})_{(i_{\mathcal{I}},i_{\mathcal{K}})} \\
     &= \sum_{i'_t \in [r_t] \text{ for all } t \in \mathcal{J} }(\prod_{k \in \mathcal{J}}(U_k)_{i_k,i'_k}) (\mathcal{G}^{(\mathcal{J})}_{i'_{\mathcal{J}}})_{i_{\mathcal{I}},i_{\mathcal{K}}} \\
     &= \sum_{i'_t \in [r_t] \text{ for all } t \in \mathcal{J}}(\mathcal{G}_{\mathcal{J}})_{(i_{\mathcal{I}},i_{\mathcal{K}}),i'_{\mathcal{J}}} (\bigotimes_{j \in \mathcal{J}} U_j)^\top_{i'_{\mathcal{J}},i_{\mathcal{J}}} \\
     &= \Big(\mathcal{G}_{\mathcal{J}}. (\bigotimes_{j \in \mathcal{J}}U_j)^\top\Big)_{(i_{\mathcal{I}},i_{\mathcal{K}}),i_{\mathcal{J}}}. 
 \end{align*}
\end{proof}

In the rest of this section, we generalize the ideas from Section~\ref{sec:Idea3}  to higher orders.


\begin{assumptions}\label{ass:orderdslices} The order-$d$ tensor $\mathcal T\in \R^{n_1 \times \dots \times n_d} = (U^{\#}_1,\dots,U^{\#}_d).\mathcal{G}^{\#}$ satisfies Assumption~\ref{ass:generalassorderd}, and  
\begin{enumerate}
    
    \item There exists a partition $\mathcal{I} \sqcup \mathcal{J} \sqcup \mathcal{K}$ of $[d]$ such that 
    $\bigotimes_{i \in S} U^{\#}_i$ satisfies the SSC for  $S \in \{\mathcal{I},\mathcal{J},\mathcal{K}\}$.  

    \item There exists $i_{\mathcal{J}}$ such that $\rank(\mathcal T^{(\mathcal{J})}_{i_{\mathcal{J}}}) = \prod_{i \in \mathcal{I}} r_i = \prod_{k \in \mathcal{K}} r_k =  r$.
    
    \item The unfolding of the core tensor along the mode $\mathcal{J}$ satisfies $\rank(\mathcal{G}^{\#}_{\mathcal{J}}) = \prod_{j \in \mathcal{J}} r_j \leq r^2$. 
\end{enumerate}
\end{assumptions}

Given an order-$d$ tensor $\mathcal T$ that satisfies Assumption~\ref{ass:orderdslices}, 
let us consider the following procedure for computing a decomposition.

\vspace{0.1cm}   
\fbox{%
	\parbox{0.95\linewidth}{%
    \begin{center} \vspace{-0.2cm}   
       \textbf{Procedure~$d$.3: Unique order-$d$ nTD under Assumption~\ref{ass:orderdslices}} 
    \end{center}
\begin{enumerate}

\item Solve the following min-vol order-2 nTD problem  
\begin{equation}\label{eq:idea3tensororderd}
    \min_{U_{\mathcal{I}}, U_{\mathcal{K}}, S_1}   |\det(S_1)| 
   \text{ such that }  \mathcal T^{(\mathcal{J})}_{i_{\mathcal{J}}} = U_{\mathcal{I}} S_1 U_{\mathcal{K}}^{\top},   
      U_i^\top e =e \text{ and } U_i \geq 0 \text{ for } i \in \{\mathcal{I}, \mathcal{K}\}, 
\end{equation} 
     to obtain an optimal solution $(U_{\mathcal{I}}^*,U_{\mathcal{K}}^*,S_1^*)$. 

   \item   Form the matrix  $S^*$  to be the matrix with $i_{\mathcal{J}}$-th column as $\vect({U_{\mathcal{I}}^*}^{ \dagger}\mathcal T^{(\mathcal{J})}_{i_{\mathcal{J}}}({U_{\mathcal{K}}^*}\top)^{\dagger})$, and solve the min-vol NMF problem 
       \begin{align} \label{eq:idea3tensororderdb} 
       \min_{G,U_{\mathcal{J}}} 
        \det(G^\top G)  
     \text{ such that }  S^* = G U_{\mathcal{J}}^\top, U_{\mathcal{J}}^\top e = e, U_{\mathcal{J}} \geq 0, 
\end{align} 
to obtain an optimal solution $(G^*, U_{\mathcal{J}}^*)$. 
    \item Compute $U_i^*$ for all $i \in [d]$:  Find permutations $\Pi^*_{\mathcal{I}},\Pi^*_{\mathcal{J}},\Pi^*_{\mathcal{K}}$ such that $U^*_{S}\Pi^*_{S}$ admits a decomposition $U^*_{S}\Pi^*_{S} =  \otimes_{j \in S} U_j^*$  such that $U_j^{* \top} e = e$ for all $j \in S$
    where $S \in \{\mathcal{I},\mathcal{J},\mathcal{K}\}$; see 
    Corollary~\ref{corr:kronpdtpermd} and Theorem~\ref{thm:uniquenesskronpdt}. 

\item Computation of $\mathcal{G}^*: $ Compute the matrix $\mathcal{G}^*_{(\mathcal{J})} = (\Pi_{\mathcal{I}}^* \otimes \Pi_{\mathcal{K}}^*)^{\top} G^* \Pi_{\mathcal{J}}^*$ and fold the matrix to compute the order-$d$ tensor $\mathcal{G}^*$.
\vspace{-0.1cm}  
\end{enumerate}
}
}
\vspace{0.1cm}

One can then prove a generalization of 
Theorem~\ref{thm:idea3order3} to show that under Assumption~\ref{ass:orderdslices}, a decomposition computed by Procedure~$d$.3 is essentially unique.

\begin{theorem} \label{thm:idea3orderd}  
Let $\mathcal T$ be an order-$d$ tensor such that Assumption \ref{ass:orderdslices} is satisfied and let $ (U^{\#}_1,\dots,U^{\#}_d).\mathcal{G}^{\#}$ be the decomposition induced by these assumptions. Then for any other decomposition $ (U^*_1,\dots,U^*_d).\mathcal{G}^*$ returned by the procedure, there exist permutation matrices $\Pi_i$ for $i \in [d]$ such that $U^*_i = U^{\#}_i \Pi_i$ for all $i \in [d]$ and $\mathcal{G}^* = (\Pi_1^\top,\dots,\Pi_4^\top).\mathcal{G}^{\#}$.
\end{theorem}
\begin{proof}
First, we show that the existence of an nTD $(U^{\#}_1,\dots,U^{\#}_d).\mathcal{G}^{\#}$ of the input tensor $\mathcal T$ that satisfies Assumptions \ref{ass:orderdslices} also proves the existence of a feasible solution to the optimization problems in \eqref{eq:idea3tensororderd} and \eqref{eq:idea3tensororderdb}. Using Lemma \ref{lem:orderdslices}, the $(\mathcal{J},\mathcal{K})$-slices have the form  $\mathcal T^{(\mathcal J)}_{i_{\mathcal{J}}} = U^{\#}_{\mathcal{I}} S^{(\mathcal{J})}_{i_{\mathcal{J}}}(U^{\#}_{\mathcal{K}})^{\top}$ for all $j_3 \in [n_3], j_4 \in [n_4]$ such that
\begin{equation}\label{eq:D12u3kronu4}
   S^{(\mathcal{J})}_{i_{\mathcal{J}}} = \sum_{i'_j \in [r_j] \text{ for all } j \in \mathcal{J}} (\prod_{j \in \mathcal{J}} (U^{\#}_j)_{i_j,i'_j}) \mathcal{G}^{(\mathcal{J})}_{i'_{\mathcal{J}}}, 
\end{equation}
where $i'_{\mathcal{J}} $ is the tuple of indices $i'_j$ for all $j \in \mathcal{J}$,  $\mathcal{G}^{(\mathcal{J})}_{i'_{\mathcal{J}}}$ are the $(\mathcal{J},\mathcal{K})$-slices of $\mathcal{G}^{\#}$ and $U^{\#}_{S} = \otimes_{j \in S} U_j^{\#}$ where $S \in \{\mathcal{I},\mathcal{K}\}$.
Hence, $\Big(U^{\#}_{\mathcal{I}}, U^{\#}_{\mathcal{J}},S^{(\mathcal{J})}_{i_{\mathcal{J}}}\Big)$ is a feasible solution to the optimization problem in (\ref{eq:idea3tensororderd}). If $S^{(\mathcal{J})}$ is the matrix with its $i_{\mathcal{J}}$ column as $\vect\Big(S^{(J)}_{i_{\mathcal{J}}}\Big)$, rewriting (\ref{eq:D12u3kronu4}) in matrix form, and using Lemma \ref{lem:orderdslices}, we get that
$
S^{(\mathcal{J})} = \mathcal{G}_{\mathcal{J}} (U^{\#}_{\mathcal{J}})^\top$ 
where $U^{\#}_{\mathcal{J}} = \bigotimes_{j \in \mathcal{J}} U_j$ and $ \mathcal{G}^{\#}_{(\mathcal{J})}$  is the unfolding of $\mathcal{G}^{\#}$ along the $\mathcal{J}$ mode (following the definition at the beginning of 
Section~\ref{app:unfoldorderd}). Then $(\mathcal{G}^{\#}_{\mathcal{J}},U^{\#}_{\mathcal{J}})$ is a feasible solution to the optimization problem in (\ref{eq:idea3tensororderdb}).

Now we want to show that a solution $(U^*_1,\dots, U^*_d).\mathcal{G}^*$ returned by the procedure is essentially unique.

Let $(U^*_{\mathcal{I}},U^*_{\mathcal{K}},S_{\mathcal{J}}^*)$ be another distinct optimal solution to the first step of the optimization problem in (\ref{eq:idea3tensororderd}). Then following the proof of Theorem \ref{thm:nmfmain}, we already get that there exist permutation matrices $\Pi'_{\mathcal{I}},\Pi'_{\mathcal{K}}$ such that $ U^*_{S} = U^{\#}_{S}\Pi'_{S}$ where $U^{\#}_S = \otimes_{i \in S}U_i^{\#}$ for $S \in \{\mathcal{I},\mathcal{K}\}$ and $S_{\mathcal{J}}^* = (\Pi'_{\mathcal{I}})^\top S_i^{\#} \Pi'_{\mathcal{J}}$. Let $S^{\#}$ be the matrix with $i_{\mathcal{J}}$-th column given by $\vect((U^{\#}_{\mathcal{I}})^{\dagger}\mathcal T^{(\mathcal{J})}_{i_{\mathcal{J}}}((U^{\#}_{\mathcal{K}})^\top)^{\dagger})$. Now, following the proof of Theorem \ref{thm:idea3order3} (refer to (\ref{eq:kronpermutation})), we get that
\begin{equation}\label{eq:DD*}
    S^* = (\Pi'_{\mathcal{I}} \otimes \Pi'_{\mathcal{K}})^\top S^{\#}.
\end{equation}
Let $(G^*, U_{\mathcal{J}}^*)$ be another distinct optimal solution to the second step of the optimization problem in  \eqref{eq:idea3tensororderdb} when run on $S^*$. Using (\ref{eq:DD*}), $((\Pi'_{\mathcal{I}} \otimes \Pi'_{\mathcal{K}})^\top \mathcal{G}_{\mathcal{J}}^{\#},   U_{\mathcal{J}}^*)$ is also an optimal solution to the optimization problem in (\ref{eq:idea3tensororderdb}) when run on $S^*$. Then using Theorem \ref{thm:nmfmain}, we get that there exists a permutation matrix $\Pi_{\mathcal{J}}$ such that $G^* = (\Pi'_{\mathcal{I}} \otimes \Pi'_{\mathcal{K}})^\top \mathcal{G}_{\mathcal{J}}^{\#}\Pi'_{\mathcal{J}}$ and $U_{\mathcal{J}}^* = (\otimes_{j \in \mathcal{J}}U_{j}^{\#})\Pi'_{\mathcal{J}}$ .

At the end of Step 3 of Procedure~$d$.3, we have recovered matrices $U_i^*$ for all $i \in S$ and permutation matrix $\Pi^*_{S}$ where $S \in \{\mathcal{I},\mathcal{J},\mathcal{K}\}$ such that 
$$
\otimes_{j \in S}U_j^*  = U_{S}^* \Pi_S^* = (\otimes_{j \in S}U_j^{\#})\Pi'_{S} \Pi_{S}^*.
$$  
Corollary \ref{corr:kronpdtpermd} shows that there exist permutation matrices $\Pi_1,\dots,\Pi_d$ such that $U_i^* = U_i^{\#}\Pi_i$ for $i \in [d]$. It also follows that $\Pi'_S \Pi_S^* = \otimes_{i \in S} \Pi_i$ for $S \in \{\mathcal{I},\mathcal{J},\mathcal{K}\}$. This gives us that $$\mathcal G_{\mathcal J}^* =  (\otimes_{i \in \mathcal{I} \cup \mathcal{K}}\Pi_i)^\top \mathcal{G}_{\mathcal{J}}^{\#}(\otimes_{j \in \mathcal{J}}\Pi_i)$$ using which we can conclude that $\mathcal{G}^* = (\Pi_1^\top,\dots,\Pi_d^\top).\mathcal{G}^{\#}$.
\end{proof}

\paragraph{Two-step procedure}  

Instead of using Procedure~$d$.3, one can also consider the following two-step procedure: 

\begin{enumerate}
    \item 
For any $\lambda_1,\lambda_2,\lambda_3 > 0$, solve the following optimization problem
\begin{align*}\label{eq:optorderdalt}
    \min_{U_i \text{ for } i \in \mathcal{I} \cup \mathcal{K}, S_1 } & |\det(S_1)| + \lambda_1\|U_{\mathcal{I}} - \otimes_{i \in \mathcal{I}} U_i\|_F^2 + \lambda_2\|U_{\mathcal{J}} - \otimes_{k \in \mathcal{K}} U_k\|_F^2 
    \\
   \text{ such that } & \mathcal T^{(\mathcal{J})}_{i_{\mathcal{J}}} = U_{\mathcal{I}} S_1U_{\mathcal{K}}^{\top}, U_i^{\top}e = e \text{ and } U_i \geq 0 \text{ for } i \in \mathcal{I} \cup \mathcal{K}, 
\end{align*}
   to obtain an optimal solution   $U^*_i$ for all $i \in \mathcal{I} \cup \mathcal{K}$ and $S^*_1$. 

\item Form the matrix  $S^*$ to be the matrix with $((i_{\mathcal{I}},i_{\mathcal{K}}),i_{\mathcal{J}})$-th column given by 
\[
\vect\left((\otimes_{i \in \mathcal{I}}U^*_i)^{\dagger}\mathcal T^{(\mathcal{J})}_{i_{\mathcal{J}}}((\otimes_{k \in \mathcal{K}}U^*_k)^\top)^{\dagger}\right), 
\] 
and solve 
\begin{align*}
   \min_{\mathcal{G}_{\mathcal{J}},U_j \text{ for } j \in \mathcal{J}} &\det(\mathcal{G}_{\mathcal{J}}^\top \mathcal{G}_{\mathcal{J}}) + \lambda_3 ||U_{\mathcal{J}} - \otimes_{j \in \mathcal{J}}U_j||^2_F \\ 
    \text{such that }& S^* = \mathcal{G}_{\mathcal{J}}U_{\mathcal{J}}^\top \text{ , }U_j^{\top} e = e \text{ and } U_j \geq 0 \text{ for all } j \in \mathcal{J}, 
\end{align*} 
to obtain an optimal solution $U^*_i$ for $i \in \mathcal J$ and $\mathcal{G}^*_{\mathcal{J}}$. 
\end{enumerate}

By a result similar to Theorem \ref{thm:unfoldsorder3alternate}, one can show that, under Assumption~\ref{ass:orderdslices}, for any $\lambda_1,\lambda_2,\lambda_3 > 0$, an nTD obtained with the above two-step procedure is essentially unique.

\subsubsection{Using random linear combinations of slices} 

The idea of Section~\ref{sec:randomidea2}, to take  random linear combinations of slices instead of just one, generalizes in a similar way to order-$d$ tensors. 
Note that an order-$d$ tensor has $\prod_{i=1}^{d-2} n_i$ many $(d-1,d)$-slices. So if we are just given access to the entire tensor, taking a random linear combination of the slices indeed requires looking at all the entries of the tensor which is computationally heavy. We could of course use a tradeoff and select a subset of slices. 
One could also look at other computational models where the underlying assumptions is that a succinct representation of the tensor is given as input and in those settings, one could access the random linear combination of slices at a much lower cost; see~\cite{Koiran2023,saha:tel-04379539} for more detailed discussions.




\section{Kronecker product of two matrices and the SSC} \label{subsec:KronSSCconj}


In quite a few cases, our identifiability results relied on the Kronecker product of two or more factors to satisfy the SSC, namely 
for order-$3$ nTDs using unfoldings in Theorem~\ref{thm:order2unfold}, 
and for higher-order nTDs in 
Theorems~\ref{thm:unfoldsorderd} for unfoldings and Theorem~\ref{thm:idea3order3} for slices. 
In this section, we discuss sufficient conditions for the Kronecker product of two factors to satisfy the SSC.

As discussed in Section \ref{sec:idenorder3tensors}, 
for order-$3$ nTDs, one has similar identifiability results when the factor matrices individually satisfy the SSC, or when the Kronecker product of two of them satisfies the SSC. One can then ask the following question 
\begin{openquestion}[SSCness of Kronecker product of SSC matrices]\label{conj:KronSSCq1}
Let $U_i \in \R^{n_i\times r_i}$
satisfy the SSC for $i =1,2$. 
Does $U_1 \otimes U_2 \in \R^{n_1n_2 \times r_1r_2}$ satisfy the SSC?  
\end{openquestion}

Unfortunately, we will see that the answer is negative as soon as $\min(r_1,r_2) \geq 3$ and $\max(r_1,r_2) \geq 4$ (Theorem~\ref{cor:counterexample_SSC_kron}), otherwise it is positive  (Corollary~\ref{cor:sscssc_for_ri_less_than_3}). 

For this reason, let us introduce a more constrained variant of the SSC.  
\begin{definition}\label{def:SSCexpanded}[Expanded sufficiently scattered condition ($p$-SSC)] 
Let $U \in \R_+^{n \times r}$, $r\ge 2$ and $p \geq 1$. The matrix $U$ satisfies the $p$-SSC if 
 $$
 \mathcal{C}_p = \left\{x \in \R_+^r \ \big| \ e^{\top}x \geq p \|x\|_2 \right\} \; \subseteq \; \cone\left(U^\top\right). 
 $$
\end{definition} 
\begin{figure}[ht!]
    \centering
    \includegraphics[width=.7\linewidth]{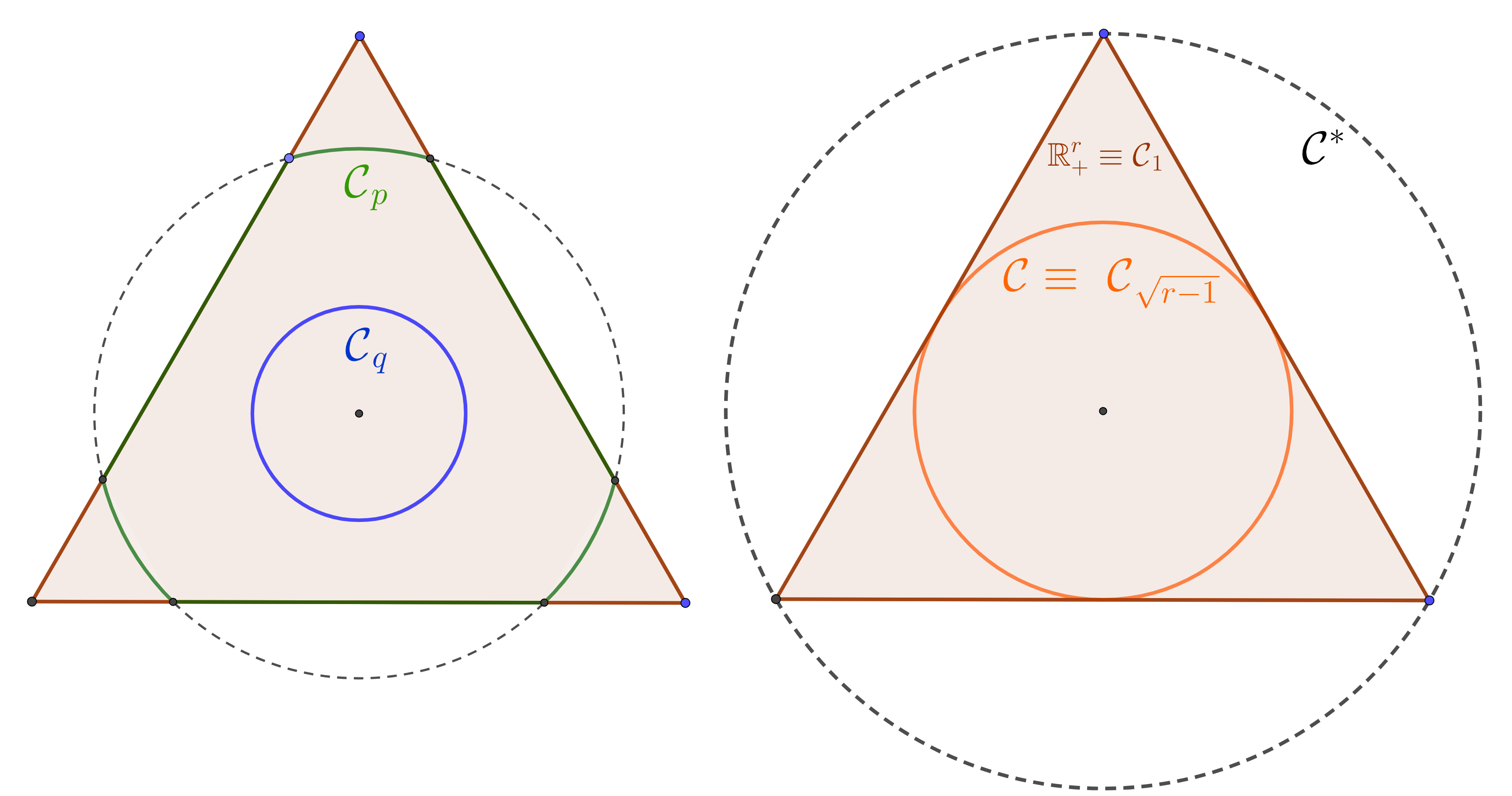}
    \caption{Visualization of $p$-SSC in the case $r=3$ on the plane $\{x\in \f R^r \ | \ e^\top x = 1\}$. 
 On the left, the cones $\mathcal C_p$ with $1< p<\sqrt{r-1}$ and $\mathcal C_q$ with $\sqrt{r-1}< q<\sqrt r$. On the right, the cone $\mathcal C\equiv \mathcal C_{\sqrt {r-1}}$ and the nonnegative orthant $\f R^r_+\equiv \mathcal C_1$.}
    \label{fig:pSSC}
\end{figure}

Intuitively, $p$-SSC of $U$ requires $\cone(U^\top)$ to contain a larger cone than the one required by SSC1 when $p < \sqrt{r-1}$. 
For $p = 1$, $p$-SSC is equivalent to separability because $\mathcal{C}_1 = \R_+^r$. 
For any $p < \sqrt{r-1}$, $p$-SSC implies the SSC, because 
$\mathcal C_q \cu \mathcal C_p \text{ for any } q\ge p$, and 
$\mathcal C = \mathcal C_{\sqrt{r-1}}$, that is, $p$-SSC with $p=\sqrt{r-1}$ is equivalent to SSC1; this is why we need $p$ smaller than $\sqrt{r-1}$ to have SSC2 and hence SSC.  

We can now ask the following question: 
\begin{openquestion}[SSCness of Kronecker product of $p$-SSC matrices]\label{conj:KronSSCq2}
Let $U_i \in \R^{n_i\times r_i}$
satisfy the $p_i$-SSC for $i =1,2$. 
Does $U_1 \otimes U_2 \in \R^{n_1n_2 \times r_1r_2}$ satisfy the SSC?  
\end{openquestion}


Before providing an answer to this question, let us discuss further the $p$-SSC condition. It is linked to the so-called \textit{uniform pixel purity level} $\gamma$ defined in \cite{lin2015identifiability} as follows: 
\[
\gamma := \sup \Big\{ s\le 1 \ \Big| \    B_s \cap \Delta^r \cu \conv\big(U^\top \big) \Big\} \quad \text{ where } \quad B_s= \{x\in \f R^r \ | \ \|x\|\le s \}, 
\]
where $U$ is assumed to be a row-stochastic matrix (w.l.o.g.). 
Since $\cone(B_s\cap \Delta^r) = \mathcal C_{1/s}$, it is immediate to see that  a  row-stochastic matrix $U$ satisfies  $p$-SSC if and only if its uniform pixel purity level is at least $\gamma \ge 1/p$. 

Moreover, 
\begin{itemize}
    \item In \cite{expandedSSC}, it is proved that $U$ satisfies the SSC according to definition \ref{def:SSCalternate} if and only if it satisfies $p$-SSC for some $p^2< r-1$ when $r>2$.
    
    \item  As a consequence, min-vol NMF~\eqref{eq:optminvol} is identifiable if $H$ satisfies $p$-SSC for some $p^2 < r-1$ when $r>2$. 
    
    \item When $r=2$, the separability condition, SSC and $1$-SSC coincide. 
\end{itemize}


\subsection{Kronecker Product of Expanded SSC matrices}

The following theorem provides an answer to Question~\ref{conj:KronSSCq2}. It will  also allow us to answer Question~\ref{conj:KronSSCq1}. 
\begin{theorem}
\label{theo:Kronecker_product_of_pSSC}
Let $U_i \in \R^{n_i\times r_i}$
satisfy the SSC and $p_i$-SSC
for $i = 1,2$. 
Then $U_1 \otimes U_2 \in \R^{n_1n_2 \times r_1r_2}$ satisfies the SSC when
\[
1\le \sqrt{\frac{r_1-p_1^2}{p_1^2(r_1-1)}} + \sqrt{\frac{r_2-p_2^2}{p_2^2(r_2-1)}}.
\]
\end{theorem}
\begin{proof}
First of all, notice that since $U_i$ satisfies the SSC, $p_i^2\le r_i-1$, and if $r_i>2$ then $p_i^2 < r_i-1$. 

The condition SSC1 on  $U_1\otimes U_2$  is equivalent to the condition
\[
(U_1\otimes U_2) v\ge 0 \implies e^\top v\ge \|v\|, 
\]
and by recasting $v$ to a matrix $V\in\f R^{r_1\times r_2}$, that is, $V = \mat(v)$, it can be written as
\[
U_1 VU_2^\top  \ge 0 \implies e^\top V e \ge \|V\|_F.
\]
Let us suppose from now on that $U_1 VU_2^\top \ge 0$. 
Since $U_1$ satisfies $p_1$-SSC with $p_1^2 \leq r_1-1$, 
$\mathcal C \cu  \mathcal C_{p_1} \cu \cone(U_1^\top)$ and $\cone^*(U_1^\top)\cu \mathcal C^* = \mathcal C_1$.
Since  $U_2$ satisfies $p_2$-SSC,  
\begin{align}\label{eq:U1U2SSC}
	\nonumber U_1 VU_2^\top \ge 0 \implies&
    VU_2^\top x\in \cone^*(U_1^\top) \ \ \forall x\ge 0
    \implies 
     Vu_2\in \cone^*(U_1^\top)  \ \ \forall u_2\in \cone(U_2^\top) \\
     \implies &
    e^\top Vu_2\ge \|Vu_2\| \ \ \forall u_2\in \cone(U_2^\top) \implies 
    e^\top Vc\ge \|Vc\| \ \ \forall c\in \mathcal C_{p_2}. 
\end{align}
Take now $\alpha_i\ge 0$ as the largest constant such that $\alpha_i e + e_1\in \mathcal C_{p_i}$. 
Since   $\alpha_i e + e_1\ge 0$, we have $e^\top (\alpha_i e + e_1) = p_i \| \alpha_i e + e_1\|$, and by squaring each term,
\begin{align}\label{eq:alpha_r_+1}
  &  (r_i\alpha_i + 1)^2 = p_i^2(\alpha_i^2r_i + 2\alpha_i + 1) = 
     \frac{p_i^2}{r_i}(\alpha_ir_i + 1)^2 + p_i^2 -  \frac{p_i^2}{r_i}
     \implies  (\alpha_ir_i + 1)^2  
     = p_i^2\frac{r_i-1}{r_i-p_i^2}.
\end{align}
Moreover, notice that for any $e_j$,  $e^\top ( e - e_j)  = r_i-1 \ge   p_i \sqrt{r_i-1} = p_i \|  e - e_j\|$ since by hypothesis $p_i\le \sqrt{r_i-1}$. As a consequence, $e-e_j\in \mathcal C_{p_i}$ and $\alpha_ie+e_j\in \mathcal C_{p_i}$. From \eqref{eq:U1U2SSC}, we find that 
\begin{equation}\label{eq:extremal_points_Cp}
U_1 VU_2^\top \ge 0 \implies     e^\top V(\alpha_2e+e_j)\ge \|V(\alpha_2e+e_j)\|, \quad e^\top V(e-e_j)\ge \|V(e-e_j)\| \quad \forall  j, 
\end{equation}
and thus
\begin{align*}
    (\alpha_2^2+\alpha_2)\|Ve\|^2 + (\alpha_2+1)\left\|Ve_j\right\|^2 &=  
    \left\|V(\alpha_2e+e_j)\right\|^2 + \alpha_2\left\|V(e-e_j)\right\|^2\\
    &\le \left( e^\top V(\alpha_2e+e_j)\right)^2
    + \alpha_2
    \left( e^\top V(e-e_j)\right)^2\\
    & = (\alpha_2^2+\alpha_2)\left( e^\top Ve\right) ^2  + (\alpha_2+1)\left(e^\top Ve_j\right)^2  \quad \forall j, 
\end{align*}
or, written in a simpler way,  
\begin{equation*}
   \alpha_2\|Ve\|^2 + \left\|Ve_j\right\|^2 \le  \alpha_2\left( e^\top Ve\right) ^2  + \left(e^\top Ve_j\right)^2\quad \forall j.
\end{equation*}
By summing the above relations over all indices $j$, we get a bound on the Frobenius norm of $V$: 
\begin{align}\label{eq:bound_on_frob_1}
\| V\|_F^2 & = \sum_j \|Ve_j\|^2
\le  \alpha_2r_2\left( e^\top Ve\right) ^2  + \|e^\top V\|^2 - \alpha_2r_2\|Ve\|^2. 
\end{align}
By applying the same reasoning by swapping $U_1$ and $U_2$, we obtain the symmetric relation
\begin{align}\label{eq:bound_on_frob_2}
\| V\|_F^2 &
\le  \alpha_1r_1\left( e^\top Ve\right) ^2  + \| Ve\|^2 - \alpha_1r_1\|e^\top V\|^2.
\end{align}
Since $e\in \mathcal C_{p_2}$, from \eqref{eq:U1U2SSC} we get that $e^\top Ve\ge \|Ve\|$ and by exchanging $U_1,U_2$ we also get the symmetric relation $e^\top Ve\ge \|e^\top V\|$. We can thus rewrite the conditions \eqref{eq:bound_on_frob_1} and \eqref{eq:bound_on_frob_2} with a change of variables as
\[
\begin{cases}
x = \| V\|_F^2/ \left( e^\top Ve\right) ^2 \\
y = \| Ve\|^2/ \left( e^\top Ve\right) ^2 \\
z = \| e^\top V\|^2/ \left( e^\top Ve\right) ^2 
\end{cases}
\implies 
\begin{cases}
x 
\le  \alpha_2r_2  + z - \alpha_2r_2y = f_2(y,z) \\
x
\le  \alpha_1r_1  + y - \alpha_1r_1z = f_1(y,z)\\
x,y,z \ge 0, \quad y,z\le 1, 
\end{cases}
\]
and, in particular,
\begin{equation}
    \label{eq:x_max_fi}
x\le \max_{0\le y,z\le 1}\min\{f_2(y,z) , f_1(y,z)\}.
\end{equation}
Notice that if $f_2(y,z) < f_1(y,z)$, then to raise the value of the minimum, one can increment $z$ to $\ol z$ until either $f_2(y,\ol z) = f_1(y,\ol z)$ or $\ol z = 1$, and in the second case one can decrease $y$ to $\ol y$ until again $f_2(\ol y,1) = f_1(\ol y,1)$ or $\ol y = 0$. Since $f_1(0,1) = 0 < 1+\alpha_2r_2 = f_2(0,1)$, we conclude that, by this procedure, we will always reach a point $(\ol y,\ol z)$ such that $f_2(\ol y,\ol z) = f_1(\ol y,\ol z)$.
The same reasoning can be repeated when $f_2(y,z) > f_1(y,z)$ to conclude that we have to check only $f_1(y,z) = f_2(y,z)$. Since
\begin{align*}
    f_1(y,z) = f_2(y,z) & \iff (1+\alpha_1r_1)z = (1+\alpha_2r_2)y + \alpha_1r_1 - \alpha_2r_2\\
    \implies f_1(y,z) = f_2(y,z) & = 
 1 + \frac{\alpha_1\alpha_2r_1r_2-1}{1+\alpha_1r_1}(1-y) = 
1 + \frac{\alpha_1\alpha_2r_1r_2-1}{1+\alpha_2r_2}(1-z), 
\end{align*}
we have 
\begin{align*}
  \max_{0\le y,z\le 1}\min\{f_2(y,z) , f_1(y,z)\} &=  \max_{\scriptsize\begin{array}{c}
      0\le y,z\le 1  \\
      f_1(y,z) = f_2(y,z)
  \end{array}} f_1(y,z) \le  \max_{ 0\le y\le 1  } 1 + \frac{\alpha_1\alpha_2r_1r_2-1}{1+\alpha_1r_1}(1-y) \\& = 
  \max\left\{ 1,  1 + \frac{\alpha_1\alpha_2r_1r_2-1}{1+\alpha_1r_1}  \right\}
 = 
  \max\left\{ 1,  \alpha_1r_1\frac{\alpha_2r_2+1}{\alpha_1r_1+1}  \right\}.
\end{align*}
Putting together \eqref{eq:x_max_fi}, \eqref{eq:alpha_r_+1} and the hypothesis, we get
\begin{align*}
    x &\le \max_{0\le y,z\le 1}\min\{f_2(y,z) , f_1(y,z)\}
    \le \max\left\{ 1,  \alpha_1r_1\frac{\alpha_2r_2+1}{\alpha_1r_1+1}  \right\} \\
    &= \max\left\{ 1,  1 + \frac{\alpha_1r_1(\alpha_2r_2+1) -(\alpha_1r_1+1) }{\alpha_1r_1+1}  \right\} \\
    &= \max\left\{ 1,  1 + (\alpha_2r_2+1)[
    1 - (\alpha_2r_2+1)^{-1} - (\alpha_1r_1+1)^{-1}
    ]\right\} \\
    &=
    \max\left\{ 1,  
    1 + \sqrt{p_2^2\frac{r_2-1}{r_2-p_2^2}}\left[
    1 - \sqrt{\frac{r_2-p_2^2}{p_2^2(r_2-1)}} - \sqrt{\frac{r_1-p_1^2}{p_1^2(r_1-1)}}
    \right]\right\} = 1.
\end{align*}
Recalling now the variable change for $x$, we conclude that  $ \|V \|_F^2  \le (e^\top Ve)^2$. By \eqref{eq:U1U2SSC},  we have that  $e^\top Ve\ge \|Ve\|\ge 0$, so finally $e^\top Ve\ge \|V\|_F$, that is,   $U_1\otimes U_2$ satisfies SSC1.

To prove that $U_1\otimes U_2$ satisfies the SSC2, let $v\in \cone^*(U_1^\top \otimes U_2^\top)\cap \textbf{bd}(\mathcal C^*)$ be a nonzero vector. Recasting $v$ as the matrix $V\in \f R^{r_1\times r_2} = \mat(v)$, 
we can rewrite the condition as $U_1VU_2^\top \ge 0$ and $e^\top Ve = \|V\|_F$. Since we already proved that $e^\top Ve\ge \|V\|_F$, we can retrace all the formulas of the proof for SSC1 and substitute  the inequality signs with equality ones. In particular, \eqref{eq:extremal_points_Cp} becomes 
\begin{equation}\label{eq:q_i_are_extremal2}
    \left\|V\left(e-e_j\right)\right\|  =  e^\top V(e-e_j) \quad \text{ for all } j. 
\end{equation}
Let now us analyze the concave function $g(z) := e^\top V z - \|Vz\|$. We know that $g(e-e_j) = 0$ for all $j$ and that $g(u_2)\ge 0$ for all $u_2\in \cone(U_2^\top)$ due to \eqref{eq:U1U2SSC}. 

Since $\cone(U_2^\top)$ is an SSC  nonnegative cone, then $\mathcal C\cu \cone(U_2^\top)\cu \f R^{r_2}_+$. 
The point $e-e_j$ belongs to the boundary of all of the three cones, and in particular it lies on the relative interior part of the  facet $F_j:= \{x\in \f R_+^{r_2}\ | \ x_j = 0\}$ of  $\f R^{r_2}_+$. 
Since $\cone(U_2^\top)$ is polyhedral, and $F_j$ is tangent to $\mathcal C$ in the point $e-e_j$, then necessarily $\cone(U_2^\top)$ must possess a facet $G_j\cu F_j$ such that $e-e_j$ lies on the interior part of $G_j$. 
Therefore, $g(z)$ is a concave nonnegative function on $G_j$ that has a minimum in the interior point $e-e_j$, so $g(z)|_{G_j}\equiv 0$.  
It means that $(e^\top Vz)^2 = \|Vz\|^2  $  or equivalently $ z^\top V^\top ( ee^\top  - I) V z= 0$ for all $z\in G_j$. 
This is a quadratic function, that is in particular analytic, taking value zero in an open set $G_j\cu \{x_j=0\} $, thus implying that  $ z^\top V^\top ( ee^\top  - I) V z= 0$ for all $z$ such that $z_j=0$ by analytic continuation. 
Moreover, for a quadratic function to be zero on $r_2$ distinct hyperplanes, necessarily 
$ z^\top V^\top ( ee^\top  - I) V z\equiv 0$ implies $V^\top V = V^\top ee^\top V$ or $r_2=2$. Doing the same reasoning by inverting $U_1$ and $U_2$ brings us to the conclusion that 
$ y^\top V ( ee^\top  - I) V^\top y= 0$ for all $y$ such that $y_j=0$,  
and, in particular, 
either $V V^\top = V ee^\top V^\top$ or $r_1=2$. 

If  $r_1=r_2=2$, then both $U_1$ and $U_2$ are separable and thus $U_1VU_2^\top \ge 0\implies V\ge 0$ and $e^\top Ve = \|V\|_F\implies V = \lambda e_ie_j^\top$ for some indices $i,j$ and $\lambda \ge 0$. 

Otherwise the matrix $V$ must have rank-$1$ and can be written as $V= pq^\top$, and, since $V\ne 0$, both $p,q$ are nonzero. 
Since   $ z^\top V^\top ( ee^\top  - I) V z= 0$ for all $z$ that have at least a zero entry, if $q_j\ne 0$ and $z = e_j$, 
we get 
\[
 (z^\top q)^2 ((e^\top p)^2 - \|p\|^2) = 0 \implies |e^\top p|  =  \|p\|, 
 \]
and, if needed by changing the sign of both $p,q$, we can conclude that $e^\top p  = \|p\|$ and $p\in \textbf{bd}(\mathcal C^*)$. 
Since $e \in \cone(U_1^\top)$, then by \eqref{eq:U1U2SSC} for any $u_2\in \cone(U_2^\top)$,  we have
 \[
e^\top Vu_2 \ge  \|Vu_2\| \implies e^\top p  \cdot q^\top u_2\ge  \|p\| \cdot | q^\top u_2|  \implies q^\top u_2  \ge | q^\top u_2| \ge 0, 
\] 
 and therefore $q\in \cone^*(U_2^\top)$, and, in particular, $q^\top e\ge 0$. 
Notice that from the symmetric reasoning, we still have that  $ y^\top V ( ee^\top  - I) V^\top y= 0$ for all $y$ that have at least a zero entry, so $|e^\top q| = \|q\|$, and from above we conclude that $e^\top q = \|q\|$ and $q\in \textbf{bd}(\mathcal C^*)\cap \cone^*(U_2^\top)$.  Analogously, we also find that $p\in \textbf{bd}(\mathcal C^*)\cap \cone^*(U_1^\top)$.

  Finally, since $U_1$ and $U_2$ satisfy the SSC2, then both $p$ and $q$ are positive multiples of canonical basis vectors. The matrix $V=pq^\top$ is thus a zero matrix with the exception of a single positive entry, and the original vector $v$ will thus also be a positive multiple of some $e_i$, proving that $U_1 \otimes U_2$ satisfies the SSC2. 
\end{proof}

Theorem \ref{theo:Kronecker_product_of_pSSC} is important because it gives precise bounds on the expansion factors $p_1,p_2$ for the SSC conditions on $U_1$ and $U_2$ to achieve the SSC condition on their product $U_1 \otimes U_2$. An immediate corollary is that if one of the two is separable and the other satisfies the SSC, then the product also satisfies the SSC.

\begin{corollary}\label{cor:sscseparable}
Let $U_1 \in \R^{n_1\times r_1}$ satisfy the SSC and let $U_2 \in \R^{n_2\times r_2}$ satisfy $p_2$-SSC for any 
   \[
p_2^2 \le   \frac{r_2}{1 + (r_2-1)f(r_1)}, \quad f(r_1) = \left(1-\frac{1}{r_1-1}\right)^2.
\]
Then $U_1 \otimes U_2 \in \R^{n_1 n_2 \times r_1 r_2}$ satisfies the SSC. \\
In particular, if $U_1 $ satisfies the SSC, and $U_2 $ is separable, then $U_1 \otimes U_2 $ satisfies the SSC. 

\end{corollary}
\begin{proof}
   Since $U_1$ satisfies the SSC, then it is $p_1$-SSC for some $p_1^2\le r_1-1$.  
   If we now analyze the condition on $p_1$ and $p_2$ given by Theorem \ref{theo:Kronecker_product_of_pSSC}, we find that
   \begin{align*}
        \sqrt{\frac{r_1-p_1^2}{p_1^2(r_1-1)}} + \sqrt{\frac{r_2-p_2^2}{p_2^2(r_2-1)}} &\ge 
        \frac {1}{r_1 -1} + \sqrt{\frac{r_2[1 + (r_2-1)f(r_1)] - r_2}{r_2(r_2-1)}}
        =
        \frac {1}{r_1 -1} + \sqrt{f(r_1)} = 1
   \end{align*}
   and thus $U_1 \otimes U_2 $ satisfies the SSC. Moreover, since $0\le f(r_1)< 1$, then the upper bound on $p_2$ in the hypothesis is at least $1$. In particular, if $p_2=1$ then $U_2$ is separable and  $U_1 \otimes U_2 $ satisfies the SSC. 
\end{proof}

We can now turn to Question \ref{conj:KronSSCq1} and give a positive answer when the dimension $r_i$ are low enough. 

\begin{corollary}
\label{cor:sscssc_for_ri_less_than_3}Let $(r_1,r_2)$ with $r_1=2$ or $r_2=2$ or $r_1=r_2=3$.   For any $U_1\in \f R^{n_1\times r_1}$ and $U_2\in \f R^{n_2\times r_2}$ satisfying the SSC, their Kronecker product $U_1\otimes U_2$ satisfies the SSC. 
\end{corollary}
\begin{proof}
    Let $r_1=2$. Since $U_1$ satisfies SSC, then $U_1$ is also separable, and the results follows from Corollary \ref{cor:sscseparable}. The same argument can be repeated for $r_2=2$. 

    The only case left is $r_1=r_2=3$. Since $U_1,U_2$ satisfy SSC, then $p_1^2,p_2^2\le 2$ and the condition in Theorem \ref{theo:Kronecker_product_of_pSSC} reads as
    \begin{align*}
        \sqrt{\frac{r_1-p_1^2}{p_1^2(r_1-1)}} + \sqrt{\frac{r_2-p_2^2}{p_2^2(r_2-1)}} & = 
        \sqrt{\frac{3-p_1^2}{2p_1^2}} + \sqrt{\frac{3-p_2^2}{2p_2^2}}
        \ge 
          \sqrt{\frac{1}{4}} + \sqrt{\frac{1}{4}} = 1,
   \end{align*}
   thus proving that $U_1\otimes U_2$ satisfies the SSC. 
\end{proof}

Sadly, the same result does not hold for all other choices of dimensions $(r_1,r_2)$, as shown in the next section. 

{
}
 
\subsection{Negative answer to Question~\ref{conj:KronSSCq1} for $r_1,r_2 \geq 3$ and $\max\{r_1,r_2\}\geq 4$}	 \label{sec:negative_answer_sscKronconj}

Notice that the definitions for SSC and $p$-SSC only talk about cones, thus we can always normalize all the nonzero rows of $U$ into stochastic vectors. 
For this reason, from now on we focus only on row stochastic matrices. 

The condition SSC1 on  $U_1\otimes U_2$  is equivalent to the condition
\[
(U_1\otimes U_2) v\ge 0 \implies e^\top v\ge \|v\|, 
\]
and by recasting $v$ to a matrix $V\in\f R^{r_1\times r_2}$, that is, $V = \mat(v)$, 
it can be written as
\[
U_1 VU_2^\top  \ge 0 \implies e^\top V e \ge \|V\|_F.
\]
As a consequence, to have a negative answer to Question~\ref{conj:KronSSCq1}, it is enough to look for matrices  $V\in\f R^{r_1\times r_2}$ for which $U_1 VU_2^\top  \ge 0$  and $e^\top V e < \|V\|_F$ when both $U_1$ and $U_2$ satisfy the SSC. 
\begin{lemma}
\label{theo:negative_answer_to_conj} Let $U_i\in \f R_+^{n_i\times r_i}$ where $\cone(U_i^\top)\cu \mathcal C_{p_i}$ for $i=1,2$, $r_1\le r_2$, and
\begin{equation} \label{eq:c1c2v1}
\left(\frac{1}{p_1^2} - \frac 1{r_1}\right)\left(\frac{1}{p_2^2} - \frac 1{r_2}\right) < \frac 1{r_2r_1} \frac {r_1-1}{r_1r_2-1},
\end{equation}
then the product $U_1\otimes U_2$ does not satisfy SSC1. 
\end{lemma} 
\begin{proof}

Call $c_i = \frac{1}{p_i^2} - \frac 1{r_i}$ and notice that we can always consider $U_i$ row stochastic. If $u_i$ is any row of $U_i$, then
\begin{align*}
    \cone(U_i^\top)\cu \mathcal C_{p_i} &\implies 1 = (e^\top u_i)^2\ge p_i^2\|u_i\|^2 = p_i^2(\|u_i-e/r_i\|^2 + \|e/r_i\|^2)\\
    &\implies \|u_i-e/r_i\|^2\le \frac{1}{p_1^2} - \frac 1{r_i} = c_i.
\end{align*}
The hypotheses~\eqref{eq:c1c2v1} can thus be rewritten as
\begin{equation} \label{eq:c1c2}
\max_{i\in [n_1]}\|U_1(i,:)^\top-e/{r_1}\|^2\le c_1,\qquad  
\max_{j\in [n_2]}\|U_2(j,:)^\top-e/{r_2}\|^2\le c_2, 
\qquad 
c_1c_2 < \frac 1{r_2r_1} \frac {r_1-1}{r_1r_2-1}.
\end{equation}

Let now $V = \lambda E/\sqrt{r_1r_2} + B\in \f R^{r_1\times r_2}$ where $E$ is the matrix of all ones, $\lambda\ge \sqrt {c_1c_2r_1r_2}$ and $B$ is any real matrix satisfying
\[
Be = 0,\quad e^\top B = 0, \quad \|B\|= 1, \quad \|B\|_F^2=r_1-1.
\]
This can be realized as $B = Q_1Q_2^\top$, where $Q_1\in \f R^{r_1\times (r_1-1)}$, $Q_2\in \f R^{r_2\times (r_1-1)}$ are matrices with orthonormal columns such that $e^\top Q_1 =0$, $e^\top Q_2 =0$. We can thus compute
\begin{align*}
    U_1VU_2^\top &= \frac \lambda{\sqrt{r_1r_2}} U_1EU_2^\top + U_1BU_2^\top
    =  \frac \lambda{\sqrt{r_1r_2}} E  + \left(U_1-\frac E {r_1}\right)B\left(U_2^\top-\frac E {r_2}\right)
     \\ &\ge  \frac \lambda{\sqrt{r_1r_2}} E  - 
     \left( \max_{i\in [n_1]}\|U_1(i,:)^\top-e/{r_1}\|\max_{j\in [n_2]}\|U_2(j,:)^\top-e/{r_2}\|\right)\|B\| E
    \\ &= \left(\frac \lambda{\sqrt{r_1r_2}} - \sqrt{c_1c_2}\right)E\ge 0,\\
    (e^\top V e)^2 &= \lambda^2 r_1r_2,\\
    \|V\|_F^2 & = \left\| \frac \lambda{\sqrt{r_1r_2}} E \right\|_F^2 + \|B\|_F^2 = \lambda^2 + r_1-1.
\end{align*}
Hence $ U_1VU_2^\top\ge 0$, and $e^\top V e <\|V\|_F$ holds if and only if  $ \lambda^2  < \frac {r_1-1}{r_1r_2-1}$. 
Since $\lambda\ge \sqrt {c_1c_2r_1r_2}$, then for 
\begin{equation*}
    c_1c_2 < \frac 1{r_1r_2} \frac {r_1-1}{r_1r_2-1}
\end{equation*}
we can find a $\lambda$ such that $ U_1VU_2^\top\ge 0$ and $e^\top V e <\|V\|_F$, that is,  the matrix $U_1\otimes U_2$ does not satisfy  SSC1. 
\end{proof}

Lemma~\ref{theo:negative_answer_to_conj} can be used to prove the following result. 
\begin{theorem}
  \label{cor:counterexample_SSC_kron} For $\min\{r_1, r_2\} \geq 3$ and $\max\{r_1,r_2\}\ge 4$, there exist $U_1$ and $U_2$ satisfying SSC while  
   $U_1\otimes U_2$ does not satisfy SSC1.  
\end{theorem} 
\begin{proof}Suppose w.l.o.g.\ that $r_1\le r_2$. 
 Let us look for $U_1, U_2, c_1, c_2$   satisfying \eqref{eq:c1c2} with $U_1, U_2$ SSC, so that 
   $U_1\otimes U_2$  automatically does not satisfy  SSC1 by Lemma~\ref{theo:negative_answer_to_conj}. 

If $U_1$ satisfies SSC1, then there exists a column $u_1$ of $U_1^\top$ that is also an extreme ray for $\cone(U_1^\top)$ that does not lie in $\mathcal C$,  and thus 
\begin{equation*}\label{eq:bound_c1_distance_from_center}
    c_1\ge \left\| u_1 - \frac e{r_1}\right\|^2 = \|u_1\|^2 - \frac 1{r_1} > \frac 1{r_1-1}- \frac 1{r_1} = \frac 1{r_1(r_1-1)}.
\end{equation*}
Analogously,   there exists a column $u_2$ of $U_2^\top$ that is a extreme ray of $\cone(U_2^\top)$ and respects $c_2\ge \|u_2-e/{r_2}\|^2> 1/(r_2(r_2-1))$.  
Therefore a necessary condition for the existence of $c_1, c_2$ satisfying \eqref{eq:c1c2} is 
\begin{equation}\label{eq:r1r2_counterexample_condition}
 \frac 1{r_2(r_2-1)}\frac 1{r_1(r_1-1)}
< c_1c_2 <\frac 1{r_2r_1} \frac {r_1-1}{r_1r_2-1}
\iff 
r_1r_2-1
< (r_1-1)^2(r_2-1).
\end{equation}
It is easy to see that for $r_1\ge 3$ and $r_2\ge 4$,  the last relation of  \eqref{eq:r1r2_counterexample_condition} is satisfied. If we impose 
$c_ir_i(r_i-1)=\lambda$ for $i=1,2$, then $c_1,c_2$ satisfy  \eqref{eq:r1r2_counterexample_condition} whenever
\[
1
< \lambda^2 <
\frac{\frac 1{r_2r_1} \frac {r_1-1}{r_1r_2-1}}
{ \frac 1{r_2(r_2-1)}\frac 1{r_1(r_1-1)}} = 
\frac {(r_2-1)(r_1-1)^2}{r_1r_2-1}, 
\]
which holds for example for  
\[
\lambda = \sqrt[4]{ 
\frac {(r_2-1)(r_1-1)^2}{r_1r_2-1}}.
\] 

To conclude, we need to build $U_i$ satisfying the SSC such that any row $u_i$, and in particular any vertex of $\cone(U_i^\top)$, satisfies $\|u_i-e/{r_i}\|^2 \le  c_i$.  In~\cite{ben2001polyhedral},  it is proved that for any $\epsilon > \delta> 0$ there exists a polyhedral cone $\mathcal G_\epsilon^r$ such that
\[
 \mathcal L^r_\delta \cu \mathcal G_\epsilon^r\cu \mathcal L^r_\epsilon,
\qquad \mathcal L^r_\gamma := \{(x,t)\in \f R^r\ |\ (1+\gamma)t\ge \|x\| \}.
\]
If $Q\in \f R^{r\times r}$ is any orthogonal matrix mapping $e_1$ into $e/\sqrt r$, and $D\in \f R^{r\times r}$ a diagonal matrix with all diagonal elements equal to $\frac 1{\sqrt{r-1}}$ except for the first equal to $1$, then it can be easily proved that $QD\mathcal L^r_{0} = \mathcal C$, and, more generally, that  
\[
y\in QD\mathcal L^r_{\gamma}, \quad e^\top y = 1 
\;  \implies \; \|y-e/{r}\|^2\le \frac{(1+\gamma)^2 }
{r(r-1)}
\;\implies \; 
QD\mathcal L^r_{\gamma}\cap \f R_+^r = \mathcal C_{\sqrt{\frac{r(r-1)}{(1+\gamma)^2+r-1}}}
.
\]
If now we set $\epsilon$ such that $(1+\epsilon)^2 = \lambda$ and $\cone(U_i^\top)= QD\mathcal G_\epsilon^{r_i}\cap \f R^{r_i}_+$, we conclude that 
\[
\mathcal L^{r_i}_\delta \cu \mathcal G_\epsilon^{r_i}\cu \mathcal L^{r_i}_\epsilon
\implies 
\mathcal C\subsetneq\mathcal C_{\sqrt{\frac{{r_i}({r_i}-1)}{(1+\delta)^2+{r_i}-1}}} \cu \cone(U_i^\top) \cu 
QD\mathcal L^{r_i}_{\epsilon}. 
\]
In particular, any row $u_i$ of $U_i$ satisfies $\|u_i-e/{r_i}\|^2\le \frac{\lambda }{r_i(r_i-1)} = c_i$, and $U_i$ is SSC since it is $p_i$-SSC with $p_i^2=\frac{{r_i}({r_i}-1)}{(1+\delta)^2+{r_i}-1}<r_i-1$. 

\end{proof}

The last result does not present an explicit construction for $U_1$ and $U_2$ that negates Question~\ref{conj:KronSSCq1}, but it proves its existence.
The proof relies on the construction of \cite[Proposition 3.1]{ben2001polyhedral} that also gives us a lower bound on the number of rows for an SSC matrix $U_i\in \f R^{n_i\times r_i}$. Recall that by Theorem \ref{theo:negative_answer_to_conj}, we need that $\mathcal C = \mathcal C_{\sqrt{r_i-1}}\cu \cone(U_i^\top)\cu \mathcal C_{p_i}$.
If $\sqrt{r_i-1} - p_i$ is quite small then $\mathcal C_{p_i}$ and $\mathcal C_{\sqrt{r_i-1}}$ are very close and every point on the boundary of $\mathcal C_{p_i}$ must be close enough to an extreme ray of $\cone(U_i^\top)$. Once we have an upper bound $\delta$ on this distance, we find that the union of the balls with center a vertex of $\cone(U_i^\top)$ and radius $\delta$ must cover the surface of $\mathcal C_{p_i}$. As a consequence, we can compute a lower bound on the number of the balls (and thus the vertices) needed for the covering. Following the proof of Theorem~\ref{cor:counterexample_SSC_kron}, one obtains
\begin{equation} \label{eq:valuesofn1n2counterex}
n_i = \left(\frac 1\epsilon\right)^{\Omega(r_i)} = \left(\frac 1{\lambda -1}\right)^{\Omega(r_i)}
 = \left(\sqrt[4]{ 
\frac {(r_2-1)(r_1-1)^2}{r_1r_2-1}} -1\right)^{-\Omega(r_i)}.
\end{equation}


Theorem \ref{theo:negative_answer_to_conj} thus ensures that whenever  $ \min\{r_1, r_2\}\ge 3$ and $\max\{r_1,r_2\}\ge 4$,  we can always generate $U_1$ and $U_2$ that satisfy the SSC, 
and whose Kronecker product does not satisfy the SSC1.

 {

}

\subsection{Numerical experiment}

As shown in Theorem~\ref{cor:counterexample_SSC_kron}, for $\min(r_1,r_2) \geq 3$ and $\max(r_1,r_2) \geq 4$, the Kronecker product of two sufficiently scattered matrices with $r_1$ and $r_2$ columns does not necessarily satisfy the SSC. 
However, this is a worse case scenario where $n_1$ and $n_2$ need to be very large in the construction of our counter examples, see~\eqref{eq:valuesofn1n2counterex}, and we have observed experimentally that this Kronecker product often satisfies the SSC. 
In fact, when we started to study this question, we were trying to find counter examples numerically, generating randomly matrices satisfying the SSC, but we were never able to generate a counter example with such a random generation process. 
In particular, we randomly generated $U_1$ and $U_2$ of dimension $20 \times 4$, 
each of which have two non-zeros per row, as done in  \cite[Section 5.1]{Gillis_2024}. 
Using the algorithm to check the SSC criterion \cite[Algorithm~1]{Gillis_2024}, in 13 out of the 20 cases, both $U_1$ and $U_2$ satisfy the SSC. 
In each of these 13 cases, their Kronecker product, $U_1 \otimes U_2 \in \R^{400 \times 16}$, also satisfied the SSC.

\subsection{Future Directions}

Since the Kronecker product of randomly generated  sufficiently scattered  matrices appear to often satisfy the SSC, it would be interesting to explain this behavior theoretically (e.g., using some well-defined generation process for $U_1$ and $U_2$, and quantifying the probability for their Kronecker product to satisfy the SSC). However, this is a hard question because checking the SSC is NP-hard in general~\cite{huang2013non}.  

For tensors of order larger than 3, we might need the product of more than two factors to satisfy the SSC; see Theorems~\ref{thm:unfoldsorderd} and~\ref{thm:orderdslices1}. 
Theorem~\ref{theo:Kronecker_product_of_pSSC} cannot be used in this context: it only guarantees the Kronecker product to be SSC, not $p$-SSC, which would be needed to use Theorem~\ref{theo:Kronecker_product_of_pSSC} recursively. 
Ideally, we would need to show that 
the Kronecker product between two matrices, one satisfying $p_1$-SSC and the other satisfying $p_2$-SSC, satisfy $p$-SSC for some $p$ depending on $p_1$ and $p_2$ (note that, for $p_1=p_2=1$, $p=1$; this is the separable case.)  This is a topic for further research.

\section{Conclusion} 

In this paper, we provided new, stronger, identifiability results for nTDs, $\mathcal T = (U_1,\dots,U_d).\mathcal G$ where $U_i \geq 0$ for $i\in [d]$. 
They relied on two main conditions: 
\begin{enumerate}

    \item The factor matrices, $U_i$ for $i \in [d]$, need to satisfy some sparsity conditions, namely the separability condition (Definition~\ref{def:separablematricesfirst}), the sufficiently scattered condition (the SSC, Definition~\ref{def:SSCalternate}), or their product must satisfy the SSC (see Section~\ref{subsec:KronSSCconj}). 
    
    \item The core tensor, $\mathcal G$, needs to have some slices or unfoldings of maximum rank, and its dimensions need to satisfy some conditions (e.g., it must be square for order-2 nTDs). 
    
\end{enumerate} 
Under such conditions, our identifiability results relied on minimizing some measure of the volume of the core tensor, while combining    
two main ingredients: the identifiability of min-vol NMF  (Theorem~\ref{thm:idenminvol} from~\cite{fu2018identifiability}),  
and of min-vol order-2 nTDs (which is new, see Theorem~\ref{thm:nmfmain}). For order-$3$ tensors, these results are summarized in Table~\ref{tab:summary3}, and they generalize to any order (Section~\ref{app:orderd}).   \\

In Part II of this paper~\cite{sahaPartII2025}, we will discuss how to compute these identifiable solutions, and compare the different variants, both in terms of computational cost and practical effectiveness to recover the hidden factors. We will compare our algorithms to the state of the art on synthetic and real data sets.

\appendix

\section{Identifiability of NTD under separability assumptions} \label{app:separab}

Since separability implies the SSC, our results apply to the more constrained case when the $U_i$'s are separable. In this section, we explicitly state two such results for completeness. 


\subsection{Generalization of Proposition 1 from \cite{agterberg2024estimating}}

One can consider a generalization of Theorem~\ref{thm:agterberg2024}, and give an identifiability result for the following model: 
\begin{equation}\label{eq:agterberggen}
    \mathcal T = (U_1,\dots,U_d).\mathcal{G} 
    \quad \text{ with } U_i \in \mathbb R_+^{n_i \times r_i}  \text{ separable  for } i \in [d], \text{ and } \mathcal{G} \in R^{r_1 \times \cdots \times r_d}. 
\end{equation}  
In the following theorem, we show that the identifiability result from \cite{agterberg2024estimating} can be generalized to any order-$d$ tensor. 

\begin{theorem}[Generalization of Theorem \ref{thm:agterberg2024}]\label{thm:agterberg2024gen}
Let $\mathcal T\in~\R^{n_1 \times \dots \times n_d}$ be an order-$d$ tensor that admits a decomposition of the form (\ref{eq:agterberggen}). Moreover if for all $k \in [d]$, the unfolding of the tensor $\mathcal T$ along the $k$th mode, denoted by $\mathcal T_{(k)}$, has rank $r_k$, then for any other decomposition of the form $\mathcal T = (U^*_1,\dots,U^*_d).\mathcal{G}^*$ where  $U^*_i $ are separable matrices, there exist permutation matrices $\Pi_i$ such that $U_i^* = U_i \Pi_i$ and $\mathcal{G}^* = (\Pi_1^\top,\dots,\Pi_d^\top).\mathcal{G}$. 
\end{theorem}
\begin{proof}
Recall from Property \ref{lem:12flatenningchangeofbasisorderd}, the unfolding of a tensor $\mathcal T$ along the mode $k$, denoted by $\mathcal T_{(k)}$ has a decomposition of the form
\begin{equation}
    \mathcal{T}_{(k)} 
    = 
    \Big( \bigotimes_{j \neq k} U_j \Big) \mathcal{G}_{(k)} U_k^\top.
\end{equation}
Using Theorem \ref{thm:idenseparableNMF}, we get that for any other decomposition of $\mathcal{T}_{(k)}$ of the form $\mathcal{T}_{(k)} = W_k^* (U_k^*)^\top$ such that $U_k^*$ is a separable matrix, there exists a permutation matrix $\Pi_k$ and diagonal matrix $D_k$ such that $W_k^* = (\bigotimes_{j \neq k} U_j) \mathcal{G}_{(k)}D_k \Pi_k$ and $U_k^* = U_k D_k^{-1} \Pi_k$ for all $k \in [d]$. Moreover, since $U_k \in [0,1]^{n_k \times r_k}$, $D_k$ is also the identity matrix of size $r_k$.

Since $U_k^*$ and $U_k^{\#}$ are separable for all $k \in [d]$, they have full column rank. Then $(U_k^*)^{\dagger} U_k^* = (U_k^{\#})^{\dagger} U_k^{\#} = I_{r_k}$ and from this, we can conclude that  $$\mathcal{G}^* = ((U_1^*)^{\dagger},\dots,(U_d^*)^{\dagger}).\mathcal{T} =  (\Pi_1^\top,\dots,\Pi_d^\top).((U_1^{\#})^{\dagger},\dots,(U_d^{\#})^{\dagger}).\mathcal{T} =  (\Pi_1^\top,\dots,\Pi_d^\top).\mathcal{G}^{\#}.$$
\end{proof}

\subsection{Unfoldings along one mode in the separable case}\label{sec:seporderd} 

\begin{assumptions}\label{ass:orderdunfoldsep}
Let $\mathcal{T}$ be an order-$d$ tensor such that it has a TD of the form $\mathcal{T} = (U^{\#}_1,\dots,U^{\#}_d).\mathcal{G}$ where
    \begin{itemize}
        \item $U^{\#}_i \in \R_+^{n_i \times r_i}$ is separable for all $i \in [d]$. 
        
        \item There exists $\mathcal{I} \subset [d]$ such that $\rank(\mathcal{G}_{\mathcal{I}})  = \prod_{i \in \mathcal{I}} r_i = \prod_{j \in \mathcal{I}^{-}} r_j$.
    \end{itemize}
\end{assumptions}

Following the proof of Theorem~\ref{thm:unfoldsorderd}  and using the fact that if $U_1,U_2$ are separable matrices, then $U_1 \otimes U_2$ is separable (this is a direct consequence of Corollary~\ref{cor:sscseparable}), one can show that if the given tensor satisfies 
Assumption~\ref{ass:orderdunfoldsep}, then 
Procedure~$d$.0 returns an \textit{essentially} unique nTD. Note however that using separability allows one to design more effective algorithms (running in polynomial time with provable guarantees, even in the presence of noise); see Part II of this paper~\cite{sahaPartII2025}.

\small 

\bibliographystyle{spmpsci} 
\bibliography{references}

\end{document}